\documentclass[preprint,12pt]{elsarticle}

\usepackage{amssymb}

\usepackage{epic,eepic,euscript,verbatim,afterpage}
\usepackage{amsmath,amsthm}
\usepackage{amsfonts}
\usepackage{diagbox}
\usepackage{overpic}
\usepackage{subfigure,epstopdf}
\usepackage{array}
\numberwithin{equation}{section}
\usepackage{times}
\allowdisplaybreaks[4]
\newtheorem{rem}{Remark}[section]
\newtheorem{thm}{Theorem}[section]
\newtheorem{cor}{Corollary}[section]
\newtheorem{lm}{Lemma}[section]

\newtheorem{prop}{Proposition}[section]

\journal{XXX}

\begin{document}

\begin{frontmatter}



\title{A Uniquely Solvable, Positivity-Preserving and Unconditionally Energy Stable Numerical Scheme for the Functionalized Cahn-Hilliard Equation with Logarithmic Potential}


\author{Wenbin Chen\fnref{label1}}
\ead{wbchen@fudan.edu.cn}
\author{Jianyu Jing\fnref{label2}}
\ead{jyjing20@fudan.edu.cn}
\author{Hao Wu\fnref{label3}}
\ead{haowufd@fudan.edu.cn}
\fntext[label1]{	School of Mathematical Sciences and Shanghai Key Laboratory for Contemporary Applied Mathematics, Fudan University, Shanghai  200433, China}
\fntext[label2]{	School of Mathematical Sciences, Fudan University, Shanghai 200433, China}
\fntext[label3]{	School of Mathematical Sciences and Shanghai Key Laboratory for Contemporary Applied Mathematics, Fudan University, Shanghai  200433, China}

\begin{abstract}
We propose and analyze a first-order finite difference scheme for the functionalized Cahn-Hilliard (FCH) equation with a logarithmic Flory-Huggins potential. The semi-implicit numerical scheme is designed based on a suitable convex-concave decomposition of the FCH free energy. We prove unique solvability of the numerical algorithm and verify its unconditional energy stability without any restriction on the time step size. Thanks to the singular nature of the logarithmic part in the Flory-Huggins potential near the pure states $\pm 1$, we establish the so-called positivity-preserving property for the phase function at a theoretic level. As a consequence, the numerical solutions will never reach the singular values $\pm 1$ in the point-wise sense and the fully discrete scheme is well defined at each time step. Next, we present a detailed optimal rate convergence analysis and derive error estimates in $l^{\infty}(0,T;L_h^2)\cap l^2(0,T;H^3_h)$ under a linear refinement requirement $\Delta t\leq C_1 h$. To achieve the goal, a higher order asymptotic expansion (up to the second order temporal and spatial accuracy) based on the Fourier projection is utilized to control the discrete maximum norm of solutions to the numerical scheme. We show that if the exact solution to the continuous problem is strictly separated from the pure states $\pm 1$, then the numerical solutions can be kept away from $\pm 1$ by a positive distance that is uniform with respect to the size of the time step and the grid. Finally, a few numerical experiments are presented. Convergence test is performed to demonstrate the accuracy and robustness of the proposed numerical scheme. Pearling bifurcation, meandering instability and spinodal decomposition are observed in the numerical simulations.
\end{abstract}




\end{frontmatter}


\section{Introduction}
\label{sec:introduction}
\setcounter{equation}{0}

We consider the following functionalized Cahn-Hilliard (FCH) equation:
\begin{equation}
	\begin{cases}
		&\partial_t\phi = \Delta\mu, \\
		&\mu = -\epsilon^2\Delta\omega + f'(\phi)\omega - \epsilon^p\eta\omega, \\
		&\omega = -\epsilon^2\Delta\phi + f(\phi),
	\end{cases}
	\qquad \text{in}\ \Omega\times(0,T).
	\label{phi}
\end{equation}
Here, we assume that $T>0$ is a given final time and  $\Omega=(0,1)^d$ with the spatial dimension $1\leq d\leq 3$. For the sake of simplicity, below we analyze the system \eqref{phi} in a periodic setting on $\Omega$ such that the variables $\phi$, $\mu$, $\omega$ satisfy periodic boundary conditions. With some minor modifications, our results can be extended to the case with homogeneous Neumann boundary conditions, or boundary conditions of mixed periodic-homogeneous Neumann type.

In the system \eqref{phi}, the scalar function $\phi: \Omega \times (0,T)\to \mathbb{R}$ denotes local concentrations of the two components in a binary mixture (e.g., an amphiphilic material in a solvent). Here we define $\phi$ as the difference of volume fractions. In view of its physical interpretation, only the values $\phi \in  [-1, 1]$ are admissible. The pure states correspond to the values $-1$ (the pure solvent) and $1$ (the pure amphiphile), whereas $\phi\in (-1, 1)$ indicates the presence of a mixing state. The scalar function $\mu: \Omega \times (0,T)\to \mathbb{R}$ is referred to as the chemical potential, which is the first variational derivative of the following free energy functional:
\begin{equation}
	\label{FCH energy}	
	E(\phi)=\int_{\Omega} \frac{1}{2}\left(\epsilon^2\Delta\phi-f(\phi)\right)^2 -\epsilon^p\eta\left(\frac{\epsilon^2}{2}|\nabla\phi|^2 +F(\phi)\right)\,\mathrm{d}x,
\end{equation}
where $\epsilon>0$, $\eta>0$ and $p\in\{1,2\}$ are given parameters. The energy \eqref{FCH energy} extends the classical Cahn-Hilliard energy in the literature (see \cite{cahn1958free})

\begin{equation}\label{CH energy}
	E_{\mathrm{CH}}(\phi)=\int_{\Omega} \frac{\epsilon^2}{2}|\nabla\phi|^2+F(\phi)\,\mathrm{d}x,
\end{equation}
where $F$ stands for the (Helmholtz) free energy density of mixing and its derivative is denoted by $f = F'$. Then the scalar function $\omega:\Omega\times (0,T)\to \mathbb{R}$ in \eqref{phi} corresponds to the first variational derivative of $E_{\mathrm{CH}}$, which is the chemical potential in the classical Cahn-Hilliard theory for binary mixtures.

In \eqref{FCH energy}, the positive parameter $\epsilon\ll 1$ characterizes the width of inner structures, e.g., the (diffuse) interfaces between two components. Next, let us pay some more attention on the parameter $\eta\in\mathbb{R}$, which plays an important role in the modelling and analysis. Indeed, the sign of $\eta$ can change fundamentally the structure of the problem as well as the physically motivating examples \cite{DaiP2013}.
When $\eta=0$, the energy $E$ reduces to the well-known phase-field formulation of the Willmore functional (PFW) that approximates the Canham-Helfrich bending energy of surfaces. The corresponding phase-field model has been used to study deformations of elastic vesicles subject to possible volume/surface constraints (see e.g., \cite{du2004, du2005}).
When $\eta<0$, the energy $E$ is related to the so-called Cahn-Hilliard-Willmore (CHW) energy, which was introduced in \cite{torabi2009, wise2007solve, chen2013efficient} to investigate strong anisotropy effects arising in the growth and coarsening of thin films. In this paper, our interest   will focus on the case $\eta>0$ such that $E$ is referred to as the functionalized Cahn-Hilliard (FCH) energy in the literature. It was first proposed in \cite{PhysRevLett.65.1116} to model phase separation of mixtures with an amphiphilic structure. Later on, the extended FCH model was used to describe nanoscale morphology changes in functionalized polymer chains \cite{promislow2009pem}, amphiphilic bilayer interfaces \cite{DaiP2013,dai2015competitive,Gavish2011}, and membranes undergoing pearling bifurcations \cite{Doelman2014,Promislow2015}.
With a positive parameter $\eta$, the FCH energy \eqref{FCH energy} reflects the balance between the square of the variational derivative of $E_{\mathrm{CH}}$ against itself: the first term stabilizes bilayers, pores and micelle structures, while the second term promotes the growth of interfaces and competition between these geometries (see e.g., \cite{jones2013}). Minimizers of the FCH energy are approximate critical points of $E_{\mathrm{CH}}$, which however favor more surface area. This fact dramatically changes the nature of the energy landscape and presents rich morphological structures that are quite different from the phase separation process described by the classical Cahn-Hilliard energy. The FCH energy can naturally incorporate the propensity of the amphiphilic surfactant phase to drive the creation of stable bilayers, or homoclinic interfaces with an intrinsic width \cite{DaiP2013}. In particular, the bilayer structure has the distinct feature that it separate two identical phases by a thin region of a second phase. For recent works on the minimization problem of the FCH energy, bilayer-structures, pearled patterns, nanostructure morphology changes and network bifurcations, we refer to \cite{Promislow2015, jones2013, KP2018,promislow2017existence,promislow2013critical} and the references cited therein. Finally, we note that the exponent $p$ in \eqref{FCH energy} distinguishes some natural limits of the FCH energy, when $p=1$, it represents the strong functionalization, while the choice $p=2$ corresponds to the weak functionalization (see \cite{dai2015competitive,Doelman2014,KP2018} for detailed discussions).

In this study, we work with the following thermodynamically relevant potential of logarithmic type:
\begin{equation}\label{F}
	F(r) = (1+r)\ln(1+r)+(1-r)\ln(1-r)-\frac{\lambda}{2}r^2,\qquad r\in(-1,1),
\end{equation}
where $\lambda\in\mathbb{R}$ is a given parameter. The potential function $F$ was proposed in the original work \cite{cahn1958free} on the phase separation of binary alloys and it also arises in the Flory–Huggins theory for polymer solutions \cite{Doi}.
The first order derivative of $F$ is denoted by $f$ such that
\begin{equation}\label{f}
	f(r) = \ln\frac{1+r}{1-r}-\lambda r,\qquad r\in(-1,1).
\end{equation}
When $\lambda> 2$, $F$ has a double-well structure with two minima $\pm\, r_*\in (-1,1)$, where $r_*$ is the positive root of the equation $f(r)=0$. Besides, we observe that $\lim_{r\to \pm 1}f(r)=\pm \infty$. This singular nature of $f$ near the pure states $\pm 1$ brings great challenges in the study of related phase-field equations, including the classical Cahn-Hilliard equation driven by  $E_{\mathrm{CH}}$. On the other hand, to avoid possible difficulties caused by the singularity of $f$, a widely used approach is to work with an approximate polynomial like
\begin{equation}
	F(r)=\frac14(r^2-1)^2,\quad r\in \mathbb{R}, \label{regF}
\end{equation}
which is sometimes called the ``shallow quenching" approximation of the logarithmic potential \eqref{F} (see e.g., \cite{debussche1995cahn,wu2022}).
Concerning the Cahn-Hilliard equation  $\partial_t\phi=\nabla\cdot[\mathcal{M}\nabla (-\epsilon^2\Delta\phi + f(\phi))]$ with logarithmic potential \eqref{F}, we refer to \cite{debussche1995cahn, abels2007convergence, elliott1996cahn, Li2022, miranville2004robust, giorgini2017cahn, giorgini2018cahn} for extensive theoretic analysis on well-posedness as well as long-time behavior of global solutions (see also the review articles \cite{wu2022,cherfils2011cahn}), while for contributions in the numerical studies, we recall \cite{CJ3W2022, chen2019, copetti1992numerical,  jeong2016practical, jeong2015efficient, Jia2020, li2016unconditionlly, LiTang2021, yang2019linear} and the references therein.

Under suitable boundary conditions (e.g., the periodic or Neumann type), the FCH equation \eqref{phi} can be viewed as an $H^{-1}$ gradient flow of the FCH energy \eqref{FCH energy}, which not only dissipates the energy but also preserves the mass along time evolution. As long as a solution $\phi$ exists on $(0,T)$ and is sufficiently regular, it holds
\begin{equation}
	\frac{\textrm{d}}{\textrm{d}t}E(\phi) = -\int_{\Omega}\left|\nabla\mu\right|^2\textrm{d}x\leq 0,\quad \forall\, t\in(0,T),
\end{equation}
and
\begin{align}
	\int_\Omega \phi(\cdot,t)\, \mathrm{d}x= \int_\Omega \phi(\cdot,0)\, \mathrm{d}x,\quad \forall\, t\in[0,T].
\end{align}
Concerning the theoretic analysis of equation \eqref{phi}, when $\eta>0$ and the periodic boundary conditions are imposed, Dai et al. \cite{DLP2021} proved the existence of global weak solutions in the case with a general regular potential including \eqref{regF} and a degenerate mobility $\mathcal{M}=\mathcal{M}(\phi)$ (see also \cite{DLLP2021}). For the case with the regular potential \eqref{regF} and a constant mobility, existence and uniqueness of global solutions in the Gevrey class were established in \cite{CWWY} for $\eta\in \mathbb{R}$, again in the periodic setting. We  also mention \cite{Miran2014,ZhaoXP2019} in which some theoretic results for the Cahn-Hilliard-Willmore equation (with $\eta<0$) can be found. The case associated with a logarithmic potential \eqref{F} is more difficult. In the recent work \cite{SCHIMPERNA2020},  Schimperna and Wu investigated the initial boundary value problem of \eqref{phi} subject to homogeneous Neumann boundary conditions with $\eta\in \mathbb{R}$. After overcoming difficulties from the singular potential and its interaction with higher-order spatial derivatives, they proved existence, uniqueness and regularity of a global weak solution and characterized its long-time behavior. For a higher-order parabolic equation, in general we do not have maximum principle for its solution. Nevertheless, the singular character of the nonlinearity $f$ can help to keep the weak solutions of \eqref{phi} staying away from the pure states at a point-wise level such that $-1<\phi(x,t)<1$ almost everywhere in $\Omega\times (0,T)$. This is sometimes referred to as a \textit{positivity-preserving property}, since both $1+\phi>0$ and $1-\phi>0$ hold. When the spatial dimension is less than or equal to two, an improved conclusion has been drawn in \cite{SCHIMPERNA2020}: if the initial datum $\phi_0$ is \textit{strictly separated} from $\pm 1$, then there exists a uniform distance $\delta\in (0,1)$ such that $\|\phi(t)\|_{C(\overline{\Omega})}\leq 1-\delta$ for all $t\in [0,T]$. This fact is crucial for obtaining higher order spatial regularity of solutions, since the singular potential $f$ is now confined on a compact subset of $(-1,1)$ and thus is smooth. We note that similar properties have been proven for the classical Cahn-Hilliard equation with a logarithmic potential and for some extended systems, see for instance,  \cite{miranville2004robust, giorgini2017cahn, giorgini2018cahn} and the references therein.

Based on the analytic results obtained in \cite{SCHIMPERNA2020}, we are going to provide a detailed numerical analysis of the FCH equation \eqref{phi} with  logarithmic potential \eqref{F}, subject to periodic boundary conditions. Observe that \eqref{phi} is a sixth order parabolic equation for $\phi$ (see \eqref{E:mu}--\eqref{E:u} below). Those nonlinear terms in the chemical potential $\mu$ such as $\beta(\phi)\beta'(\phi)$, $\beta''(\phi)|\nabla \phi|^2$, $\Delta \beta(\phi)$ with $\beta(\phi)\triangleq f(\phi)+\lambda \phi$ (see \eqref{E:u}) present a highly nonlinear and singular nature of the problem, which turns out to be more involved than the fourth order Cahn-Hilliard equation. This feature yields great challenges in the numerical study. One obvious difficulty is to overcome the step restriction in the time-space discretization such that an explicit numerical scheme will require a rather restrictive Courant-Friedrichs-Lewy (CFL) condition like $\Delta t\leq Ch^6$, with $\Delta t$ and $h$ being the time and spatial step sizes, respectively. On the other hand, a fully implicit scheme like backward Euler may be difficult to prove its unique solvability and only conditionally energy stable. Our goal is to present a numerical scheme that is uniquely solvable, preserving the mass conservation, energy dissipation and the positivity property, under some mild refinement requirement. Besides, we shall apply some efficient numerical solvers to solve the numerical scheme in the long-time scale simulation.

There have been some progresses on the numerical study of the FCH equation \eqref{phi} with $\eta>0$ and a polynomial potential like \eqref{regF}.
%
In \cite{jones2013}, Jones proposed a semi-implicit
numerical scheme and proved its energy stability but the unique solvability was not given. He also made comparisons among semi-implicit, fully-implicit and explicit exponential time differencing (ETD) numerical schemes. A fully implicit scheme with pseudo-spectral approximation in space was introduced in Christlieb et al. \cite{Christlieb2014}, where the authors performed numerical tests to show its accuracy and efficiency but did not prove energy stability and solvability. In \cite{guo2015local}, Guo et al. presented a local discontinuous Galerkin method to overcome the difficulty associated with higher order spatial derivatives. Energy stability was shown for the semi-discrete (time-continuous) scheme therein. Later on, a convex-concave splitting scheme  was proposed in Feng et al. \cite{feng2018}, where the unique solvability and unconditional energy stability were theoretically verified. Besides, the authors performed the convergence analysis and applied an efficient preconditioned steepest descent algorithm to solve their scheme. Recently, theoretical analysis on the energy stability of a stabilized semi-implicit scheme for the FCH equation was carried out in \cite{Zhang2021analysis}. We also refer to \cite{ZhangOu2021} for an unconditionally energy stable second-order numerical scheme based on the scalar auxiliary variable (SAV) approach, to \cite{Zhang2021highly} for a highly accurate, linear and unconditionally energy stable large time-stepping scheme, and to \cite{zhang2020numerical} for numerical comparison between those schemes based on the stable SAV approach and the classical BDF method.

It is worth mentioning that all the previous numerical studies mentioned above are associated with a regular (polynomial) potential. To the best of our knowledge, there is no result on the numerical analysis of the FCH equation \eqref{phi} with the logarithmic potential \eqref{F} in the existing literature. Here we aim to make the first attempt in this direction. The main contributions of this paper are summarized below.

(1) We propose a first-order semi-implicit finite difference numerical scheme (see (\ref{numerical phi})--(\ref{numerical mu})), based on a suitable convex-concave decomposition of the highly nonlinear FCH energy (see Proposition \ref{c-c decomposition}).

One main difficulty comes from the troublesome term $-\int_{\Omega}\beta(\phi)\Delta\phi\, \mathrm{d}x$ in the FCH energy, which is neither convex nor concave (cf. \eqref{FCH energy2}). Using integration by parts and the periodic boundary conditions, we can rewrite it as  $\int_{\Omega}\beta'(\phi)|\nabla\phi|^2 \,\mathrm{d}x$ instead. Thanks to the strict separation property of $\phi$ (see Proposition \ref{prop:sepa}), we are able to show that the new functional (and its discrete analogue) is strictly convex on a suitable closed set (cf. \cite{SCHIMPERNA2020,MR3032968} for conclusions in some weaker settings). This observation enables us to derive the desired convex splitting scheme that treats the convex part of $E$ implicitly and the concave part explicitly. Our argument is different from the case with the regular potential \eqref{regF} studied in \cite{feng2018} such that no auxiliary term is required here to guarantee the necessary  convexity.

(2) We prove that the proposed numerical scheme (\ref{numerical phi})--(\ref{numerical mu}) is uniquely solvable, without any restriction on the size of the time step and the grid (see Theorem \ref{main1}). Moreover, solutions to this fully discrete scheme preserve the mass at each time step (see Proposition \ref{prop:dmass}) and satisfy the positivity-preserving property in the point-wise sense (see \eqref{masssep}).

The positivity-preserving property at the discrete level is crucial, since it guarantees the numerical scheme to be unconditionally well-defined. Concerning the Cahn-Hilliard equation with logarithmic potential \eqref{F}, this issue was first tackled in \cite{copetti1992numerical} where Copetti and Elliott worked with a backward Euler scheme. Under certain constraint on the time step, they proved the unique solvability and positivity-preserving property of  numerical solutions by applying a regularization of the singular potential and then passing to the limit. An alternative approach was recently introduced in Chen et al. \cite{chen2019},  based on the convex-concave decomposition technique (see e.g., \cite{ES93,Eyre}). They proposed a first order semi-implicit convex splitting scheme as well as a second order accurate scheme involving the implicit backward differential formulation (BDF). Then they overcame the time step restriction and theoretically justified the unique solvability and energy stability of those schemes. In particular, the logarithmic part was treated in an implicit way and its convexity as well as singularity helped to prevent the numerical solution from approaching the singular values $\pm 1$. Later in Chen et al. \cite{CJ3W2022}, a similar idea was applied to show the positivity-preserving property of a second order accurate scheme with a modified Crank-Nicolson approximation. Further applications can be found in recent works on the Cahn-Hilliard equation with a Flory-Huggins-de Gennes energy \cite{Dong2019,Dong2020}, the Poisson-Nernst-Planck system \cite{liu2021positivity} and the reaction-diffusion system with detailed balance \cite{LiuJCP2021} etc.

In this study, we successfully generalize the argument in \cite{chen2019} to the sixth order FCH equation \eqref{phi}. Roughly speaking, for any $n\in \mathbb{N}$, we look for a discrete solution $\phi^{n+1}$ to the numerical scheme (\ref{numerical phi})--(\ref{numerical mu}) by solving a minimization problem form certain discrete energy $\mathcal{J}^n$ derived from the convex-concave energy decomposition (see \eqref{J-energy}). If the numerical solution in the previous time step (denoted by $\phi^n$) takes all its grid values in $(-1,1)$ and its discrete average also belongs to $(-1,1)$, we can use the specific form of the logarithmic term $\beta(\phi)$ and in particular, its singular nature near $\pm1$ to show that the solution $\phi^{n+1}$ that is a (global) minimum of $\mathcal{J}^n$, must exist and can only possibly occur at an interior point in the numerical solution variable domain $\mathcal{A}_h$ (see \eqref{Ah}). Besides, uniqueness of the numerical solution is a direct consequence of the strict convexity of the discrete energy function.

(3) We verify the unconditional energy stability of the proposed semi-implicit numerical scheme (\ref{numerical phi})--(\ref{numerical mu}), see Theorem \ref{e-stable}. The proof is based on the convex-concave decomposition of the discrete FCH energy given in Corollary \ref{dc-c decomposition}. Moreover, the unconditional energy stability enables us to derive a uniform in time $H^2_h$-bound for the numerical solutions (see Corollary \ref{es-HH2}), which plays an important role in the subsequent convergence analysis.

(4) We perform an optimal rate convergence analysis for the numerical scheme (\ref{numerical phi})--(\ref{numerical mu}) and establish an error estimate with $O(\Delta t+h^2)$ accuracy in $l^{\infty}(0,T;L_h^2)\cap l^2(0,T;H^3_h)$, under a linear refinement requirement on the time step size such that $\Delta t\leq C_1 h$ (see Theorem \ref{thm1}).

To obtain the expected error estimate under the above linear refinement constraint, a crucial step is to construct a higher order asymptotic expansion of the exact solution up to the second order temporal and spatial accuracy (see \eqref{pertubation expansion}). A similar idea was used in \cite{liu2021positivity} for the optimal rate convergence analysis of the Poisson-Nernst-Planck system. Here in our case, based on the lower order error estimate in the discrete space $l^{\infty}(0,T;H_h^{-1})\cap l^2(0,T;H^2_h)$ (see \eqref{convergence analysis}) we can first derive a rough estimate on the modified numerical error function in $l^{\infty}(0,T;L_h^2)$ (see \eqref{apri-es}), which combined with the inverse inequality further leads to an error estimate in $l^{\infty}(0,T;L_h^\infty)$ (see \eqref{priori assumptionA}), provided that $\Delta t$ and $h$ are suitably small and satisfy the linear constraint $\Delta t\leq C_1 h$. Furthermore, this $L_h^\infty$ estimate allows us to show that all the numerical solutions can be kept away from the pure states $\pm 1$ with a uniform positive distance (see \eqref{strict separation numerical}). Hence, the singular term $\beta$ (and its derivatives) is confined on a compact subset of $(-1,1)$, becoming regular and uniformly bounded. With this crucial observation, we can proceed to derive the required higher order error estimate in $l^{\infty}(0,T;L_h^2)\cap l^2(0,T;H^3_h)$ by using the energy method.

The remaining part of this paper is organized as follows. In Section \ref{sec:numerical scheme} we present some analytic results on the continuous problem and propose the numerical scheme. In Section \ref{sec:positivity}, we provide detailed proofs on unique solvability of the scheme and the positivity-preserving property of numerical solutions. The unconditional energy stability is then established in Section \ref{sec:energy stability}. In Section \ref{sec:convergence}, we perform an optimal rate convergence analysis and derive the error estimates. Numerical experiments are presented in Section \ref{sec:numerical results} to verify the accuracy and  efficiency of our scheme. The mass conservation, energy dissipation and positivity-preserving properties are verified numerically. In the final Section \ref{sec:remarks}, we make some concluding remarks.


\section{The Numerical Scheme}
\label{sec:numerical scheme}
\setcounter{equation}{0}

\subsection{Preliminaries: well-posedness of the continuous problem}
Let $X$ be a (real) Banach or Hilbert space with the  associated norm denoted by $\|\cdot\|_X$. The symbol $X'$ stands for the dual space of $X$ and $\langle \cdot,\cdot\rangle$ denotes the duality product between $X'$ and $X$. Throughout the paper, we consider the domain $\Omega=(0, 1)^d$ with $1\leq d\leq 3$. For $m\in \mathbb{N}$, we set
\begin{align*}
	C^m_{\mathrm{per}}(\overline{\Omega})=\big\{f\in C^m(\mathbb{R}^d)\,|\, f\ \text{is}\ \Omega\text{-periodic}\big\}\quad\text{and}\quad  \mathring{C}^m_{\mathrm{per}}(\overline{\Omega})=\Big\{f\in C^m_{\mathrm{per}}(\overline{\Omega})\,\Big|\, \int_\Omega f \textrm{d}x=0\Big\}.
\end{align*}
Next, we recall some standard notations for Lebesgue and Sobolev spaces in the periodic setting (see, e.g., \cite{CWWY}).
For $p\in [1,+\infty)$, we define
\begin{align*}
	L^p_{\mathrm{per}}(\Omega)=\big\{f\in L^p_{\text{loc}}(\mathbb{R}^d)\,|\, f\ \text{is}\ \Omega\text{-periodic}\big\}\quad \text{and}\quad \mathring{L}^p_{\mathrm{per}}(\Omega)=\Big\{f\in L^p_{\mathrm{per}}(\Omega)\,\Big|\, \int_\Omega f\,\mathrm{d}x=0\Big\}.
\end{align*}
For $m\in \mathbb{Z}^+$ and $p\in [1,+\infty)$, we define
\begin{align*}
	W^{m,p}_{\mathrm{per}}(\Omega)=\{f\in W^{m,p}_{\text{loc}}(\mathbb{R}^d)\,|\, f\ \text{is}\ \Omega\text{-periodic}\},\quad
	\mathring{W}^{m,p}_{\mathrm{per}}(\Omega)=W^{m,p}_{\mathrm{per}}(\Omega)\cap \mathring{L}^p_{\mathrm{per}}(\Omega).
\end{align*}
When $p=2$, we simply denote $W^{m,p}_{\mathrm{per}}(\Omega)$, $\mathring{W}^{m,p}_{\mathrm{per}}(\Omega)$ by $H^{m}_{\mathrm{per}}(\Omega)$, $\mathring{H}^{m}_{\mathrm{per}}(\Omega)$, respectively. The standard scalar product of~$L^2_{\mathrm{per}}(\Omega)$ will be denoted by  $(\cdot,\cdot)_\Omega$. For $m\in \mathbb{Z}^+$, we also define the dual spaces
$$
H^{-m}_{\mathrm{per}}(\Omega)=(H^{m}_{\mathrm{per}}(\Omega))' \quad\text{and}\quad  \mathring{H}^{-m}_{\mathrm{per}}(\Omega)=\{f\in H^{-m}_{\mathrm{per}}(\Omega)\,|\, \langle f, 1\rangle=0\}.
$$

For the sake of convenience, we denote by $B$ the convex part of the logarithmic potential $F$ and by $\beta$ the monotone part of its derivative $f$ (see  \cite{SCHIMPERNA2020}), that is,
\begin{align}
	&B(r) := F(r)+\frac{\lambda}{2}r^2=(1+r)\ln(1+r)+(1-r)\ln(1-r),\quad r\in (-1,1), \label{B}\\
	&\beta(r) :=f(r) + \lambda r= \displaystyle{\ln\frac{1+r}{1-r}},\quad r\in (-1,1). \label{beta}
\end{align}

In \cite{SCHIMPERNA2020}, Schimperna and Wu investigated the system \eqref{phi} in a bounded smooth domain $\Omega \subset\mathbb{R}^d$, $1\leq d\leq 3$, subject to the homogeneous Neumann boundary conditions $\partial_\mathbf{n}\phi=\partial_\mathbf{n} \omega=\partial_\mathbf{n} \mu=0$ on $\partial \Omega$ ($\mathbf{n}$ denotes the unit outer normal vector on the boundary $\partial \Omega$). They proved existence and uniqueness of a global weak solution to the resulting initial boundary value problem and established its  regularity for any positive time, see \cite[Theorems 2.3, 2.4]{SCHIMPERNA2020}. With minor modifications on the proofs therein, we are able to obtain similar results for the FCH equation \eqref{phi} in the periodic setting.
\begin{prop}[Well-posedness]\label{prop:well}
	Let $\Omega=(0,1)^d$ with $1\leq d\leq 3$. Suppose that $F$ is determined by \eqref{F}, $\lambda,\,\eta\in \mathbb{R}$ and $\epsilon>0$ are given constants.
	For any initial datum $\phi_0$ satisfying
	$$
	\phi_0\in H^{2}_{\mathrm{per}}(\Omega), \quad \beta(\phi_0)\in L^{2}_{\mathrm{per}}(\Omega),\quad \int_\Omega \phi_0 \,\mathrm{d}x\in (-1,1),
	$$
	there exists a unique global weak solution $\phi$ to the system \eqref{phi} subject to the periodic boundary conditions such that, for any $T>0$,
	\begin{align*}
		& \phi \in H^1(0,T;H^{-1}_{\mathrm{per}}(\Omega)) \cap L^\infty(0,T;H^{2}_{\mathrm{per}}(\Omega)) \cap L^4(0,T;H^{3}_{\mathrm{per}}(\Omega)),\\
		& \mu \in L^2(0,T;H^{1}_{\mathrm{per}}(\Omega)),\\
		& \omega \in L^\infty(0,T; L^2_{\mathrm{per}}(\Omega))\cap L^4(0,T; H^{1}_{\mathrm{per}}(\Omega)),\\
		& \beta(\phi) \in L^\infty(0,T;L^{2}_{\mathrm{per}}(\Omega)) \cap L^4(0,T;H^{1}_{\mathrm{per}}(\Omega)),\\
		& \beta(\phi)\beta'(\phi) \in L^2(0,T;L^1_{\mathrm{per}}(\Omega)), \quad  \beta''(\phi) |\nabla \phi|^2 \in L^2(0,T;L^1_{\mathrm{per}}(\Omega)),\\
		& -1<\phi(x,t)<1 \quad \text{a.e.~in }\,\Omega\times(0,T),
	\end{align*}
	and the initial condition is satisfied $\phi|_{t=0} = \phi_0(x)$ almost everywhere in $\Omega$. The following weak formulations hold for almost all $t\in (0,T)$:
	\begin{align}
		& \partial_t \phi =\Delta \mu \quad \text{in }\,H^{-1}_{\mathrm{per}}(\Omega), \label{E:mu}\\
		& \mu = \epsilon^4\Delta^2\phi + \beta(\phi)\beta'(\phi) + \epsilon^2\beta''(\phi)|\nabla \phi|^2 -2\epsilon^2\Delta\beta(\phi) -\lambda\phi\beta'(\phi) + \epsilon^2(2\lambda+\epsilon^p\eta)\Delta\phi \notag \\
		&\qquad - (\lambda+\epsilon^p\eta)\beta(\phi) + \lambda(\lambda+\epsilon^p\eta)\phi,\quad \text{in }\,H^{-1}_{\mathrm{per}}(\Omega) + L^1_{\mathrm{per}}(\Omega). \label{E:u}
	\end{align}
	Besides, the weak solution $\phi$ satisfies the mass conservation such that
	$$
	\int_\Omega \phi(x,t)\, \mathrm{d}x= \int_\Omega \phi_0(x) \,\mathrm{d}x,\quad \forall\, t\in [0,T],
	$$
	and for any $t_1,t_2$ satisfying $0\leq t_1 < t_2 \le T$, we have the energy equality:
	\begin{equation}\notag
		E(\phi(t_2)) + \int_{t_1}^{t_2} \| \nabla \mu(s) \|_{L^2(\Omega)}^2 \, \mathrm{d}s = E(\phi(t_1)).
	\end{equation}
\end{prop}

Next, we present the property of strict separation from pure states $\pm 1$ for global weak solutions. This property holds globally (in time) when the spatial dimension is less than or equal to two, while it can only hold locally in the three dimensional case. In analogy to \cite[Proposition 2.8]{SCHIMPERNA2020} and keeping in mind the continuous embedding $H^1(0,T;H^{-1}_{\mathrm{per}}(\Omega)) \cap L^\infty(0,T;H^{2}_{\mathrm{per}}(\Omega))\hookrightarrow C([0,T]; C_{\mathrm{per}}(\overline{\Omega}))$ for $1\leq d\leq 3$ (cf. \cite{simon}), we can prove the following result:
\begin{prop}[Strict separation from pure states]  \label{prop:sepa}
	Let the hypotheses of Proposition~\ref{prop:well} be satisfied.
	
	$(1)$ Assume that $d=1,2$ and $T>0$. For any $\tau \in (0,T)$, there exists a constant $\delta\in (0,1)$ such that the global weak solution $\phi$ obtained in Proposition~\ref{prop:well} satisfies
	\begin{equation}\notag
		\|\phi(t)\|_{C(\overline{\Omega})} \leq 1 - \delta,\quad \forall\, t\in [\tau,\,T].
	\end{equation}
	Moreover, if $\|\phi_0\|_{C(\overline{\Omega})}\leq 1-\delta_0$ for some given $\delta_0\in (0,1)$, then there exists $\delta_1\in (0,\delta_0]$ such that
	\begin{equation}\notag
		\|\phi(t)\|_{C(\overline{\Omega})} \leq 1 - \delta_1,\quad \forall\, t\in [0,\,T].
	\end{equation}
	
	$(2)$ Assume that $d=3$ and $\|\phi_0\|_{C(\overline{\Omega})}\leq 1-\delta_0$ for some given $\delta_0\in (0,1)$. Then there exists some $T_0\in (0,T]$ (depending on $\delta_0$) such that
	\begin{equation}\label{separ2}
		\|\phi(t)\|_{C(\overline{\Omega})} \leq 1 - \frac12\delta_0,\quad \forall\, t\in [0,\,T_0].
	\end{equation}
\end{prop}
\begin{rem}\label{sepreg}
	Proposition \ref{prop:sepa} indicates that if the initial datum $\phi_0$ is strictly separated from $\pm 1$, then in the one or two dimensional case, for an arbitrary but fixed final time $T>0$, the solution $\phi$ is strictly separated from $\pm 1$ on the whole time interval $[0,T]$ as well. On the other hand, in the three dimensional case, the strict separation property holds only in a ``local" sense such that if $\phi_0$ is strictly separated from $\pm 1$, then at least for some short period of time, the corresponding solution $\phi$ can be strictly separated from $\pm 1$. Whether a global in time separation property holds in the three dimensional case remains an open problem. The above separation property is crucial, since it enables us to avoid the singularity of the logarithmic type function $\beta$ (which is indeed smooth on a closed subset of $(-1,1)$) and then obtain arbitrary higher order regularity of solutions to (\ref{phi}), provided that the initial datum is sufficiently smooth.
\end{rem}

\subsection{The finite difference spatial discretization}
We first recall some standard notations for discrete functions and operators, see e.g., \cite{wise2010,wise2009}. For the spatial discretization, we apply the centered finite difference approximation and adopt the periodic setting similar to that in \cite{CJ3W2022,chen2019,liu2021positivity}. Without loss of generality, below we just state the setting in the computational domain $\Omega=(0,1)^3$. It can be easily extended to more general cells like $\Omega=(0,L_x)\times (0,L_y)\times (0,L_z)$ and to the lower dimensional cases $d=1,2$.

Let $N\in\mathbb{N}$ be a given positive integer. We define the uniform spatial grid size as $h = 1/N >0$. The numerical value of a function $f$ at the cell centered mesh point $((i-\frac{1}{2})h,(j-\frac{1}{2})h,(k-\frac{1}{2})h)$ is denoted by $f_{i,j,k}$ such that $f_{i,j,k}=f((i-\frac{1}{2})h,(j-\frac{1}{2})h,(k-\frac{1}{2})h)$. Then we introduce the space of three dimensional discrete $N^3$-periodic functions:
$$
\mathcal{C}_{\mathrm{per}}:=\left\{f:\mathbb{Z}^3\rightarrow\mathbb{R}\,|\,f_{i,j,k}=f_{i+\gamma_1 N,j+\gamma_2 N,k+\gamma_3 N},\ \ \forall\, i,j,k,\gamma_1,\gamma_2,\gamma_3\in\mathbb{Z}\right\}.
$$
The discrete inner product on $\mathcal{C}_{\mathrm{per}}$ is defined as $\left<f,g\right>_{\Omega}:=h^3\sum_{i,j,k=1}^{N}f_{i,j,k}g_{i,j,k}$ for any $f, g\in \mathcal{C}_{\mathrm{per}}$.
For any $f\in\mathcal{C}_{\mathrm{per}}$, its mean value is given by $\overline{f}= |\Omega|^{-1}\left<f,1\right>_{\Omega}$ (here we simply have $|\Omega|=1$). Then we introduce the subspace of $\mathcal{C}_{\mathrm{per}}$ with zero mean such that
$$
\mathring{\mathcal{C}}_{\mathrm{per}}:=\left\{f\in\mathcal{C}_{\mathrm{per}}\,|\,\overline{f}=0\right\}.
$$
Next, we use the notation $\mathcal{E}_{\mathrm{per}}=\mathcal{E}_{\mathrm{per}}^x\times \mathcal{E}_{\mathrm{per}}^y\times \mathcal{E}_{\mathrm{per}}^z$ to denote the set of face-centered vector grid functions $\vec{f} = (f^x, f^y, f^z)^T$ that are $\Omega$-periodic, i.e., $N$-periodic in each direction. Here, we set
$$
\mathcal{E}_{\mathrm{per}}^x:=\left\{f:\mathbb{Z}^3\rightarrow\mathbb{R}\,\Big|\,f_{i+\frac12,j,k}=f_{i+\frac12+\gamma_1 N,j+\gamma_2 N,k+\gamma_3 N},\ \  \forall\, i,j,k,\gamma_1,\gamma_2,\gamma_3\in\mathbb{Z}\right\},
$$
where $f_{i+\frac12,j,k}=f(ih,(j-\frac{1}{2})h,(k-\frac{1}{2})h)$. The other two spaces $\mathcal{E}_{\mathrm{per}}^y$, $\mathcal{E}_{\mathrm{per}}^z$ are defined in an analogous manner.

The discrete average and difference operators acting on $\mathcal{C}_{\mathrm{per}}$ yield face-centered functions, evaluated at the east-west faces, $(i+\frac{1}{2},j,k)$; north-south faces, $(i, j + \frac{1}{2}, k)$; and up-down faces, $(i, j, k + \frac{1}{2})$, respectively. We define $A_x, D_x: \mathcal{C}_{\mathrm{per}} \to \mathcal{E}_{\mathrm{per}}^x$, $A_y, D_y: \mathcal{C}_{\mathrm{per}} \to \mathcal{E}_{\mathrm{per}}^y$ and $A_z, D_z: \mathcal{C}_{\mathrm{per}} \to \mathcal{E}_{\mathrm{per}}^z$ such that
\begin{align*}
	&A_xf_{i+\frac{1}{2},j,k}:=\frac{1}{2}\left(f_{i+1,j,k}+f_{i,j,k}\right),\qquad D_xf_{i+\frac{1}{2},j,k}:=\frac{1}{h}\left(f_{i+1,j,k}-f_{i,j,k}\right),\\
	&A_yf_{i,j+\frac{1}{2},k}:=\frac{1}{2}\left(f_{i,j+1,k}+f_{i,j,k}\right),\qquad D_yf_{i,j+\frac{1}{2},k}:=\frac{1}{h}\left(f_{i,j+1,k}-f_{i,j,k}\right),\\
	&A_zf_{i,j,k+\frac{1}{2}}:=\frac{1}{2}\left(f_{i,j,k+1}+f_{i,j,k}\right),\qquad D_zf_{i,j,k+\frac{1}{2}}:=\frac{1}{h}\left(f_{i,j,k+1}-f_{i,j,k}\right).
\end{align*}
On the other hand, the corresponding operators at the staggered mesh points, that is, $a_x, d_x: \mathcal{E}_{\mathrm{per}}^x\to \mathcal{C}_{\mathrm{per}}$, $a_y, d_y:  \mathcal{E}_{\mathrm{per}}^y \to \mathcal{C}_{\mathrm{per}}$, $a_z, d_z: \mathcal{E}_{\mathrm{per}}^z \to \mathcal{C}_{\mathrm{per}}$, are given by
\begin{align*}
	&a_xf_{i,j,k}:=\frac{1}{2}\left(f_{i+\frac{1}{2},j,k}+f_{i-\frac{1}{2},j,k}\right),\qquad d_xf_{i,j,k}:=\frac{1}{h}\left(f_{i+\frac{1}{2},j,k}-f_{i-\frac{1}{2},j,k}\right),\\
	&a_yf_{i,j,k}:=\frac{1}{2}\left(f_{i,j+\frac{1}{2},k}+f_{i,j-\frac{1}{2},k}\right),\qquad d_yf_{i,j,k}:=\frac{1}{h}\left(f_{i,j+\frac{1}{2},k}-f_{i,j-\frac{1}{2},k}\right),\\
	&a_zf_{i,j,k}:=\frac{1}{2}\left(f_{i,j,k+\frac{1}{2}}+f_{i,j,k-\frac{1}{2}}\right),\qquad d_zf_{i,j,k}:=\frac{1}{h}\left(f_{i,j,k+\frac{1}{2}}-f_{i,j,k-\frac{1}{2}}\right).
\end{align*}

Let us now introduce the discrete gradient and divergence operators $\nabla_h: \mathcal{C}_{\mathrm{per}} \to \mathcal{E}_{\mathrm{per}}$, $\nabla_h\cdot : \mathcal{E}_{\mathrm{per}}\to \mathcal{C}_{\mathrm{per}}$. For any $f\in \mathcal{C}_{\mathrm{per}}$, we set $\nabla_h f_{i,j,k}=(D_x f_{i+\frac12,j,k}, D_y f_{i,j+\frac12,k}, D_z f_{i,j,k+\frac12})^T$, while for any $\vec{f} = (f^x, f^y, f^z)^T\in \mathcal{E}_{\mathrm{per}}$, we define $\nabla_h\cdot \vec{f}_{i,j,k}=d_xf^x_{i,j,k}+d_yf^y_{i,j,k}+d_zf^z_{i,j,k}$.
Suppose that $g$ is an arbitrary scalar function defined at all the face center points (east-west, north-south, and up-down) and $\vec{f} = (f^x, f^y, f^z)^T \in \mathcal{E}_{\mathrm{per}}$, then we have
\begin{equation}
	\nabla_h\cdot(g\vec{f})_{i,j,k}=d_x(g_{\text{ew}}f^x)_{i,j,k}+d_y(g_{\text{ns}}f^y)_{i,j,k}+d_z(g_{\text{ud}}f^z)_{i,j,k}.
	\notag
\end{equation}
If $\vec{f} = \nabla_h\phi$ for certain scalar grid function $\phi \in \mathcal{C}_{\mathrm{per}}$, it follows that
\begin{equation}
	\nabla_h\cdot\left(g\nabla_h\phi\right)_{i,j,k}=d_x(g_{\text{ew}}D_x\phi)_{i,j,k}+d_y(g_{\text{ns}}D_y\phi)_{i,j,k}+d_z(g_{\text{ud}}D_z\phi)_{i,j,k}.
	\notag
\end{equation}
As a consequence, the standard three dimensional discrete Laplacian $\Delta_h : \mathcal{C}_{\mathrm{per}} \rightarrow \mathcal{C}_{\mathrm{per}}$ can be defined as
\begin{equation}
	\Delta_h\phi:=\nabla_h\cdot(1\nabla_h\phi)_{i,j,k}=d_x(D_x\phi)_{i,j,k}+d_y(D_y\phi)_{i,j,k}+d_z(D_z\phi)_{i,j,k}, \quad \forall\,\phi \in \mathcal{C}_{\mathrm{per}}.
	\notag
\end{equation}

For $1\leq p<\infty$, the discrete $l^p$ norms are given by $\|f\|_{L_h^p}:=\left(\left<|f|^p,1\right>_{\Omega}\right)^{\frac{1}{p}}$. When $p=2$, we simply denote $\|f\|_{L_h^2}$ by $\|f\|_2$. In addition, the discrete maximum norm in $L^\infty_h$ is defined as $\|f\|_{\infty} := \max_{1\leq i,j,k\leq N} |f_{i,j,k}|$. For any two vector grid functions $\vec{f} = (f^x, f^y, f^z)^T, \vec{g} = (g^x, g^y, g^z)^T\in\mathcal{E}_{\mathrm{per}}$, we define the discrete inner product as follows
$$
\big[\vec{f},\vec{g}\big]_{\Omega}:=\left[f^x,g^x\right]_x + \left[f^y,g^y\right]_y + \left[f^z,g^z\right]_z,
$$
where
$$
\left[f^x,g^x\right]_x:=\left<a_x(f^xg^x),1\right>_{\Omega},
\quad \left[f^y,g^y\right]_y:=\left<a_y(f^yg^y),1\right>_{\Omega},
\quad \left[f^z,g^z\right]_z:=\left<a_z(f^zg^z),1\right>_{\Omega}.
$$
With the above notations, we can further define
\begin{align}
	&\|\nabla_h f\|_{L_h^p}:=\left(\big[|D_xf|^p,1\big]_x+\big[|D_yf|^p,1\big]_y+\big[|D_zf|^p,1\big]_z\right)^{\frac{1}{p}},\quad 1\leq p<\infty,		\label{gradient p}
\end{align}
and the discrete Sobolev-type norms through
\begin{equation*}
	\|f\|_{W_h^{1,p}}^p:=\|f\|_{L_h^p}^p + \|\nabla_h f\|_{L_h^p}^p,\quad
	\|f\|_{W_h^{2,p}}^p:=\|f\|_{W_h^{1,p}}^p + \|\Delta_h f\|_{L_h^p}^p,\quad \forall\,f\in \mathcal{C}_{\mathrm{per}},
\end{equation*}
and so on.
When $p=2$, we simply denote the norm $\|\cdot\|_{W_h^{m,2}}$ by  $\|\cdot\|_{H_h^m}$, $m\in \mathbb{Z}^+$.

A discrete version of the space $\mathring{H}^{-1}_{\mathrm{per}}(\Omega)$ is also necessary in the subsequent analysis. For any $f\in\mathring{\mathcal{C}}_{\mathrm{per}}$, let $\psi[f]:=(-\Delta_h)^{-1}f\in\mathring{\mathcal{C}}_{\mathrm{per}}$ be the unique periodic solution of the discrete elliptic equation $-\Delta_h\psi=f$ (cf. \cite[Lemma 3.2]{wise2009}). Then the inner product $\left<\cdot,\cdot\right>_{-1,h}$ and the associated $\| \cdot \|_{-1,h}$ norm can be defined by
\begin{align*}
	\left<f,g\right>_{-1,h} &:= \big[\nabla_h \psi[f],\nabla_h \psi[g]\big]_{\Omega}= \left<f,\psi[g]\right>_{\Omega}=\left<\psi[f],g\right>_{\Omega},\quad \forall\, f,g\in \mathring{\mathcal{C}}_{\mathrm{per}},
\end{align*}
and $\|f\|^2_{-1,h}  := \left<f,f\right>_{-1,h}$ for any $f\in\mathring{\mathcal{C}}_{\mathrm{per}}$. The above definition is a direct consequence of the following lemma on the discrete summation-by-parts formulae, cf. \cite[Lemma 3.3]{wise2009}, see also \cite{wise2010,guo2016,wang2011}.

\begin{lm}\label{lem:sumbyparts}
	Let $\mathcal{D}$ be an arbitrary periodic scalar function defined on all of the face center points.
	For any $\psi, \phi \in \mathcal{C}_{\mathrm{per}}$ and any $\vec{f}\in\mathcal{E}_{\mathrm{per}}$, the following summation-by-parts formulae are valid:
	\begin{equation}
		\big<\psi,\nabla_h\cdot\vec{f}\,\big>_{\Omega}=-\big[\nabla_h\psi,\vec{f}\,\big]_{\Omega},\qquad \big<\psi,\nabla_h\cdot(\mathcal{D}\nabla_h\phi)\big>_{\Omega}=-\big[\nabla_h\psi,\mathcal{D}\nabla_h\phi\big]_{\Omega}.\notag
	\end{equation}
\end{lm}

Before ending this subsection, we also recall the following result (see \cite[Lemma 3.1]{chen2019}) that will be useful in the subsequent analysis:
\begin{lm}\label{lemma2}
	Suppose that $\phi_1,\phi_2\in\mathcal{C}_{\mathrm{per}}$, satisfying $\left<\phi_1-\phi_2,1\right>_{\Omega}=0$ and $\|\phi_1\|_{\infty},\,\|\phi_2\|_{\infty}<1$, the following estimate holds:
	\begin{equation}
		\|(-\Delta_h)^{-1}(\phi_1-\phi_2)\|_{\infty}\leq C_0,
	\end{equation}
	where the constant $C_0>0$ only depends on $\Omega$.
\end{lm}

\subsection{A convex-concave energy decomposition}
Before proceeding, let us mention that although the sign of the parameters $\lambda, \eta$ does not play an essential role in the theoretic analysis of the continuous problem \eqref{phi}, it will leads to some differences in the corresponding numerical schemes (see Remark \ref{diffeta} for detailed discussions). In the remaining part of this paper, we will focus on the physical relevant case (i.e., the FCH energy with a logarithmic, possibly double-well potential $F$) and assume that
$$\lambda,\ \eta,\ \epsilon >0.$$

In order to design a numerical scheme with unconditional energy stability, we shall perform a convex-concave decomposition for the FCH energy given by \eqref{FCH energy}. In view of \eqref{B} and \eqref{beta}, it can be rewritten in the following way:
\begin{align}
	E(\phi)
	&=\frac{\epsilon^4}{2}\|\Delta\phi\|_{L^2(\Omega)}^2 + \frac{1}{2}\|\beta(\phi)\|_{L^2(\Omega)}^2
	+ \left(\frac{\lambda^2}{2}+\frac{\lambda\epsilon^p\eta}{2}\right)\|\phi\|_{L^2(\Omega)}^2 -\epsilon^p\eta\int_{\Omega}B(\phi)\,\mathrm{d}x \notag
	\\
	&\quad- \left(\frac{\epsilon^{2+p}\eta}{2}+\lambda\epsilon^2\right)\|\nabla\phi\|_{L^2(\Omega)}^2 -\lambda\int_{\Omega}\phi\beta(\phi)\,\mathrm{d}x -\epsilon^2\int_{\Omega}\beta(\phi)\Delta\phi\, \mathrm{d}x.
	\label{FCH energy2}
\end{align}
The term $-\int_{\Omega}\beta(\phi)\Delta\phi\, \mathrm{d}x$ in \eqref{FCH energy2} turns out to be tricky since it is neither convex nor concave.
To overcome this difficulty, using the periodic boundary conditions, we perform integration by parts and obtain
\begin{align}
	E(\phi)
	&=\frac{\epsilon^4}{2}\|\Delta\phi\|_{L^2(\Omega)}^2 + \frac{1}{2}\|\beta(\phi)\|_{L^2(\Omega)}^2
	+ \left(\frac{\lambda^2}{2}+\frac{\lambda\epsilon^p\eta}{2}\right)\|\phi\|_{L^2(\Omega)}^2
	 -\epsilon^p\eta\int_{\Omega}B(\phi)\,\mathrm{d}x\notag \\
	&\quad- \left(\frac{\epsilon^{2+p}\eta}{2}+\lambda\epsilon^2\right)\|\nabla\phi\|_{L^2(\Omega)}^2 -\lambda\int_{\Omega}\phi\beta(\phi)\,\mathrm{d}x +\epsilon^2\int_{\Omega}\beta'(\phi)|\nabla\phi|^2\, \mathrm{d}x.
	\label{FCH energy3}
\end{align}
The above procedure is meaningful, if the function  $\phi$ is sufficiently regular, for example, $\phi\in H^{2}_{\mathrm{per}}(\Omega)$ and satisfies the strict separation property (keeping in mind the singular nature of $\beta$ and its derivatives).

Thanks to \eqref{FCH energy3}, we can derive the following result:
\begin{prop}\label{c-c decomposition}
	The reformulated FCH energy \eqref{FCH energy3} possesses a convex-concave decomposition such that
	\begin{equation}
		E(\phi) = E^c(\phi) - E^e(\phi), \notag
	\end{equation}
	where
	\begin{align*}
		\begin{aligned}
			&E^c(\phi) = \frac{\epsilon^4}{2}\|\Delta\phi\|_{L^2(\Omega)}^2 + \frac{1}{2}\|\beta(\phi)\|_{L^2(\Omega)}^2 + \left(\frac{\lambda^2}{2}+\frac{\lambda\epsilon^p\eta}{2}\right)\|\phi\|_{L^2(\Omega)}^2
			+ \epsilon^2\int_{\Omega}\beta'(\phi)|\nabla\phi|^2\, \mathrm{d}x, \\
			&E^e(\phi) = \left(\frac{\epsilon^{2+p}\eta}{2}+\lambda\epsilon^2\right)\|\nabla\phi\|_{L^2(\Omega)}^2 + \lambda\int_{\Omega}\phi\beta(\phi)\,\mathrm{d}x +\epsilon^p\eta\int_{\Omega}B(\phi)\,\mathrm{d}x.
		\end{aligned}	
	\end{align*}
	For any given $\delta\in (0,1)$, both $E^c(\phi)$ and $E^e(\phi)$ are strictly convex when restricted to the set of functions $\phi\in H^{2}_{\mathrm{per}}(\Omega)$ satisfying $|\phi(x)|\leq 1-\delta$ in $\Omega$.
\end{prop}
\begin{proof}
	We adapt the arguments in \cite[Lemma 5.3]{SCHIMPERNA2020}. For any $\phi\in H^{2}_{\mathrm{per}}(\Omega)$ satisfying $|\phi(x)|\leq 1-\delta$ for some $\delta\in(0,1)$, by a direct calculation we can deduce that
	\begin{align*}
		\left<(E^c)'(\phi),v\right> &= \epsilon^4\int_\Omega \Delta \phi\Delta v \,\mathrm{d}x + \int_\Omega \beta(\phi)\beta'(\phi)v \,\mathrm{d}x
		+ \left(\lambda^2+\lambda\epsilon^p\eta\right)\int_\Omega \phi v \,\mathrm{d}x \\
		&\quad + \epsilon^2 \int_\Omega \big(2\beta'(\phi)\nabla \phi \cdot \nabla v+ \beta''(\phi)|\nabla\phi|^2v\big) \,\mathrm{d}x,\\
		\left<(E^e)'(\phi),v\right> & =  \left(\epsilon^{2+p}\eta+2\lambda\epsilon^2\right)\int_\Omega \nabla \phi \cdot \nabla v \,\mathrm{d}x
		+ \lambda \int_\Omega (\beta(\phi)+\phi\beta'(\phi))v \,\mathrm{d}x + \epsilon^p\eta\int_{\Omega}\beta(\phi)v \,\mathrm{d}x,
	\end{align*}
	for any $v\in H^{2}_{\mathrm{per}}(\Omega)$. A further calculation yields
	\begin{align*}
		\left<(E^c)''(\phi)v,w\right> &= \epsilon^4\int_\Omega \Delta v\Delta w \,\mathrm{d}x
		+ \int_{\Omega}\left[(\beta'(\phi))^2+\beta(\phi)\beta''(\phi)\right] vw \,\mathrm{d}x
		+ \left(\lambda^2+\lambda\epsilon^p\eta\right)\int_\Omega vw \,\mathrm{d}x \\
		&\quad + \epsilon^2\int_\Omega \big(2\beta'(\phi)\nabla v \cdot \nabla w + 2\beta''(\phi)v\nabla \phi \cdot \nabla w\big)\,\mathrm{d}x\\
		&\quad +\epsilon^2\int_\Omega \big( 2 \beta''(\phi)w\nabla\phi\cdot\nabla v + \beta'''(\phi)|\nabla\phi|^2 vw \big)\,\mathrm{d}x,\\
		\left<(E^e)''(\phi)v,w\right> &= (\epsilon^{2+p}\eta+2\lambda\epsilon^2)\int_\Omega \nabla v\cdot\nabla w \,\mathrm{d}x
		+ \int_{\Omega}\big[\left(2\lambda+\epsilon^p\eta\right)\beta'(\phi) +\lambda\phi\beta''(\phi)\big]vw \,\mathrm{d}x,
	\end{align*}
	for any $v, w\in H^{2}_{\mathrm{per}}(\Omega)$. Then taking $w=v$, we find that (cf. \cite[(6.5)]{MR3032968})
	\begin{align*}
		\left<(E^c)''(\phi)v,v\right>
		&\geq \epsilon^4\int_\Omega |\Delta v|^2 \,\mathrm{d}x
		+ \int_{\Omega}\left[(\beta'(\phi))^2+\beta(\phi)\beta''(\phi)\right] v^2 \,\mathrm{d}x
		+ \left(\lambda^2+\lambda\epsilon^p\eta\right)\int_\Omega v^2 \,\mathrm{d}x \\
		&\quad + \epsilon^2\int_\Omega \left[2\beta'(\phi)-\frac{4(\beta''(\phi))^2}{\beta'''(\phi)}\right]|\nabla v|^2\,\mathrm{d}x,\\
		\left<(E^e)''(\phi)v,v\right> &= (\epsilon^{2+p}\eta+2\lambda\epsilon^2)\int_\Omega |\nabla v|^2 \,\mathrm{d}x
		+ \int_{\Omega}\big[\left(2\lambda+\epsilon^p\eta\right)\beta'(\phi) +\lambda\phi\beta''(\phi)\big]v^2 \,\mathrm{d}x.
	\end{align*}
	Recalling \eqref{beta}, we observe that for $r\in (-1,1)$, it holds
	\begin{align}
		&\beta'(r)=\frac{2}{1-r^2} \geq 2, \label{be1}\\
		&\beta(r)\beta''(r) = \ln\frac{1+r}{1-r}\frac{4r}{(1-r^2)^2}\geq 0, \label{be2} \\
		&\beta'(r)-\frac{2(\beta''(r))^2}{\beta'''(r)}= \frac{2}{1+3r^2}>\frac12, \label{be3}\\
		&r\beta''(r)=\frac{4r^2}{(1-r^2)^2}\geq 0. \label{be4}
	\end{align}
	As a consequence, we can conclude
	$$
	\left<(E^c)''(\phi)v,v\right>\geq C\|v\|_{H^2(\Omega)}^2\quad \text{and}\quad \left<(E^e)''(\phi)v,v\right>\geq C\|v\|_{H^1(\Omega)}^2,\quad \forall\, v\in H^{2}_{\mathrm{per}}(\Omega),
	$$
	where $C>0$ is a constant that may depend on the positive parameters $\lambda$, $\eta$, $\epsilon$ and $\delta$, but are independent of $\phi$ and $v$. The proof is complete.
\end{proof}

Corresponding to \eqref{FCH energy3}, for any $h>0$, we introduce the discrete energy
\begin{align}
	E_h(\phi)
	&= \frac{\epsilon^4}{2}\|\Delta_h\phi\|_2^2 + \frac{1}{2}\|\beta(\phi)\|_2^2 + \left(\frac{\lambda^2}{2} +\frac{\lambda\epsilon^p\eta}{2}\right)\|\phi\|_2^2 - \left(\frac{\epsilon^{2+p}\eta}{2}+\lambda\epsilon^2\right) \|\nabla_h\phi\|_2^2\notag \\
	&\quad-\lambda\left<\phi,\beta(\phi)\right>_{\Omega} +\epsilon^2\left<\beta'(\phi),\sum_{\zeta=x,y,z}a_{\zeta}\left(\left|D_{\zeta}\phi\right|^2\right)\right>_{\Omega} -\epsilon^p\eta\left<B(\phi),1\right>_{\Omega}.
	\label{discrete energy}
\end{align}
In analogy to Proposition \ref{c-c decomposition}, we can deduce that
\begin{cor}\label{dc-c decomposition}
	The discrete energy $E_h$ possesses a convex-concave decomposition:
	\begin{equation}
		E_h(\phi) = E_h^c(\phi) - E_h^e(\phi), \label{Eh}
	\end{equation}
	where
	\begin{align}
		&E_h^c(\phi) = \frac{\epsilon^4}{2}\|\Delta_h\phi\|_2^2 + \frac{1}{2}\|\beta(\phi)\|_2^2 + \left(\frac{\lambda^2}{2}+\frac{\lambda\epsilon^p\eta}{2}\right)\|\phi\|_2^2  + \epsilon^2\left<\beta'(\phi),\sum_{\zeta=x,y,z}a_{\zeta}\left(\left|D_{\zeta}\phi\right|^2\right)\right>_{\Omega}, \label{Ehc}\\
		&E_h^e(\phi) = \left(\frac{\epsilon^{2+p}\eta}{2}+\lambda\epsilon^2\right)\|\nabla_h\phi\|_2^2 + \lambda\left<\phi,\beta(\phi)\right>_{\Omega}+\epsilon^p\eta\left<B(\phi),1\right>_{\Omega}. \label{Ehe}
	\end{align}
	Both $E_h^c$ and $E_h^e$ are strictly convex when restricted to the set of functions $\phi\in \mathcal{C}_{\mathrm{per}}$ satisfying $\|\phi\|_\infty< 1$.
\end{cor}

\begin{rem}\label{dissep}
	Since $\phi\in \mathcal{C}_{\mathrm{per}}$ only takes its values at finite points in $\Omega$, then the condition  $\|\phi\|_\infty< 1$ indeed implies that there exists some $\delta\in (0,1)$ such that $\|\phi\|_\infty\leq 1-\delta$.
\end{rem}

\subsection{The first order numerical scheme}
Using the convex splitting method \cite{ES93,Eyre} and the convex-concave energy decomposition in Corollary \ref{dc-c decomposition}, we introduce a first order semi-implicit numerical scheme for the functionalized Cahn-Hilliard equation \eqref{phi} (cf. also \eqref{E:u} for an alternative expression of $\mu$). To this end, for a uniform time step $\Delta t>0$, given $\phi^n\in \mathcal{C}_{\mathrm{per}}$ with $ \|\phi^n\|_{\infty} < 1$ ($n\in \mathbb{N}$), we look for the discrete solution $\phi^{n+1}$ such that
\begin{align}
	&\frac{\phi^{n+1}-\phi^{n}}{\Delta t} = \Delta_h\mu^{n+1},\label{numerical phi}\\
	&\mu^{n+1} =\epsilon^4\Delta_h^2\phi^{n+1} + \beta(\phi^{n+1})\beta'(\phi^{n+1}) + \epsilon^2\sum_{\zeta=x,y,z}\beta''(\phi^{n+1})a_{\zeta}\left(|D_{\zeta}\phi^{n+1}|^2\right)\notag	\\
	&\qquad\quad - 2\epsilon^2\sum_{\zeta=x,y,z}d_{\zeta}\left(A_{\zeta}\big(\beta'(\phi^{n+1})\big)D_{\zeta}\phi^{n+1}\right)+ \lambda(\lambda+\epsilon^p\eta)\phi^{n+1}\notag 	\\
	&\qquad\quad   + \epsilon^2(2\lambda+\epsilon^p\eta)\Delta_h\phi^{n}-\lambda\phi^{n}\beta'(\phi^{n})-(\lambda+\epsilon^p\eta)\beta(\phi^n).\label{numerical mu}
\end{align}
It is straightforward to check that for every solution to the numerical scheme (\ref{numerical phi})--(\ref{numerical mu}), if it exists, then it must satisfy the property of mass conservation (cf. Proposition \ref{prop:well} for the continuous problem). More precisely, we have
\begin{prop}[Mass conservation]\label{prop:dmass}
	For every $n\in \mathbb{Z}^+$, it holds $\overline{\phi^n}=\overline{\phi^0}$.
	As a consequence, we have  $\big<\phi^n-\overline{\phi^0},1\big>_{\Omega}=0$ for all $n\in \mathbb{Z}^+$.
\end{prop}

\begin{rem}\label{diffeta}
	For the Cahn-Hilliard-Willmore system, namely, the case with $\eta\leq 0$ and $\lambda>0$, using a similar idea for \eqref{numerical phi}--\eqref{numerical mu}, we can propose the following convex-splitting scheme:
	\begin{align}
		&\frac{\phi^{n+1}-\phi^{n}}{\Delta t} = \Delta_h\mu^{n+1},\label{CHW phi}\\
		&\mu^{n+1} =\epsilon^4\Delta_h^2\phi^{n+1} + \beta(\phi^{n+1})\beta'(\phi^{n+1}) + \epsilon^2\sum_{\zeta=x,y,z}\beta''(\phi^{n+1})a_{\zeta}\left(|D_{\zeta}\phi^{n+1}|^2\right)\notag	\\
		&\qquad\quad - 2\epsilon^2\sum_{\zeta=x,y,z}d_{\zeta}\left(A_{\zeta}\big(\beta'(\phi^{n+1})\big)D_{\zeta}\phi^{n+1}\right)+ \lambda^2\phi^{n+1} + \epsilon^{p+2}\eta\Delta_h\phi^{n+1}\notag 	\\
		&\qquad\quad -\epsilon^p\eta\beta(\phi^{n+1})  + \lambda\epsilon^p\eta\phi^n + 2\lambda\epsilon^2\Delta_h\phi^{n}-\lambda\phi^{n}\beta'(\phi^{n})-\lambda\beta(\phi^n).
		\label{CHW mu}
	\end{align}
	Furthermore, when $\eta\leq 0$ and $\lambda \leq 0$, we can replace the last four terms on the right-hand side of  \eqref{CHW mu} by
	$ \lambda\epsilon^p\eta\phi^{n+1} + 2\lambda\epsilon^2\Delta_h\phi^{n+1}-\lambda\phi^{n+1}\beta'(\phi^{n+1})-\lambda\beta(\phi^{n+1})$.
	Theoretic analysis and numerical experiments of these modified schemes will be carried out in the future study.
\end{rem}

\section{Unique Solvability and Positivity-Preserving Property}
\label{sec:positivity}
\setcounter{equation}{0}

Our aim in this section is to show that for every $n\in \mathbb{N}$, the numerical scheme \eqref{numerical phi}--\eqref{numerical mu} is uniquely solvable for any time step size $\Delta t>0$ and grid size $h>0$. In particular, the discrete solution $\phi^n$ satisfied the positivity-preserving property such that it is strictly separated from the pure states $\pm 1$, i.e., $-1<\phi^n_{i,j,k}<1$ (cf. \cite{CJ3W2022,chen2019} for the classical Cahn-Hilliard equation).

The main result of this section reads as follows:
\begin{thm}[Unique solvability and positivity-preserving property] \label{main1}
	Let $n\in \mathbb{N}$ and $\Delta t, h>0$. Given $\phi^n\in \mathcal{C}_{\mathrm{per}}$ with $ \|\phi^n\|_{\infty} < 1$, then the discrete system  \eqref{numerical phi}--\eqref{numerical mu} admits a unique solution $\phi^{n+1}\in \mathcal{C}_{\mathrm{per}}$ such that
	\begin{align}
		\phi^{n+1}-\overline{\phi^0} \in\mathring{\mathcal{C}}_{\mathrm{per}}\quad \text{and}\quad \|\phi^{n+1}\|_{\infty}<1. \label{masssep}
	\end{align}
\end{thm}
\begin{proof}
	We divide the proof of Theorem \ref{main1} into several steps.
	
	\textbf{Step 1.} Define the discrete energy:
	\begin{align}
		\mathcal{J}^n(\phi)
		&=\frac{1}{2\Delta t}\|\phi-\phi^n\|_{-1,h}^2 + E_h^c(\phi) +\left<f^n, \phi\right>_{\Omega},
		\label{J-energy}
	\end{align}
	where $E_h^c(\phi)$ is given by \eqref{Ehc} and
	\begin{align}
		f^n &:= \epsilon^2(2\lambda+\epsilon^p\eta)\Delta_h\phi^n-\lambda\phi^n\beta'(\phi^n)-(\lambda+\epsilon^p\eta)\beta(\phi^n).
		\label{fn}
	\end{align}
	In view of Proposition \ref{prop:dmass}, we introduce the admissible set
	\begin{equation}
		\mathcal{A}_h:=\left\{\phi\in\mathcal{C}_{\mathrm{per}}\,\left|\ \|\phi\|_{\infty}< 1,\ \left<\phi-\overline{\phi^0},1\right>_{\Omega}=0\right.\right\} \subset\mathbb{R}^{N^3}. \label{Ah}
	\end{equation}
	The existence of a solution to the discrete system (\ref{numerical phi})--(\ref{numerical mu}) can be proved by seeking a minimizer of the energy $\mathcal{J}^n$ over the set $\mathcal{A}_h$. For the sake of convenience and thanks to the mass conservation, we introduce the new variable with zero mean
	$$\varphi=\phi- \overline{\phi^0}$$
	and rewrite $\mathcal{J}^n$ into the following equivalent form:
	\begin{align}
		\mathcal{F}^n(\varphi)&:=\mathcal{J}^n(\varphi+\overline{\phi^0})\notag \\
		&=\frac{1}{2\Delta t}\|\varphi+\overline{\phi^0}-\phi^n\|_{-1,h}^2 +\frac{\epsilon^4}{2}\|\Delta_h\varphi\|_2^2 +\frac{1}{2}\|\beta(\varphi+\overline{\phi^0})\|_2^2 \notag \\
		&\quad +\epsilon^2\left<\beta'(\varphi+\overline{\phi^0}),\sum_{\zeta=x,y,z}a_{\zeta}\left(\left|D_{\zeta}\phi^{n+1}\right|^2\right)\right>_{\Omega} +\frac{\lambda(\lambda+\epsilon^p\eta)}{2}\|\varphi+\overline{\phi^0}\|_2^2 \notag \\
		&\quad +\left<f^n, \varphi+\overline{\phi^0}\right>_{\Omega}.\notag
	\end{align}
	Then $\mathcal{F}^n$ is defined on the shifted admissible set given by
	$$
	\mathring{\mathcal{A}}_h :=\left\{\varphi\in\mathring{\mathcal{C}}_{\mathrm{per}}\, \Big{|}\, -1-\overline{\phi^0}< \varphi_{i,j,k} < 1-\overline{\phi^0},\ \forall i,j,k\right\} \subset\mathbb{R}^{N^3}.
	$$
	It is straightforward to check that if $\varphi\in\mathring{\mathcal{A}}_h$ minimizes $\mathcal{F}^n$, then $\phi:=\varphi+\overline{\phi^0}\in\mathcal{A}_h$ minimizes $\mathcal{J}^n$, and vice versa. \smallskip
	
	\textbf{Step 2.} To show the existence of a minimizer of $\mathcal{F}^n$ in $\mathring{\mathcal{A}}_{h}$, for any given $\delta\in(0,1/2)$, let us consider the auxiliary set
	\begin{equation} \label{A_h^delta}
		\mathring{\mathcal{A}}_{h,\delta} :=\left\{\varphi\in\mathring{\mathcal{C}}_{\mathrm{per}}\, \Big{|}\, -1-\overline{\phi^0}+\delta\leq\varphi\leq 1-\overline{\phi^0}-\delta\right\}\subset \mathring{\mathcal{A}}_h,
	\end{equation}
	which is a bounded, closed and convex subset of $\mathring{\mathcal{C}}_{\mathrm{per}}$.
	Thus from the construction of $\mathcal{F}^n$, we easily infer that it admits at least one minimizer over $\mathring{\mathcal{A}}_{h,\delta}$. \smallskip
	
	\textbf{Step 3.} We now derive the key property that every minimizer of $\mathcal{F}^n$ cannot occur on the boundary of $\mathring{\mathcal{A}}_{h,\delta}$, provided that the parameter $\delta$ is sufficiently small. Here, by the boundary of $\mathring{\mathcal{A}}_{h,\delta}$, we mean the set of functions $\psi\in\mathring{\mathcal{A}}_{h,\delta}$ such that $\|\psi+\overline{\phi^0}\|_{\infty}=1-\delta$.
	
	The above claim can be proved by using a contradiction argument (cf. \cite{chen2019} for the Cahn-Hilliard equation). Suppose that for some $\delta \in (0,1/2)$, a minimizer of $\mathcal{F}^n$ in $\mathring{\mathcal{A}}_{h,\delta}$, which is denoted by  $\varphi^{*}$, occurs at a boundary point of $\mathring{\mathcal{A}}_{h,\delta}$. Then there is at least one grid point $\vec{\alpha}_0=(i_0,j_0,k_0)$ such that $|\varphi_{\vec{\alpha}_0}^{*}+\overline{\phi^0}|=1-\delta$. Without loss of generality, we assume that $\varphi_{\vec{\alpha}_0}^{*}+\overline{\phi^0}=-1+\delta$. This implies that the grid function $\varphi^*$ takes its (global) minimum at $\vec{\alpha}_0$. On the other hand, let $\vec{\alpha}_1=(i_1,j_1,k_1)$ be a grid point at which $\varphi^*$ achieves its maximum. It follows that $\varphi_{\vec{\alpha}_1}^*+\overline{\phi^0}\leq 1-\delta$, since $\varphi^{*}\in \mathring{\mathcal{A}}_{h,\delta}$.
	Besides, from the fact $\overline{\varphi^*}=0$, we infer that $\varphi_{\vec{\alpha}_1}^*\geq 0$.
	
	Since $\mathcal{F}^n$ is indeed smooth over $\mathring{\mathcal{A}}_{h,\delta}$, then its directional derivative can be calculated as
	\begin{align*}
		& \frac{\mathrm{d}\mathcal{F}^n(\varphi^*+s\psi)}{\mathrm{d}s}\Big{|}_{s=0}\\
		&\quad = \frac{1}{\Delta t}\left<(-\Delta_h)^{-1}(\varphi^*-\phi^n+\overline{\phi^0}),\psi\right>_{\Omega} +\epsilon^4\left<\Delta_h^2\varphi^*,\psi\right>+\lambda(\lambda+\epsilon^p\eta)\left<\varphi^*+\overline{\phi^0},\psi\right>_{\Omega}\\
		&\qquad +\epsilon^2\left<\sum_{\zeta=x,y,z}\left[\beta''(\varphi^*+\overline{\phi^0})a_{\zeta}\left(|D_{\zeta}\varphi^*|^2\right) -2d_{\zeta}\left(A_{\zeta}\big(\beta'(\varphi^*+\overline{\phi^0})\big)D_{\zeta}\varphi^*\right)\right],\psi\right>_{\Omega}\\
		&\qquad +\left<\beta(\varphi^*+\overline{\phi^0})\beta'(\varphi^*+\overline{\phi^0}),\psi\right>_{\Omega}+\left<f^n,\psi\right>_{\Omega},\quad \forall\, \psi\in\mathring{\mathcal{C}}_{\mathrm{per}}.	\end{align*}
	In particular, we choose the direction $\psi\in\mathring{\mathcal{C}}_{\mathrm{per}}$ with $\psi_{i,j,k} = \delta_{i,i_0}\delta_{j,j_0}\delta_{k,k_0}-\delta_{i,i_1}\delta_{j,j_1}\delta_{k,k_1}$.
	This yields that
	\begin{align}
		&	\frac{1}{h^3}\frac{\mathrm{d}\mathcal{F}^n(\varphi^*+s\psi)}{\mathrm{d}s} \Big{|}_{s=0} \notag \\
		& \quad
		=\frac{1}{\Delta t}(-\Delta_h)^{-1}(\varphi^*-\phi^n+\overline{\phi^0})_{\vec{\alpha}_0}-\frac{1}{\Delta t}(-\Delta_h)^{-1}(\varphi^*-\phi^n+\overline{\phi^0})_{\vec{\alpha}_1}\notag \\
		&\qquad +\epsilon^4\Delta_h^2(\varphi_{\vec{\alpha}_0}^*-\varphi_{\vec{\alpha}_1}^*) +\lambda(\lambda+\epsilon^p\eta)(\varphi_{\vec{\alpha}_0}^*-\varphi_{\vec{\alpha}_1}^*) \notag\\
		&\qquad +\epsilon^2\left(\sum_{\zeta=x,y,z}\left[ \beta''(\varphi_{\vec{\alpha}_0}^*+\overline{\phi^0})a_{\zeta}\left(|D_{\zeta}\varphi_{\vec{\alpha}_0}^*|^2\right) -2d_{\zeta}\left(A_{\zeta}\big(\beta'(\varphi_{\vec{\alpha}_0}^*+\overline{\phi^0})\big)D_{\zeta}\varphi_{\vec{\alpha}_0}^*\right)\right] \right)\notag \\
		&\qquad -\epsilon^2\left(\sum_{\zeta=x,y,z}\left[ \beta''(\varphi_{\vec{\alpha}_1}^*+\overline{\phi^0})a_{\zeta}\left(|D_{\zeta}\varphi_{\vec{\alpha}_1}^*|^2\right) -2d_{\zeta}\left(A_{\zeta}\big(\beta'(\varphi_{\vec{\alpha}_1}^*+\overline{\phi^0})\big)D_{\zeta}\varphi_{\vec{\alpha}_1}^*\right)\right] \right)\notag \\
		&\qquad +\beta(\varphi_{\vec{\alpha}_0}^*+\overline{\phi^0})\beta'(\varphi_{\vec{\alpha}_0}^* +\overline{\phi^0}) -\beta(\varphi_{\vec{\alpha}_1}^*+\overline{\phi^0})\beta'(\varphi_{\vec{\alpha}_1}^*+\overline{\phi^0})\notag \\
		&\qquad +(f_{\vec{\alpha}_0}^n-f_{\vec{\alpha}_1}^n).
		\label{directional derivative}
	\end{align}
	The first two terms on the right-hand side of \eqref{directional derivative} can be estimated by using Lemma \ref{lemma2} such that
	\begin{equation}\label{sec3est1}
		-2C_0\leq(-\Delta_h)^{-1}(\varphi^*-\phi^n+\overline{\phi^0})_{\vec{\alpha}_0} -(-\Delta_h)^{-1}(\varphi^*-\phi^n+\overline{\phi^0})_{\vec{\alpha}_1}\leq 2C_0.
	\end{equation}
	For the third and fourth terms, recalling the definition of $\Delta_h$ and (\ref{A_h^delta}), we find that
	\begin{align}
		&\|\Delta_h^2(\varphi_{\vec{\alpha}_0}^*-\varphi_{\vec{\alpha}_1}^*)\|_{\infty} \leq\frac{288(1-\delta)}{h^4}\leq\frac{288}{h^4}, \label{sec3est2a}\\
		&\|\varphi_{\vec{\alpha}_0}^*-\varphi_{\vec{\alpha}_1}^*\|_{\infty} \leq 2(1-\delta)\leq 2.\label{sec3est2b}
	\end{align}
	
	Estimate for the fifth term involving $\beta''(\varphi_{\vec{\alpha}_0}^*+\overline{\phi^0})a_{\zeta}\big(|D_{\zeta}\varphi_{\vec{\alpha}_0}^*|^2\big) -2d_{\zeta}\big(A_{\zeta}\big(\beta'(\varphi_{\vec{\alpha}_0}^*+\overline{\phi^0})\big)D_{\zeta}\varphi_{\vec{\alpha}_0}^*\big)$ turns out to be more tricky. For simplicity, we only consider the estimate along the $x$-direction. A direct calculation yields that
	\begin{equation*}
		\begin{aligned}
			& h^2\Big[\beta''(\varphi_{\vec{\alpha}_0}^*+\overline{\phi^0})a_x\left(|D_x\varphi_{\vec{\alpha}_0}^*|^2\right) -2d_x\left(A_x\big(\beta'(\varphi_{\vec{\alpha}_0}^*+\overline{\phi^0})\big)D_x\varphi_{\vec{\alpha}_0}^*\right)\Big]\\
			&\quad =\frac{1}{2}\beta''(\varphi_{i_0}^*+\overline{\phi^0})\left( (\varphi_{i_0+1}^*-\varphi_{i_0}^*)^2 + (\varphi_{i_0}^*-\varphi_{i_0-1}^*)^2 \right)\\
			&\qquad -\left(\beta'(\varphi_{i_0}^*+\overline{\phi^0})+\beta'(\varphi_{i_0+1}^*+\overline{\phi^0})\right)(\varphi_{i_0+1}^*-\varphi_{i_0}^*)\\
			&\qquad + \left(\beta'(\varphi_{i_0}^*+\overline{\phi^0})+\beta'(\varphi_{i_0-1}^*+\overline{\phi^0})\right) (\varphi_{i_0}^*-\varphi_{i_0-1}^*).
		\end{aligned}
	\end{equation*}
	Here, we denote $\varphi_{i,j_0,k_0}^*$ by $\varphi_{i}^*$ for simplicity (as we only consider the $x$-direction). Using Taylor's expansion for $\beta'(\varphi_{i_0+1}^*+\overline{\phi^0})$ and $\beta'(\varphi_{i_0-1}^*+\overline{\phi^0})$ at $\varphi_{i_0}^*+\overline{\phi^0}$,
	we obtain
	\begin{align}
		&h^2\Big[\beta''(\varphi_{\vec{\alpha}_0}^*+\overline{\phi^0})a_x\left(|D_x\varphi_{\vec{\alpha}_0}^*|^2\right) -2d_x\left(A_x\big(\beta'(\varphi_{\vec{\alpha}_0}^*+\overline{\phi^0})\big)D_x\varphi_{\vec{\alpha}_0}^*\right)\Big] \notag \\
		&\quad =\frac{1}{2}\beta''(\varphi_{i_0}^*+\overline{\phi^0})\left( (\varphi_{i_0+1}^*-\varphi_{i_0}^*)^2 + (\varphi_{i_0}^*-\varphi_{i_0-1}^*)^2 \right)  \notag \\
		&\qquad -\left(\beta'(\varphi_{i_0}^*+\overline{\phi^0}) +\frac{1}{2}\beta'(\varphi_{i_0+1}^*+\overline{\phi^0})\right)(\varphi_{i_0+1}^*-\varphi_{i_0}^*) \notag \\
		&\qquad + \left(\beta'(\varphi_{i_0}^*+\overline{\phi^0})+\frac{1}{2}\beta'(\varphi_{i_0-1}^*+\overline{\phi^0})\right) (\varphi_{i_0}^*-\varphi_{i_0-1}^*) \notag \\
		&\qquad - \frac{1}{2}\beta'(\varphi_{i_0}^*+\overline{\phi^0})(\varphi_{i_0+1}^*-\varphi_{i_0}^*) - \frac{1}{2}\beta''(\varphi_{i_0}^*+\overline{\phi^0})(\varphi_{i_0+1}^* -\varphi_{i_0}^*)^2 -\frac{\beta'''(\xi_1+\overline{\phi^0})}{4}(\varphi_{i_0+1}^*-\varphi_{i_0}^*)^3 \notag \\
		&\qquad +\frac{1}{2}\beta'(\varphi_{i_0}^*+\overline{\phi^0}) (\varphi_{i_0}^*-\varphi_{i_0-1}^*) - \frac{1}{2}\beta''(\varphi_{i_0}^*+\overline{\phi^0})(\varphi_{i_0}^* -\varphi_{i_0-1}^*)^2
		+\frac{\beta'''(\xi_2+\overline{\phi^0})}{4}(\varphi_{i_0}^*-\varphi_{i_0-1}^*)^3 \notag \\
		&\quad =  -\left(\frac32\beta'(\varphi_{i_0}^*+\overline{\phi^0}) +\frac{1}{2}\beta'(\varphi_{i_0+1}^*+\overline{\phi^0})\right)(\varphi_{i_0+1}^*-\varphi_{i_0}^*) \notag \\
		&\qquad + \left(\frac32\beta'(\varphi_{i_0}^*+\overline{\phi^0})+\frac{1}{2}\beta'(\varphi_{i_0-1}^*+\overline{\phi^0})\right) (\varphi_{i_0}^*-\varphi_{i_0-1}^*) \notag \\
		&\qquad -\frac{\beta'''(\xi_1+\overline{\phi^0})}{4}(\varphi_{i_0+1}^*-\varphi_{i_0}^*)^3 +\frac{\beta'''(\xi_2+\overline{\phi^0})}{4}(\varphi_{i_0}^*-\varphi_{i_0-1}^*)^3 \notag \\
		&\quad \leq 0. \label{sec3est3}
	\end{align}
	Here $\xi_1$ is between $\varphi_{i_0}^*$ and $\varphi_{i_0+1}^*$, while $\xi_2$ is between $\varphi_{i_0}^*$ and $\varphi_{i_0-1}^*$. The last inequality in \eqref{sec3est3} holds thanks to the definition of $\vec{\alpha}_0$ (i.e., the minimum point of $\varphi^*$) so that
	$\varphi_{i_0+1}^*-\varphi_{i_0}^*\geq 0$, $\varphi_{i_0}^*-\varphi_{i_0-1}^*\leq 0$,  and the fact that $\beta'(r),\ \beta'''(r)>0$ for any $r\in(-1,1)$. Similar results can be derived for the $y$, $z$ direction. Besides, for the term involving $\vec{\alpha}_1$, we can conclude
	
	\begin{equation}
		\beta''(\varphi_{\vec{\alpha}_1}^* +\overline{\phi^0})a_{\zeta}\left(|D_{\zeta}\varphi_{\vec{\alpha}_1}^*|^2\right) -2d_{\zeta}\left(A_{\zeta}\big(\beta'(\varphi_{\vec{\alpha}_1}^*+\overline{\phi^0})\big)D_{\zeta}\varphi_{\vec{\alpha}_1}^*\right) \geq0, \quad\text{for}\ \zeta=x,y,z.
		\label{sec3est4}
	\end{equation}

	Next, we infer from \eqref{be1}--\eqref{be2} that the nonlinear function $\beta\beta'$ is strictly increasing on $(-1,1)$. Since  $\varphi_{\vec{\alpha}_0}^*+\overline{\phi^0}  =-1+\delta$ and $\varphi_{\vec{\alpha}_1}^*\geq0$, then we have
	\begin{align}
		&	\beta(\varphi_{\vec{\alpha}_0}^* +\overline{\phi^0})\beta'(\varphi_{\vec{\alpha}_0}^* +\overline{\phi^0}) -\beta(\varphi_{\vec{\alpha}_1}^* +\overline{\phi^0})\beta'(\varphi_{\vec{\alpha}_1}^* +\overline{\phi^0})\notag \\
		&\quad  \leq\left(\frac{1}{\delta}+\frac{1}{2-\delta}\right) \ln\frac{\delta}{2-\delta} -\left(\frac{1}{1+\overline{\phi^0}} +\frac{1}{1-\overline{\phi^0}}\right) \ln\frac{1+\overline{\phi^0}}{1-\overline{\phi^0}}.
		\label{sec3est5}
	\end{align}

	Finally, we treat the term involving $f^n$ (see \eqref{fn}). From the assumption $\|\phi^n\|_{\infty}<1$ and the fact that for any given $h>0$, there are only finite number of grid points, we can find some $\delta_n\in (0,1)$ such that
	\begin{equation}\label{priori est}
		-1+\delta_n\leq\phi_{i,j,k}^n\leq1-\delta_n, \quad\forall\,1\leq i,j,k\leq N.
	\end{equation}
	Here, the constant $\delta_n$ may depend on the parameter $n$. From the refined bound (\ref{priori est}) and the assumption $\lambda, \eta, \epsilon>0$, we see that
	\begin{equation}
		|f^n| \leq \frac{24\epsilon^2(2\lambda+\epsilon^p\eta)(1-\delta_n)}{h^2}-2\lambda(-1+\delta_n)\left(\frac{1}{\delta_n}+\frac{1}{2-\delta_n}\right) -2(\lambda+\epsilon^p\eta)\ln\frac{\delta_n}{2-\delta_n}.
		\label{sec3est6}
	\end{equation}
	
	Collecting the estimates (\ref{sec3est1})--(\ref{sec3est6}), we can deduce from \eqref{directional derivative} that
	\begin{align}
		&\frac{1}{h^3}\frac{\mathrm{d}\mathcal{F}^n(\varphi^*+s\psi)}{\mathrm{d}s}\Big{|}_{s=0} \leq g(\delta) + \max\{\widehat{C},1\}, \notag
	\end{align}
	where
	$$
	g(\delta)=\left(\frac{1}{\delta}+\frac{1}{2-\delta}\right)\ln\frac{\delta}{2-\delta}
	$$
	and
	\begin{align}
		\widehat{C}
		&= \frac{2C_2}{\Delta t} +\frac{288\epsilon^4}{h^4}+2\lambda(\lambda+\epsilon^p\eta) -\left(\frac{1}{1+\overline{\phi^0}}+\frac{1}{1-\overline{\phi^0}}\right) \ln\frac{1+\overline{\phi^0}}{1-\overline{\phi^0}}\notag \\
		&\quad +\frac{24\epsilon^2(2\lambda+\epsilon^p\eta)(1-\delta_n)}{h^2} -2\lambda(-1+\delta_n)\left(\frac{1}{\delta_n}+\frac{1}{2-\delta_n}\right)\notag \\
		&\quad -2(\lambda+\epsilon^p\eta)\ln\frac{\delta_n}{2-\delta_n}. \notag
	\end{align}
	Observe that $\widehat{C}$ is a constant that may depend on $\Delta t,\ h,\ \delta_n$, $\overline{\phi^0}$, $\lambda$, $\eta$, $\epsilon$, but is independent of $\delta$. Since
	$$
	\lim_{\delta\searrow 0^+} g(\delta)=-\infty,
	$$
	and $g(\delta)= \beta(-1+\delta)\beta'(-1+\delta)$ is strictly increasing on $(0,1)$, we can find some $\widehat{\delta}\in (0,1/2)$ sufficiently small such that
	\begin{align}
		g(\delta)+\max\{\widehat{C},1\}\leq -1,\quad \forall\, \delta\in (0,\widehat{\delta}\,]. \label{hatdelta}
	\end{align}
	which gives
	$$
	\frac{\mathrm{d}\mathcal{F}^n(\varphi^*+s\psi)}{\mathrm{d}s}\Big{|}_{s=0}< 0.
	$$
	This contradicts the assumption that $\mathcal{F}^n$ attains its minimum at $\varphi^*$, because the directional derivative is strictly negative along a specific direction pointing into the interior of $\mathring{\mathcal{A}}_{h,\delta}$.
	
	By a similar argument, we can show that for any $\delta \in (0,\widehat{\delta}\,]$, a minimizer of $\mathcal{F}^n$ over $\mathring{\mathcal{A}}_{h,\delta}$ cannot occur at a boundary point $\varphi^*$ such that $\varphi_{\vec{\alpha}_0}^*+\overline{\phi^0}=1-\delta$, for some $\vec{\alpha}_0$.
	
	In summary, there exists a $\widehat{\delta}\in (0,1/2)$ such that for every $\delta \in (0,\widehat{\delta}\,]$, the (global) minimum of $\mathcal{F}^n$ over the set  $\mathring{\mathcal{A}}_{h,\delta}$ exists and it can be only attained at a certain interior point $\varphi^*\in\mathring{\mathcal{A}}_{h,\delta}$. \smallskip
	
	\smallskip
	
	\textbf{Step 4.} Let $\varphi^*$ be a minimizer of $\mathcal{F}^n$ over $\mathring{\mathcal{A}}_{h,\widehat{\delta}}$, where $\widehat{\delta}$ is given in the previous step (see \eqref{hatdelta}). Below we show that $\varphi^*$ must be the unique minimizer of $\mathcal{F}^n$ over $\mathring{\mathcal{A}}_{h}$. In other words, $\phi^*=\varphi^*+\overline{\phi^0}$ is the unique minimizer of $\mathcal{J}^n$ over $\mathcal{A}_{h}$.
	
	First, for every $\delta \in (0,\widehat{\delta}\,]$, if $\varphi^*$ is a minimizer of $\mathcal{F}^n$ over $\mathring{\mathcal{A}}_{h,\delta}$, then it is unique. This is equivalent to say that  $\phi^*=\varphi^*+\overline{\phi^0}$ is the unique minimizer of $\mathcal{J}^n$ over $\mathcal{A}_{h,\delta}$. Indeed, similar to the proof of Proposition \ref{c-c decomposition}, we know that the discrete energy $E_h^c$ is strictly convex on $\mathcal{A}_{h,\delta}$. Since the first term in $\mathcal{J}^n$ is quadratic and the third term is linear in $\phi$, we deduce that $\mathcal{J}^n$ is strictly convex on the set $\mathcal{A}_{h,\delta}$ as well. Thanks to the strict convexity of $\mathcal{J}^n$ over $\mathcal{A}_{h,\delta}$, the minimizer $\phi^*$ is unique.
	
	Next, for every $\delta\in (0,\widehat{\delta}\,)$, we denote the unique minimizer of $\mathcal{F}^n$ over $\mathring{\mathcal{A}}_{h,\delta}$ by $\varphi^{\delta}$. Since $\mathring{\mathcal{A}}_{h,\widehat{\delta}}\subset  \mathring{\mathcal{A}}_{h,\delta}$, then we have
	$\mathcal{F}^n(\varphi^{\delta})\leq \mathcal{F}^n(\varphi^{*})$. Let us consider two cases.
	
	\textit{Case 1}. For all $\delta\in (0,\widehat{\delta}\,)$, it holds $\mathcal{F}^n(\varphi^{\delta})= \mathcal{F}^n(\varphi^{*})$. In this case, $\varphi^*$ is a minimizer of $\mathcal{F}^n$ over $\mathring{\mathcal{A}}_{h,\delta}$ as well. Thanks to the uniqueness, it easily follows that $\varphi^{\delta}=\varphi^*$ for any $\delta\in (0,\widehat{\delta}\,)$ and thus $\varphi^*$ is indeed the unique minimizer of $\mathcal{F}^n$ over $\mathring{\mathcal{A}}_{h}$.
	
	\textit{Case 2}. There exists some $\widetilde{\delta}\in (0,\widehat{\delta}\,)$, the corresponding unique minimizer of $\mathcal{F}^n$ over $\mathring{\mathcal{A}}_{h,\widetilde{\delta}}$, denoted by $\varphi^{**}$, satisfies $\mathcal{F}^n(\varphi^{**})< \mathcal{F}^n(\varphi^{*})$. In this case, from Step 3, we can find a constant $\delta^{**}>\widetilde{\delta}$, such that $\|\varphi^{**}+\overline{\phi^0}\|_\infty=1-\delta^{**}$. Besides, since $\mathcal{F}^n(\varphi^{**})< \mathcal{F}^n(\varphi^{*})$, we must have 	$\delta^{**}< \widehat{\delta}$. Now let us denote the unique minimizer of $\mathcal{F}^n$ over $\mathring{\mathcal{A}}_{h,\delta^{**}}$ by $\varphi^{\sharp}$. From Step 3, we find    $\|\varphi^{\sharp}+\overline{\phi^0}\|_\infty<1-\delta^{**}$, which implies that $\varphi^{\sharp}\neq \varphi^{**}$. On one hand, since $\varphi^{**}\in \mathring{\mathcal{A}}_{h,\delta^{**}}$, then $\mathcal{F}^n(\varphi^{\sharp})\leq \mathcal{F}^n(\varphi^{**})$. On the other hand, it follows from $\mathring{\mathcal{A}}_{h,\delta^{**}} \subset \mathring{\mathcal{A}}_{h,\widetilde{\delta}}$ that $\mathcal{F}^n(\varphi^{**})\leq \mathcal{F}^n(\varphi^{\sharp})$. Thus, $\mathcal{F}^n(\varphi^{\sharp})= \mathcal{F}^n(\varphi^{**})$ so that $\varphi^{\sharp}$ should also be a minimizer of $\mathcal{F}^n$ over $\mathring{\mathcal{A}}_{h,\widetilde{\delta}}$. Thanks to the uniqueness of minimizer in $\mathring{\mathcal{A}}_{h,\widetilde{\delta}}$, we have $\varphi^{\sharp}= \varphi^{**}$, which leads to a contradiction. Hence, the situation described in Case 2 will never happen.
	
	In summary, the energy functional $\mathcal{J}^n$ admits a unique minimizer $\phi^{*}$ over the set $\mathcal{A}_{h}$ such that $\varphi^*=\phi^*-\overline{\phi^0}$ satisfies
	$$
	\frac{\mathrm{d}\mathcal{F}^n(\varphi^*+s\psi)}{\mathrm{d}s}\Big{|}_{s=0}= 0,\quad \forall\, \psi\in\mathring{\mathcal{C}}_{\mathrm{per}}.
	$$
	This yields that $\phi^{n+1}=\phi^*$ is the unique solution to the discrete system \eqref{numerical phi}--\eqref{numerical mu}. In particular, by the same argument as that for \eqref{priori est}, we can find some $\delta_{n+1}\in (0,1/2)$ such that
	\begin{equation}\label{priori estb}
		-1+\delta_{n+1}\leq\phi_{i,j,k}^{n+1}\leq 1-\delta_{n+1},\quad\forall\,1\leq i,j,k\leq N.
	\end{equation}
	The proof of Theorem \ref{main1} is complete.
\end{proof}
\begin{rem}\label{se-bound}
	In view of \eqref{hatdelta}, the distance $\delta_{n+1}$ between the discrete solution $\phi^{n+1}$ and the pure states $\pm 1$ may not be uniform with respect to the index $n$. Nevertheless, when the spatial dimension is less than or equal to two, from Proposition \ref{prop:sepa}, the convergence analysis and error estimates given in Section \ref{sec:convergence}, we can derive a uniform separation property for the discrete solution $\phi^{n+1}$ (see \eqref{strict separation numerical} below).
\end{rem}

\section{Unconditional Energy Stability}
\label{sec:energy stability}
\setcounter{equation}{0}

In this section, we show that the numerical scheme \eqref{numerical phi}--\eqref{numerical mu} is unconditionally energy stable.
This property follows from the convex-concave decomposition of the discrete energy (recall Corollary \ref{dc-c decomposition}).
\begin{thm}[Unconditional energy stability]\label{e-stable}
	The discrete system \eqref{numerical phi}--\eqref{numerical mu} is unconditionally energy stable. Namely, for any $n\in \mathbb{N}$ and $\Delta t>0$, we have
	\begin{equation}
		E_h(\phi^{n+1}) +\Delta t\,\|\nabla_h\mu^{n+1}\|_2^2 \leq E_h(\phi^n),\label{disdecay}
	\end{equation}
	where the discrete energy $E_h$ is given by \eqref{Eh}.
\end{thm}

\begin{proof}
	Similar to the proof of Proposition \ref{c-c decomposition}, for the decomposed discrete energy $E_h(\phi)=E_h^c(\phi)-E_h^e(\phi)$, we have
	\begin{align*}
		\delta_{\phi}E_h^c(\phi)
		&= \epsilon^4\Delta_h^2\phi + \beta(\phi)\beta'(\phi) + \epsilon^2\sum_{\zeta=x,y,z}\beta''(\phi)a_{\zeta}\left(|D_{\zeta}\phi|^2\right)\notag	\\
		&\quad - 2\epsilon^2\sum_{\zeta=x,y,z}d_{\zeta}\left(A_{\zeta}\big(\beta'(\phi)\big)D_{\zeta}\phi\right)+ \lambda(\lambda+\epsilon^p\eta)\phi, \notag 	\\
		\delta_{\phi}E_h^e(\phi) &= -\epsilon^2(2\lambda+\epsilon^p\eta)\Delta_h\phi^{n}+\lambda\phi^{n}\beta'(\phi^{n}) +(\lambda+\epsilon^p\eta)\beta(\phi^n),
	\end{align*}
	for any $\phi\in \mathcal{C}_{\mathrm{per}}$ satisfying $\|\phi\|_\infty< 1$.
	Thanks to Theorem \ref{main1} and Corollary \ref{dc-c decomposition}, we can deduce the following inequalities for the discrete solution $\phi^{n+1}$:
	\begin{align*}
		& E_h^c(\phi^{n})-E_h^c(\phi^{n+1}) \geq \left< \delta_{\phi}E_h^c(\phi^{n+1}), \phi^n-\phi^{n+1}\right>_\Omega,\\
		& E_h^e(\phi^{n+1})-E_h^e(\phi^{n}) \geq \left< \delta_{\phi}E_h^e(\phi^{n}), \phi^{n+1}-\phi^{n}\right>_\Omega.
	\end{align*}
	Combining the above results with \eqref{numerical phi}--\eqref{numerical mu}, and using the summation-by-parts formulae (see Lemma \ref{lem:sumbyparts}),  we obtain the energy dissipative property \eqref{disdecay}:
	\begin{equation*}
		\begin{aligned}
			E_h(\phi^{n+1})-E_h(\phi^n)
			&= E_h^c(\phi^{n+1})-E_h^c(\phi^{n})-(E_h^e(\phi^{n+1})-E_h^e(\phi^{n})) \\
			&\leq\left<\delta_{\phi}E_h^c(\phi^{n+1})-\delta_{\phi}E_h^e(\phi^{n}),\phi^{n+1}-\phi^n\right>_{\Omega} \\
			&=\Delta t\left<\mu^{n+1},\Delta_h\mu^{n+1}\right>_{\Omega} \\
			&=-\Delta t\,\|\nabla_h\mu^{n+1}\|_2^2 \leq 0.
		\end{aligned}
	\end{equation*}
	The proof is complete.
\end{proof}

The energy dissipation property enables us to derive an uniform $H_h^2$-estimate for the solution to the discrete system \eqref{numerical phi}--\eqref{numerical mu}.
\begin{cor}[Uniform $H_h^2$-estimate] \label{es-HH2}
	For every positive integer $m$, the discrete solution $\phi^m$ to the system \eqref{numerical phi}--\eqref{numerical mu} satisfies the following estimate
	\begin{equation}
		\|\phi^m\|_{H^2_h} \leq C, \label{destH2}
	\end{equation}
	where the constant $C>0$ only depends on $\lambda$, $\eta$, $\epsilon$, $|\Omega|$ and $\phi^0$.
\end{cor}
\begin{proof}
	It follows from \eqref{disdecay} that $E_h(\phi^{m})\leq E_h(\phi^{0})$ for any $m\in \mathbb{Z}^+$.
	On the other hand, from the definition of $E_h$, \eqref{be1} and the $L^\infty_h$-bound \eqref{masssep} for $\phi^m$, we can deduce that
	\begin{align}
		E_h(\phi^m)
		&= \frac{\epsilon^4}{2}\|\Delta_h\phi^m\|_2^2 + \frac{1}{2}\|\beta(\phi^m)\|_2^2 + \left(\frac{\lambda^2}{2} +\frac{\lambda\epsilon^p\eta}{2}\right)\|\phi^m\|_2^2 - \left(\frac{\epsilon^{2+p}\eta}{2}+\lambda\epsilon^2\right)\|\nabla_h\phi^m\|_2^2\notag \\
		&\quad-\lambda\left<\phi^m,\beta(\phi^m)\right>_{\Omega} +\epsilon^2\left<\beta'(\phi^m),\sum_{\zeta=x,y,z}a_{\zeta}\left(\left|D_{\zeta}\phi^m\right|^2\right)\right>_{\Omega} -\epsilon^p\eta\left<B(\phi^m),1\right>_{\Omega} \notag \\			
		&\geq\frac{\epsilon^4}{2}\|\Delta_h\phi^m\|_2^2 +\frac{1}{4}\|\beta(\phi^m)\|_2^2 -C(1+\|\nabla_h\phi^m\|_2^2) \notag \\
		&\geq \frac{\epsilon^4}{2}\|\Delta_h\phi^m\|_2^2 +\frac{1}{4}\|\beta(\phi^m)\|_2^2 -C(1+\|\Delta_h\phi^m\|_2\|\phi^m\|_2)\notag \\
		&\geq\frac{\epsilon^4}{4}\|\Delta_h\phi^m\|_2^2 +\frac{1}{4}\|\beta(\phi^m)\|_2^2-C,
		\notag
	\end{align}
	where $C>0$ may depend on $\lambda$, $\eta$, $\epsilon$, $|\Omega|$, but is independent of $m$.
	Thus, as long as the initial discrete energy $E_h(\phi^0)$ is finite, it holds that $\|\Delta_h\phi^m\|_2\leq C$ for any $m\in\mathbb{Z}^+$. This estimate combined with the fact $\|\phi^m\|_\infty<1$ leads to the conclusion \eqref{destH2}. The proof is complete.
\end{proof}

\section{Optimal Rate Convergence Analysis and Error Estimate}	
\label{sec:convergence}
\setcounter{equation}{0}

In this section, we perform the convergence analysis and derive error estimates for the discrete approximation of solutions to the FCH equation (\ref{phi}).

For the continuous problem (\ref{phi}) in the periodic setting with $\Omega=(0,1)^d$, $1\leq d\leq 3$, existence and uniqueness of a global weak solution denoted by $\Phi$ has already been established in Proposition \ref{prop:well}, provided that its initial datum $\Phi_0$ satisfies $\Phi_0\in H^2_{\mathrm{per}}(\Omega)$, $\beta(\Phi_0)\in L^2_{\mathrm{per}}(\Omega)$ and $\int_\Omega \Phi_0 \mathrm{d}x\in (-1,1)$. Besides, in view of Proposition \ref{prop:sepa} and the discussions made in Remark \ref{sepreg}, hereafter we assume that the initial datum $\Phi_0$ is  smooth enough and strictly separated from $\pm 1$ so that the corresponding solution $\Phi$ to the continuous problem \eqref{phi} belongs to the regularity class $\mathcal{R}$ such that
\begin{equation}\label{regularity}
	\Phi\in \mathcal{R}:= H^4(0,T;H^2_{\mathrm{per}}(\Omega))\cap H^2(0,T;H_{\mathrm{per}}^8(\Omega))\cap C([0,T];H_{\mathrm{per}}^{12}(\Omega)),
\end{equation}
on $[0,T]$ for some $T>0$. Moreover, we assume that the following strict separation property holds:
\begin{equation}\label{strict separation true}
	-1+\theta_0	\leq \Phi(x,t)\leq 1-\theta_0,\quad \forall\, (x,t)\in \overline{\Omega}\times [0,T],
\end{equation}
where $\theta_0\in \left(0,1\right)$. The assumption \eqref{regularity} on the smoothness of the exact solution $\Phi$ may not be optimal, but they are sufficient for the subsequent convergence analysis.

Adapting the notations in \cite{chen2019,liu2021positivity}, we introduce the projection operator $\mathcal{P}_N: C_{\mathrm{per}}(\overline{\Omega})\to \mathcal{B}^{K}$, which is the space of trigonometric polynomials of degree less than or equal to $K=[N/2]$. Define
$$
\Phi_N(\cdot,t):=\mathcal{P}_N\Phi(\cdot,t)\in \mathcal{B}^{K}$$
the (spatial) Fourier projection of the exact solution $\Phi$. If $\Phi\in L^\infty(0,T;H_{\mathrm{per}}^l(\Omega))$ for some $l\in \mathbb{Z}^+$, the following estimate is well-known (see \cite{KO79}):
\begin{equation}\label{fourier projection}
	\|\Phi_N-\Phi\|_{L^\infty(0,T;H^{k}(\Omega))}\leq Ch^{l-k}\|\Phi\|_{L^\infty(0,T;H^{l}(\Omega))},\quad \forall\, 0\leq k\leq l.
\end{equation}
In particular, from \eqref{regularity}  and the Sobolev embedding $H_{\mathrm{per}}^2(\Omega)\hookrightarrow C_{\mathrm{per}}(\overline{\Omega})$ in three dimensions, we can take $k=2$, $l=3$ in (\ref{fourier projection}) and $h$ sufficiently small (i.e., $N$ is sufficiently large) so that the Fourier projection $\Phi_N$ still preserves the strict separation property (however, with a relaxed bound):
\begin{equation}\label{strict separation fourier}
	-1+\frac{\theta_0}{2}	\leq \Phi_N(x,t)\leq 1-\frac{\theta_0}{2},\quad \forall\, (x,t)\in \overline{\Omega}\times [0,T].
\end{equation}

Given an arbitrary positive integer $M$, we set the uniform time step as $\Delta t= T/M>0$. For every $m\in \mathbb{N}$ with $m\leq M$, we set
$$
\Phi^m=\Phi(\cdot,t_m)\quad \text{and}\quad \Phi_N^m=\Phi_N(\cdot,t_m),\quad \text{with}\ t_m = m \Delta t.
$$
Since the exact solution $\Phi$ preserves the mass (recall Proposition \ref{prop:well}) and $\Phi_N\in\mathcal{B}^K$, it is straightforward to check that for all $0\leq m\leq M-1$, the property of mass conservation also holds for the Fourier projection $\Phi_N$ at the discrete level of time:
\begin{equation}\label{projection estimate}
	\int_{\Omega}\Phi_N(\cdot,t_{m+1})\,\mathrm{d}x =
	\int_{\Omega}\Phi(\cdot,t_{m+1})\,\mathrm{d}x =
	\int_{\Omega}\Phi(\cdot,t_{m})\,\mathrm{d}x =\int_{\Omega}\Phi_N(\cdot,t_{m})\,\mathrm{d}x.
\end{equation}

Next, we define the canonical grid projection operator $\mathcal{P}_h: C_{\mathrm{per}}(\overline{\Omega}) \to \mathcal{C}_{\mathrm{per}}$. For any $g\in C_{\mathrm{per}}(\overline{\Omega})$, we set $\widetilde{g}=\mathcal{P}_h{g} \in \mathcal{C}_{\mathrm{per}}$ with grid values
$$
\widetilde{g}_{i,j,k}=g\Big(\big(i-\frac{1}{2}\big)h,\big(j-\frac{1}{2}\big)h, \big(k-\frac{1}{2}\big)h\Big),\quad  1\leq i,j,k\leq N.
$$
Then for the fully discrete system \eqref{numerical phi}--\eqref{numerical mu}, we take its initial datum as $\phi^0=\mathcal{P}_h\Phi_N^0=\mathcal{P}_h(\mathcal{P}_N\Phi_0)$, i.e.,
\begin{equation}
	\phi_{i,j,k}^0:=\Phi_N\Big(\big(i-\frac{1}{2}\big)h,\big(j-\frac{1}{2}\big)h,\big(k-\frac{1}{2}\big)h,0\Big),\quad 1\leq i,j,k\leq N.\label{discrete-ini}
\end{equation}

The error grid function that we are going to investigate is given  by
\begin{equation}
	e^m:=\mathcal{P}_h\Phi_N^m-\phi^m,\qquad\forall\, m\in\mathbb{N},\ m\leq M.
	\label{old-err}
\end{equation}
Keeping the definitions of $\mathcal{P}_N$ and $\mathcal{P}_h$ in mind, we infer from \eqref{projection estimate}, \eqref{discrete-ini} and Proposition \ref{prop:dmass} that
\begin{equation}
	\overline{\mathcal{P}_h\Phi_N^m}
	= \frac{1}{|\Omega|}\int_\Omega  \Phi_N^m\,\mathrm{d}x
	= \frac{1}{|\Omega|}\int_\Omega  \Phi_N^0\,\mathrm{d}x
	= \overline{\mathcal{P}_h\Phi_N^0}
	= \overline{\phi^0}=\overline{\phi^m}.
	\label{projection estimate-h}
\end{equation}
As a consequence, we have $$\overline{e^m}=0,\quad \forall\, m\in\mathbb{N},\ m\leq M.
$$

The main result of this section reads as follows

\begin{thm}[Error estimate]\label{thm1}
	Suppose that the exact solution $\Phi$ of the FCH equation (\ref{phi}) belongs to the regularity class $\mathcal{R}$ (see \eqref{regularity}) and satisfies the strict separation property \eqref{strict separation true}. Then, provided that $\Delta t$ and $h$ are sufficiently small, and under the linear refinement requirement $\Delta t\leq C_1 h$ for some fixed $C_1>0$, we have
	\begin{equation}
		\|e^{m}\|_2+\left[\epsilon^4\Delta t\sum_{k=1}^{m}\Big(\|\nabla_h\Delta_h e^{k}\|_2^2 + \|\Delta_h e^{k}\|_2^2\Big) \right]^{\frac12} \leq C(\Delta t + h^2),
		\label{L2-est}
	\end{equation}
	for any $m\in \mathbb{Z}^+$ such that $t_m=m\Delta t\leq T$, where $e^m$ is defined as in \eqref{old-err} and the  constant $C>0$ is independent of $m$, $\Delta t$ and $h$.
\end{thm}
\begin{rem}
The linear refinement requirement $\Delta t\leq C_1 h$ on the size of time step is a technical assumption due to the usage of inverse inequalities between norms (see \eqref{priori assumptionA} and \eqref{priori assumptionB} below).
\end{rem}

\subsection{Higher-order consistency analysis}

First, applying the temporal discretization, we see that the Fourier projection $\Phi_N$ solves the following semi-implicit equation:
\begin{align}
	\frac{\Phi_N^{n+1}-\Phi_N^n}{\Delta t} &= \Delta\Big[\epsilon^4\Delta^2\Phi_N^{n+1} + \beta(\Phi_N^{n+1})\beta'(\Phi_N^{n+1}) + \epsilon^2 \beta''(\Phi_N^{n+1})|\nabla\Phi_N^{n+1}|^2 \notag \\
	&\quad -2\epsilon^2\nabla\cdot\left(\beta'(\Phi_N^{n+1})\nabla\Phi_N^{n+1}\right) + \lambda(\lambda+\epsilon^p\eta) \Phi_N^{n+1}
	- \lambda\Phi_N^n\beta'(\Phi_N^{n}) \notag \\
	&\quad +\epsilon^2(2\lambda+\epsilon^p\eta)\Delta\Phi_N^n
	-(\lambda+\epsilon^p\eta)\beta(\Phi_N^{n})\Big] +\widetilde{F}_N^{n+1},
	\label{trucation error}
\end{align}
where the truncation error is given by
\begin{align*}
	\widetilde{F}_N^{n+1} &= \left(\frac{\Phi_N^{n+1}-\Phi_N^n}{\Delta t}- \partial_t \Phi\right) + \epsilon^4\Delta^3(\Phi-\Phi_N^{n+1})
	+\Delta\big[ \beta(\Phi)\beta'(\Phi)- \beta(\Phi_N^{n+1})\beta'(\Phi_N^{n+1})\big]\\
	&\quad  + \epsilon^2
	\Delta\big[ \beta''(\Phi)|\nabla \Phi|^2- \beta''(\Phi_N^{n+1})|\nabla\Phi_N^{n+1}|^2\big]
	-2\epsilon^2 \Delta^2 \big[\beta(\Phi)- \beta(\Phi_N^{n+1})\big]\\
	& \quad  + \lambda(\lambda+\epsilon^p\eta)\Delta (\Phi-\Phi_N^{n+1})
	-\lambda \Delta \big[\Phi\beta'(\Phi)- \Phi_N^n\beta'(\Phi_N^{n})\big]\\
	&\quad + \epsilon^2(2\lambda+\epsilon^p\eta)\Delta^2(\Phi-\Phi_N^n)
	- (\lambda+\epsilon^p\eta) \Delta \big[\beta(\Phi)-\beta(\Phi_N^{n})\big].
\end{align*}
Recalling the assumptions \eqref{regularity}--\eqref{strict separation true}, then using Taylor's expansion in time and the error estimate \eqref{fourier projection} (e.g., with $l=11$, $k=8$), we have
$$
\widetilde{F}_N^{n+1}=\Delta tF_{N,1}^{n+1} +O(\Delta t^2) + O(h^{3}),
$$
where the spatial function $F_{N,1}^{n+1}$ only depends on $\Phi_N$, $\Phi$, and it is sufficiently smooth in the sense that its derivatives are bounded. From the mass conservation \eqref{projection estimate} and the periodic boundary condition, we infer that  $\int_\Omega \widetilde{F}_N^{n+1}\,\mathrm{d}x=0$. This fact combined with Taylor's expansion further implies $\int_\Omega F_{N,1}^{n+1}\,\mathrm{d}x=0$.

Performing a further spatial discretization to \eqref{trucation error} and applying a careful consistency analysis via Taylor's expansion, we can actually show that $\mathcal{P}_h\Phi_N^{n+1}$ solves the discrete system \eqref{numerical phi}--\eqref{numerical mu} with a first order accuracy in time and a second order accuracy in space. However, a restrictive refinement condition like $C'h^2\leq\Delta t\leq C''h^2$ for some $C''>C'>0$  is needed to complete the proof. The detailed calculations are left to the interested readers.

In order to prove Theorem \ref{thm1} under a weaker (i.e., linear) refinement condition on the time step size $\Delta t$, we perform a higher-order consistency analysis for the auxiliary profile
\begin{equation}\label{pertubation expansion}
	\widehat{\Phi} = \Phi_N + \Delta t\mathcal{P}_N\Phi_{\Delta t,1}+\Delta t^2\mathcal{P}_N\Phi_{\Delta t,2}+h^2\mathcal{P}_N\Phi_{h,1}.
\end{equation}
Here we note that the centered difference used in the spatial discretization gives local truncation errors with only even order terms. In this manner, we find that a higher order consistency of $O(\Delta t^3+ h^3)$ holds for the numerical scheme \eqref{numerical phi}--\eqref{numerical mu} with $\mathcal{P}_h\widehat{\Phi}^{n+1}$, see \eqref{correction trucation error}--\eqref{tau-n1} below. The approximate expansion \eqref{pertubation expansion} enables us to derive a higher order convergence estimate of the modified numerical error function $\widehat{e}^n$ (see \eqref{new-err}) between the constructed profile $\mathcal{P}_h \widehat{\Phi}^n$ and the numerical solution $\phi^n$ in the discrete $\|\cdot\|_{-1,h}$ norm (see \eqref{convergence analysis}). This combined with the inverse inequality eventually yields a refined $\|\cdot\|_\infty$ bound of the numerical solution $\phi^n$, more precisely, the property of strict separation from pure states $\pm1$ (see \eqref{strict separation numerical}) that is uniform with respect to the time step and the grid size (cf. Remark \ref{se-bound}).

The supplementary field $\Phi_{\Delta t,1}$ in the profile \eqref{pertubation expansion} only depends on the exact solution $\Phi$ and can be constructed via a perturbation expansion argument (see e.g., \cite{liu2021positivity,wang2015energy}). Indeed, we find the temporal correction $\Phi_{\Delta t,1}$ by solving the following linear equation:
\begin{align}
	\partial_t\Phi_{\Delta t,1} &= \Delta\bigg[ \epsilon^4\Delta^2\Phi_{\Delta t,1} + \left(\beta'(\Phi_{N})^2+\beta(\Phi_N)\beta''(\Phi_N)\right)\Phi_{\Delta t,1}	\notag	\\
	&\qquad + \epsilon^2\left(2\beta''(\Phi_N)\nabla\Phi_N\cdot\nabla\Phi_{\Delta t,1}+\beta'''(\Phi_N)|\nabla\Phi_N|^2\Phi_{\Delta t,1}\right) \notag	\\
	&\qquad -2\epsilon^2\nabla\cdot\left(\beta'(\Phi_N)\nabla\Phi_{\Delta t,1}+\beta''(\Phi_N)\Phi_{\Delta t,1}\nabla\Phi_N\right)	\notag		\\
	&\qquad + \lambda(\lambda+\epsilon^p\eta)\Phi_{\Delta t,1} - \lambda\left(\Phi_N\beta''(\Phi_N)+\beta'(\Phi_N)\right)\Phi_{\Delta t,1} \notag	\\
	&\qquad + \epsilon^2(2\lambda+\epsilon^p\eta)\Delta\Phi_{\Delta t,1} - (\lambda+\epsilon^p\eta)\beta'(\Phi_N)\Phi_{\Delta t,1}\bigg]-F_{N,1}, \label{trucation error-dt1}
\end{align}
subject to the zero initial condition $\Phi_{\Delta t,1}(\cdot,0)=0$ and periodic boundary conditions.
Once the Fourier projection  $\Phi_N$ is given, we take the external source term $F_{N,1}$ in \eqref{trucation error-dt1} to be a smooth function satisfying
$$F_{N,1}(\cdot,t_n)=F_{N,1}^n\quad \text{and}\quad \int_\Omega F_{N,1}(\cdot, t)\,\mathrm{d}x=0,\quad \forall\, t\in [0,T].$$
The existence and uniqueness of a solution to the linear sixth order parabolic equation \eqref{trucation error-dt1} on $[0,T]$ can be easily proved, for instance, by using the Faedo-Galerkin method. In particular, the solution $\Phi_{\Delta t,1} $ only depends on $\Phi_N$, and its  derivatives in various orders are bounded (recalling that $\Phi_N$ is sufficiently smooth). Moreover, it satisfies the zero mass property:
$$
\int_\Omega \Phi_{\Delta t,1}(\cdot,t)\,\mathrm{d}x=\int_\Omega \Phi_{\Delta t,1}(\cdot,0)\,\mathrm{d}x=0,\quad \forall\, t\in[0,T].
$$
Next, an application of the semi-implicit discretization  for \eqref{trucation error-dt1}  yields
\begin{align}
	\frac{\Phi_{\Delta t,1}^{n+1}-\Phi_{\Delta t,1}^n}{\Delta t} &= \Delta\bigg[ \epsilon^4\Delta^2\Phi_{\Delta t,1}^{n+1} + \left(\beta'(\Phi_N^{n+1})^2+\beta(\Phi_N^{n+1})\beta''(\Phi_N^{n+1})\right)\Phi_{\Delta t,1}^{n+1}	\notag	\\
	&\qquad + \epsilon^2\left(2\beta''(\Phi_N^{n+1})\nabla\Phi_N^{n+1}\cdot\nabla\Phi_{\Delta t,1}^{n+1}+\beta'''(\Phi_N^{n+1})|\nabla\Phi_N^{n+1}|^2\Phi_{\Delta t,1}^{n+1}\right)	\notag\\
	&\qquad -2\epsilon^2\nabla\cdot\left(\beta'(\Phi_N^{n+1})\nabla\Phi_{\Delta t,1}^{n+1}+\beta''(\Phi_N^{n+1})\Phi_{\Delta t,1}^{n+1}\nabla\Phi_N^{n+1}\right)	\notag		\\
	&\qquad + \lambda(\lambda+\epsilon^p\eta)\Phi_{\Delta t,1}^{n+1} - \lambda\left(\Phi_N^n\beta''(\Phi_N^n)+\beta'(\Phi_N^n)\right)\Phi_{\Delta t,1}^{n} \notag	\\
	&\qquad + \epsilon^2(2\lambda+\epsilon^p\eta)\Delta\Phi_{\Delta t,1}^{n} - (\lambda+\epsilon^p\eta)\beta'(\Phi_N^n)\Phi_{\Delta t,1}^{n}\bigg]-F_{N,1}^{n+1} +O(\Delta t), \label{sympototic expansion}
\end{align}
where we denote by $\Phi_{\Delta t,1}^m=\Phi_{\Delta t,1}(\cdot, t_m)$, $t_m=m\Delta t$, $m\in \mathbb{N}$, $m\leq M$. At the discrete level of time, we note that for every $n\in \mathbb{N}$, it also holds $\int_{\Omega}\Phi_{\Delta t,1}^{n+1}\,\mathrm{d}x =\int_{\Omega}\Phi_{\Delta t,1}^n\,\mathrm{d}x$. This fact further implies
\begin{equation}\label{projection estimate-b}
	\int_{\Omega} \mathcal{P}_N \Phi_{\Delta t,1}^{m}\,\mathrm{d}x =0, \quad \forall\, m\in \mathbb{N}, \ m\leq M.
\end{equation}			
Combining \eqref{trucation error} and \eqref{sympototic expansion}, we can obtain the following second order in temporal space truncation error for $$\widehat{\Phi}_1:=\Phi_N + \Delta t\mathcal{P}_N\Phi_{\Delta t,1}$$ 
such that
\begin{align}\label{hat-Phi-1}
	\frac{\widehat{\Phi}_1^{n+1}-\widehat{\Phi}_1^n}{\Delta t}=\Delta&\Big[\epsilon^4 \Delta^2\widehat{\Phi}_1^{n+1} + \beta(\widehat{\Phi}_1^{n+1})\beta'(\widehat{\Phi}_1^{n+1}) +\epsilon^2\beta'' (\widehat{\Phi}_1^{n+1})|\nabla\widehat{\Phi}_1^{n+1}|^2 \notag \\
	&-2\epsilon^2\nabla\cdot \left(\beta'(\widehat{\Phi}_1^{n+1})\nabla\widehat{\Phi}_1^{n+1}\right) +\lambda(\lambda+\epsilon^p\eta)\widehat{\Phi}^{n+1} -\lambda\widehat{\Phi}_1^n\beta'(\widehat{\Phi}_1^{n}) \notag \\
	&+\epsilon^2(2\lambda+\epsilon^p\eta)\Delta\widehat{\Phi}_1^n -(\lambda+\epsilon^p\eta)\beta(\widehat{\Phi}^{n})\Big] + \Delta t^2F_{N,2}^{n+1} + O(\Delta t^3) + O(h^{3}).
\end{align}

  Following arguments similar to those in \cite{liu2021positivity,li2021convergence,chen2022error}, we can further derive the second order temporal correction $\Phi_{\Delta t,2}$ as well as the spatial correction $\Phi_{h,1}$, respectively. 
  
  To this end, the next order temporal correction function $\Phi_{\Delta t,2}$ is given by the linear equation
\begin{align}
	\partial_t\Phi_{\Delta t,2} &= \Delta\bigg[ \epsilon^4\Delta^2\Phi_{\Delta t,2} + \left(\beta'(\widehat{\Phi}_1)^2+\beta(\widehat{\Phi}_1)\beta''(\widehat{\Phi}_1)\right)\Phi_{\Delta t,2}	\notag	\\
	&\qquad + \epsilon^2\left(2\beta''(\widehat{\Phi}_1)\nabla\widehat{\Phi}_1\cdot\nabla\Phi_{\Delta t,2}+\beta'''(\widehat{\Phi}_1)|\nabla\widehat{\Phi}_1|^2\Phi_{\Delta t,2}\right) \notag	\\
	&\qquad -2\epsilon^2\nabla\cdot\left(\beta'(\widehat{\Phi}_1)\nabla\Phi_{\Delta t,2}+\beta''(\widehat{\Phi}_1)\Phi_{\Delta t,2}\nabla\widehat{\Phi}_1\right)	\notag		\\
	&\qquad + \lambda(\lambda+\epsilon^p\eta)\Phi_{\Delta t,2} - \lambda\left(\widehat{\Phi}_1\beta''(\widehat{\Phi}_1)+\beta'(\widehat{\Phi}_1)\right)\Phi_{\Delta t,2} \notag	\\
	&\qquad + \epsilon^2(2\lambda+\epsilon^p\eta)\Delta\Phi_{\Delta t,2} - (\lambda+\epsilon^p\eta)\beta'(\widehat{\Phi}_1)\Phi_{\Delta t,2}\bigg]-F_{N,2}, \label{trucation error-dt2}
\end{align}
subject to the zero initial condition $\Phi_{\Delta t,2}(\cdot,0)=0$ and periodic boundary conditions. 
Like before, the solution $\Phi_{\Delta t,2}$ depends only on the exact solution $\Phi$ and its derivatives of various orders are bounded. An application of the semi-implicit discretization to \eqref{trucation error-dt2} implies that
	\begin{align}
		\frac{\Phi_{\Delta t,2}^{n+1}-\Phi_{\Delta t,2}^n}{\Delta t} &= \Delta\bigg[ \epsilon^4\Delta^2\Phi_{\Delta t,2}^{n+1} + \left(\beta'(\widehat{\Phi}_1^{n+1})^2+\beta(\widehat{\Phi}_1^{n+1})\beta''(\widehat{\Phi}_1^{n+1})\right)\Phi_{\Delta t,2}^{n+1}	\notag	\\
		&\qquad + \epsilon^2\left(2\beta''(\widehat{\Phi}_1^{n+1})\nabla\widehat{\Phi}_1^{n+1}\cdot\nabla\Phi_{\Delta t,2}^{n+1}+\beta'''(\widehat{\Phi}_1^{n+1})|\nabla\widehat{\Phi}_1^{n+1}|^2\Phi_{\Delta t,2}^{n+1}\right)	\notag\\
		&\qquad -2\epsilon^2\nabla\cdot\left(\beta'(\widehat{\Phi}_1^{n+1})\nabla\Phi_{\Delta t,2}^{n+1}+\beta''(\widehat{\Phi}_1^{n+1})\Phi_{\Delta t,2}^{n+1}\nabla\widehat{\Phi}_1^{n+1}\right)	\notag		\\
		&\qquad + \lambda(\lambda+\epsilon^p\eta)\Phi_{\Delta t,2}^{n+1} - \lambda\left(\widehat{\Phi}_1^n\beta''(\widehat{\Phi}_1^n)+\beta'(\widehat{\Phi}_1^n)\right)\Phi_{\Delta t,2}^{n} \notag	\\
		&\qquad + \epsilon^2(2\lambda+\epsilon^p\eta)\Delta\Phi_{\Delta t,2}^{n} - (\lambda+\epsilon^p\eta)\beta'(\widehat{\Phi}_1^n)\Phi_{\Delta t,2}^{n}\bigg]-F_{N,2}^{n+1} +O(\Delta t). \label{sympototic expansion dt2}
	\end{align}
Then a combination of \eqref{trucation error-dt2} and \eqref{sympototic expansion dt2} yields the third order temporal truncation error for the profile 
$$\widehat{\Phi}_2:=\Phi_N + \Delta t\mathcal{P}_N\Phi_{\Delta t,1} + \Delta t^2\mathcal{P}_N\Phi_{\Delta t,2}$$
such that 
		\begin{align}
			\frac{\widehat{\Phi}_2^{n+1}-\widehat{\Phi}_2^n}{\Delta t}=\Delta&\Big[\epsilon^4 \Delta^2\widehat{\Phi}_2^{n+1} + \beta(\widehat{\Phi}_2^{n+1})\beta'(\widehat{\Phi}_2^{n+1}) +\epsilon^2\beta'' (\widehat{\Phi}_2^{n+1})|\nabla\widehat{\Phi}_2^{n+1}|^2 \notag \\
			&-2\epsilon^2\nabla\cdot \left(\beta'(\widehat{\Phi}_2^{n+1})\nabla\widehat{\Phi}_2^{n+1}\right) +\lambda(\lambda+\epsilon^p\eta)\widehat{\Phi}_2^{n+1} -\lambda\widehat{\Phi}_2^n\beta'(\widehat{\Phi}_2^{n}) \notag \\
			&+\epsilon^2(2\lambda+\epsilon^p\eta)\Delta\widehat{\Phi}_2^n -(\lambda+\epsilon^p\eta)\beta(\widehat{\Phi}_2^{n})\Big] + O(\Delta t^3)+O(h^{3}).\label{correction trucation error-b2}
		\end{align}
  
Finally, we construct the spatial correction term $\Phi_{h,1}$ to upgrade the spatial accuracy order.
The following truncation error analysis for the spatial discretization can be obtained by using a straightforward Taylor expansion for $\widehat{\Phi}_2$:
\begin{align}
	\frac{\widehat{\Phi}_2^{n+1}-\widehat{\Phi}_2^n}{\Delta t}=\Delta_h&\Big[\epsilon^4\Delta_h^2\widehat{\Phi}_2^{n+1} + \beta(\widehat{\Phi}_2^{n+1})\beta'(\widehat{\Phi}_2^{n+1})+\epsilon^2\sum_{\zeta=x,y,z}\beta'' (\widehat{\Phi}_2^{n+1})a_{\zeta} \left(|D_{\zeta}\widehat{\Phi}_2^{n+1}|^2\right) \notag \\
	&-2\epsilon^2\sum_{\zeta=x,y,z}d_{\lambda} \left(A_{\zeta}\big(\beta'(\widehat{\Phi}_2^{n+1})\big) D_{\zeta}\widehat{\Phi}_2^{n+1}\right) +\lambda(\lambda+\epsilon^p\eta)\widehat{\Phi}_2^{n+1} \notag \\
	&-\lambda\widehat{\Phi}_2^n\beta'(\widehat{\Phi}_2^{n}) + \epsilon^2(2\lambda+\epsilon^p\eta)\Delta_h\widehat{\Phi}_2^n -(\lambda+\epsilon^p\eta)\beta(\widehat{\Phi}_2^{n})\Big] \notag\\
&+ h^2H_{N,1}^{n+1}+O(\Delta t^3)+ O(h^3).
\label{trucation error full dis}
\end{align}
Subsequently, the spatial correction function $\Phi_{h,1}$ is given by solving the linear equation
\begin{align}
	\partial_t\Phi_{h,1} &= \Delta\bigg[ \epsilon^4\Delta^2\Phi_{h,1} + \left(\beta'(\widehat{\Phi}_2)^2+\beta(\widehat{\Phi}_2)\beta''(\widehat{\Phi}_2)\right)\Phi_{h,1}	\notag	\\
	&\qquad + \epsilon^2\left(2\beta''(\widehat{\Phi}_2)\nabla\widehat{\Phi}_2\cdot\nabla\Phi_{h,1}+\beta'''(\widehat{\Phi}_2)|\nabla\widehat{\Phi}_2|^2\Phi_{h,1}\right) \notag	\\
	&\qquad -2\epsilon^2\nabla\cdot\left(\beta'(\widehat{\Phi}_2)\nabla\Phi_{h,1}+\beta''(\widehat{\Phi}_2)\Phi_{h,1}\nabla\widehat{\Phi}_2\right)	\notag		\\
	&\qquad + \lambda(\lambda+\epsilon^p\eta)\Phi_{h,1} - \lambda\left(\widehat{\Phi}_2\beta''(\widehat{\Phi}_2)+\beta'(\widehat{\Phi}_2)\right)\Phi_{h,1} \notag	\\
	&\qquad + \epsilon^2(2\lambda+\epsilon^p\eta)\Delta\Phi_{h,1} - (\lambda+\epsilon^p\eta)\beta'(\widehat{\Phi}_2)\Phi_{h,1}\bigg]-H_{N,1}, \label{trucation error-h2}
\end{align}
subject to the zero initial condition $\Phi_{h,1}(\cdot,0)=0$ and periodic boundary conditions. 
Again, the solution $\Phi_{h,1}$ depends only on the (sufficiently) exact solution $\Phi$, with bounded differences of various orders. 

Combining the above arguments, we have thus constructed the required auxiliary profile
 $\widehat{\Phi}$ (recall \eqref{pertubation expansion}).  
  
\begin{rem} Since the temporal/spatial correction functions $\Phi_{\Delta t,1}$, $\Phi_{\Delta t,2}$, $\Phi_{h,1}$ are sufficiently smooth and bounded, applying the projection estimate \eqref{fourier projection} and the separation property \eqref{strict separation fourier} for $\Phi_N$, we can obtain a strict separation property for $\widehat{\Phi}$ such that
	\begin{equation}\label{modified separation property}
		-1+\theta \leq\widehat{\Phi}(x,t)\leq 1-\theta, \quad \forall\, (x,t)\in \overline{\Omega}\times [0,T],
	\end{equation}
	with $\theta=\theta_0/4$, provided that the time step $\Delta t$ and the grid size $h$ are sufficiently small.
Such a uniform separation property will be useful in the subsequent convergence analysis.
\end{rem}

An application of a full discretization to \eqref{trucation error-h2} yields that 
\begin{align}
	\frac{\Phi_{h,1 }^{n+1}-\Phi_{h,1 }^n}{\Delta t} &= \Delta_h\bigg[ \epsilon^4\Delta_h^2\Phi_{h,1 }^{n+1} + \left(\beta'(\widehat{\Phi}_2^{n+1})^2+\beta(\widehat{\Phi}_2^{n+1})\beta''(\widehat{\Phi}_2^{n+1})\right)\Phi_{h,1 }^{n+1}	\notag	\\
	&\qquad + \epsilon^2\left(2\beta''(\widehat{\Phi}_2^{n+1})\nabla_h\widehat{\Phi}_2^{n+1}\cdot\nabla_h\Phi_{h,1 }^{n+1}+\beta'''(\widehat{\Phi}_2^{n+1})|\nabla_h\widehat{\Phi}_2^{n+1}|^2\Phi_{h,1 }^{n+1}\right)	\notag\\
	&\qquad -2\epsilon^2\nabla_h\cdot\left(\beta'(\widehat{\Phi}_2^{n+1})\nabla_h\Phi_{h,1 }^{n+1}+\beta''(\widehat{\Phi}_2^{n+1})\Phi_{h,1 }^{n+1}\nabla_h\widehat{\Phi}_2^{n+1}\right)	\notag		\\
	&\qquad + \lambda(\lambda+\epsilon^p\eta)\Phi_{h,1 }^{n+1} - \lambda\left(\widehat{\Phi}_2^n\beta''(\widehat{\Phi}_2^n)+\beta'(\widehat{\Phi}_2^n)\right)\Phi_{h,1 }^{n} \notag	\\
	&\qquad + \epsilon^2(2\lambda+\epsilon^p\eta)\Delta_h\Phi_{h,1 }^{n} - (\lambda+\epsilon^p\eta)\beta'(\widehat{\Phi}_2^n)\Phi_{h,1 }^{n}\bigg]-H_{N,1}^{n+1} +O(\Delta t+h^2). \label{sympototic expansion h2}
\end{align}  
 Hence, for the profile $\widehat{\Phi}=\Phi_N + \Delta t\mathcal{P}_N\Phi_{\Delta t,1}+\Delta t^2\mathcal{P}_N\Phi_{\Delta t,2}+h^2\mathcal{P}_N\Phi_{h,1}$, by a careful consistency analysis via Taylor's expansion, we can eventually arrive at the following fully discrete scheme for its discrete counterpart $\mathcal{P}_h\widehat{\Phi}^{n+1}$  such that (for the sake of simplicity, below we simply drop the projection operator $\mathcal{P}_h$ in $\mathcal{P}_h\widehat{\Phi}^{n}$ and $\mathcal{P}_h\widehat{\Phi}^{n+1}$)
\begin{align}
	\frac{\widehat{\Phi}^{n+1}-\widehat{\Phi}^n}{\Delta t}=\Delta_h&\Big[\epsilon^4\Delta_h^2\widehat{\Phi}^{n+1} + \beta(\widehat{\Phi}^{n+1})\beta'(\widehat{\Phi}^{n+1})+\epsilon^2\sum_{\zeta=x,y,z}\beta'' (\widehat{\Phi}^{n+1})a_{\zeta} \left(|D_{\zeta}\widehat{\Phi}^{n+1}|^2\right) \notag \\
	&-2\epsilon^2\sum_{\zeta=x,y,z}d_{\lambda} \left(A_{\zeta}\big(\beta'(\widehat{\Phi}^{n+1})\big) D_{\zeta}\widehat{\Phi}^{n+1}\right) +\lambda(\lambda+\epsilon^p\eta)\widehat{\Phi}^{n+1} \notag \\
	&-\lambda\widehat{\Phi}^n\beta'(\widehat{\Phi}^{n}) + \epsilon^2(2\lambda+\epsilon^p\eta)\Delta_h\widehat{\Phi}^n -(\lambda+\epsilon^p\eta)\beta(\widehat{\Phi}^{n})\Big] + \widehat{\tau}^{n+1},
\label{correction trucation error}
\end{align}
with the higher order
truncation error
\begin{equation}
	\label{tau-n1}
	\|\widehat{\tau}^{n+1}\|_2\leq C(\Delta t^3+h^3).
\end{equation}

\begin{rem}\label{rem1}
	We note that the  profile $\widehat{\Phi}=\Phi_N + \Delta t\mathcal{P}_N\Phi_{\Delta t,1}+\Delta t^2\mathcal{P}_N\Phi_{\Delta t,2}+h^2\mathcal{P}_N\Phi_{h,1} \in \mathcal{B}^K$ still satisfies the mass conservation property at the discrete level of both time and space. Indeed, from its definition, we observe that
	$\mathcal{P}_h\widehat{\Phi}^0 = \mathcal{P}_h \Phi_N^0=  \phi^0$. Then from the mass conservation of $\Phi_N^m$, $\mathcal{P}_N\Phi_{\Delta t,1}^m$, $\mathcal{P}_N\Phi_{\Delta t,2}^m$ and $\mathcal{P}_N\Phi_{h,1}^m$ (see \eqref{projection estimate-h}, \eqref{projection estimate-b}), we find that for every $m\in \mathbb{Z}^+$, $m\leq M$, it holds
	\begin{equation}\label{mass correction2}
		\overline{\mathcal{P}_h\widehat{\Phi}^m}
		= \overline{\mathcal{P}_h\Phi_N^m} + \Delta t \overline{\mathcal{P}_h(\mathcal{P}_N\Phi_{\Delta t,1}^m)} + \Delta t^2 \overline{\mathcal{P}_h(\mathcal{P}_N\Phi_{\Delta t,2}^m)} + h^2 \overline{\mathcal{P}_h(\mathcal{P}_N\Phi_{h,1}^m)}=
		\overline{\mathcal{P}_h\Phi_N^0}
		=\overline{\phi^0}=\overline{\phi^m}.
	\end{equation}
	Besides, we infer from \eqref{correction trucation error} and \eqref{mass correction2}  that  the local truncation error $\widehat{\tau}^n$ also fulfills
	\begin{equation}
		\overline{\widehat{\tau}^m}=0,\quad\forall\, m\in \mathbb{Z}^+,\ m\leq M.
		\label{tau-n2}
	\end{equation}
\end{rem}
%
%

%

\subsection{Error estimate in $l^{\infty}(0,T;H^{-1}_h)\cap l^2(0,T;H^2_h)$}

We first derive a preliminary error estimate on the error grid function $e^m$ in the lower order space $l^{\infty}(0,T;H^{-1}_h)\cap l^2(0,T;H^2_h)$. Since $\overline{e^m}=0$ for any $m\in\mathbb{N}$, $m\leq M$ as shown before, the discrete norm $\|\cdot\|_{-1,h}$ is well defined for $e^m$.
\begin{prop}\label{loworder}
	Suppose that the exact solution $\Phi$ of the FCH equation (\ref{phi}) belongs to the regularity class $\mathcal{R}$ (recall \eqref{regularity}) and satisfies the strict separation property \eqref{strict separation true}. Then, provided that $h$ is sufficiently small, and under the linear refinement requirement $\Delta t\leq C_1 h$ for some fixed $C_1>0$, we have
	\begin{equation}
		\|e^{m}\|_{-1,h}+\left[\Delta t\sum_{k=1}^{m}\Big(\epsilon^4\|\Delta_he^{k}\|_2^2+2\lambda(\lambda+\epsilon^p\eta)\|e^k\|_2^2\Big)\right]^{\frac12}\leq C(\Delta t + h^2),
		\label{des-h-1}
	\end{equation}
	for any $m\in\mathbb{Z}^+$ such that $t_m=m\Delta t\leq T$, where the  constant $C>0$ is independent of $m$, $\Delta t$ and $h$.
\end{prop}

\begin{proof}
	With the aid of the auxiliary profile $\widehat{\Phi}$ constructed above, we introduce an alternative numerical error function
	\begin{equation}
		\widehat{e}^m:= \mathcal{P}_h\widehat{\Phi}^m - \phi^m,\qquad \forall\, m\in\mathbb{N},\ m\leq M,
		\label{new-err}
	\end{equation}
where $\phi^m$ is the discrete solution to the system \eqref{numerical phi}--\eqref{numerical mu}.
	From Remark \ref{rem1} we see that $\overline{\widehat{e}^m}=0$. Hence, the discrete norm $\|\cdot\|_{-1,h}$ is also well defined for the new error grid function $\widehat{e}^m$ for all $m\in\mathbb{N}$, $m\leq M$.
	
	To proceed, we make the following \emph{a priori} assumption at the
	previous time step:
	\begin{equation}
		\|\widehat{e}^{n}\|_2 \leq \Delta t^\frac{5}{3}+ h^\frac{5}{3}.
\label{apri-es}
	\end{equation}
	The \emph{a priori} assumption \eqref{apri-es} will later be recovered by the optimal rate convergence analysis at the
	next time step. From \eqref{apri-es}, the linear refinement constraint $\Delta t \leq C_1h$, and  the inverse inequality, we can derive a discrete $\|\cdot\|_{\infty}$ bound for the numerical error function:
	\begin{equation}\label{priori assumptionA}
		 \|\widehat{e}^{n}\|_{\infty} \leq\frac{C\|\widehat{e}^{n}\|_2}{h^{\frac{3}{2}}} \leq
		\widetilde{C}h^\frac{1}{6} \leq\frac{\theta}{2},
	\end{equation}
	provided that $h>0$ is sufficiently small (we note that $\widetilde{C}>0$ is independent of $\Delta t$, $h$ and $n$). Here, the constant $\theta\in (0,1)$ is determined as in  \eqref{modified separation property}, which is independent of $n$. From \eqref{modified separation property} and \eqref{priori assumptionA}, we can  deduce that the numerical solution $\phi^n$ also satisfies the strict separation property at the previous time step:
	\begin{equation}\label{strict separation numerical}
		-1+\frac{\theta}{2}\leq
		\big(\mathcal{P}_h\widehat{\Phi}^n\big)_{i,j,k} -\|\widehat{e}^{n}\|_{\infty} \leq \phi^n_{i,j,k}\leq \big(\mathcal{P}_h\widehat{\Phi}^n\big)_{i,j,k}+\|\widehat{e}^{n}\|_{\infty} \leq 1-\frac{\theta}{2},\quad 1\leq i,j,k\leq N.
	\end{equation}
	
	For any $m\in \mathbb{N}$, $m\leq M$, since $\overline{\widehat{e}^m}=0$, we find that $(-\Delta_h)^{-1}\widehat{e}^m$ is well defined. This allows us to subtract the numerical scheme \eqref{numerical phi}--\eqref{numerical mu} from \eqref{correction trucation error} and then take (discrete) inner product between the resultant and $2(-\Delta_h)^{-1}\widehat{e}^{n+1}$. More precisely, we get
	\begin{align}\label{sec5est1}
		& \frac{1}{\Delta t}\big(\|\widehat{e}^{n+1}\|_{-1,h}^2 -\|\widehat{e}^{n}\|_{-1,h}^2 +\|\widehat{e}^{n+1}-\widehat{e}^{n}\|_{-1,h}^2\big) + 2\epsilon^4\|\Delta_h\widehat{e}^{n+1}\|_2^2\notag\\
&\quad +2\lambda(\lambda+\epsilon^p\eta)\|\widehat{e}^{n+1}\|_2^2 = \sum_{i=1}^5 I_i,
	\end{align}
	where
	\begin{equation}
		\begin{aligned}
			&I_1 = -2\left<\beta(\widehat{\Phi}^{n+1})\beta'(\widehat{\Phi}^{n+1}) -\beta(\phi^{n+1})\beta'(\phi^{n+1}),\,\widehat{e}^{n+1}\right>_{\Omega},\\
			&I_2 = -2\epsilon^2\left<\sum_{\zeta=x,y,z}\left(\beta''(\widehat{\Phi}^{n+1}) a_{\zeta}\big(|D_{\zeta}\widehat{\Phi}^{n+1}|^2\big) -\beta''(\phi^{n+1})a_{\zeta}\big(|D_{\zeta}\phi^{n+1}|^2\big)\right)\right.\\
			&\qquad\qquad\qquad\left. -2\sum_{\zeta=x,y,z}\left[d_{\zeta}\left(A_{\zeta} \big(\beta'(\widehat{\Phi}^{n+1})\big) D_{\zeta}\widehat{\Phi}^{n+1}-A_{\zeta} \big(\beta'(\phi^{n+1})\big)D_{\zeta}\phi^{n+1}\right)\right],\, \widehat{e}^{n+1}\right>_{\Omega},\\
			&I_3 = 2 \left<\lambda\widehat{\Phi}^n \beta'(\widehat{\Phi}^{n}) -\lambda\phi^n\beta'(\phi^{n}) +(\lambda+\epsilon^p\eta)\beta(\widehat{\Phi}^{n}) -(\lambda+\epsilon^p\eta)\beta(\phi^{n}),\,\widehat{e}^{n+1}\right>_{\Omega},\\
			&I_4 = -2\epsilon^2(2\lambda+\epsilon^p\eta) \left<\Delta_h\widehat{e}^n,\,\widehat{e}^{n+1}\right>_{\Omega},\\
			&I_5 = 2\left<\widehat{\tau}^{n+1},\,\widehat{e}^{n+1}\right>_{-1,h}.
		\end{aligned}
		\notag
	\end{equation}
	Due to the monotonicity of $\beta\beta'$ on $(-1,1)$, we easily see that $I_1\leq0$.
	Besides, from the convexity of
	$$
	\left<\beta'(\phi),\sum_{\zeta=x,y,z}a_{\zeta} \left(\left|D_{\zeta}\phi\right|^2\right)\right>_{\Omega},
	$$
	we can also conclude $I_2\leq0$. Concerning $I_3$, we can find  some $\xi\in \mathcal{C}_{\mathrm{per}}$ with its values  $\xi_{i,j,k}$ staying between $\big(\mathcal{P}_h\widehat{\Phi}^n\big)_{i,j,k}$ and $\phi^n_{i,j,k}$, $1\leq i,j,k\leq N$, such that
	\begin{equation}\notag
		I_3 = 2\Big<\big[\lambda\xi\beta''(\xi) +(2\lambda+\epsilon^p\eta)\beta'(\xi)\big] \widehat{e}^n,\, \widehat{e}^{n+1}\Big>_{\Omega}.
	\end{equation}
	Then using \eqref{modified separation property} and (\ref{strict separation numerical}), we obtain
	\begin{equation}
		\begin{aligned}
			I_3\leq C\left|\left<\widehat{e}^{n},\,\widehat{e}^{n+1}\right>_{\Omega}\right|
			\leq 2C \big( \|\widehat{e}^{n}\|_2^2+\|\widehat{e}^{n+1}\|_2^2\big),
		\end{aligned} \notag
	\end{equation}
	where $C>0$ depends on $\lambda$, $\epsilon$, $\eta$ and $\theta$. The remaining two terms $I_4$ and $I_5$ can be estimated by simply using the Cauchy-Schwarz inequality such that
	\begin{align*}
		I_4
		& = -2\epsilon^2(2\lambda+\epsilon^p\eta) \left<\widehat{e}^n,\,\Delta_h\widehat{e}^{n+1}\right>_{\Omega} \leq  \frac{\epsilon^4}{2}\|\Delta_h\widehat{e}^{n+1}\|_2^2
		+ 2(2\lambda+\epsilon^p\eta)^2  \|\widehat{e}^{n}\|_2^2,
	\end{align*}
	and
	\begin{equation}
		I_5\leq \|\widehat{e}^{n+1}\|_{-1,h}^2 + \|\widehat{\tau}^n\|_{-1,h}^2.\notag
	\end{equation}
	Collecting the above estimates, we infer from  (\ref{sec5est1}) that
	\begin{align}
		\frac{1}{\Delta t}&\left(\|\widehat{e}^{n+1}\|_{-1,h}^2 -\|\widehat{e}^{n}\|_{-1,h}^2\right) + \frac{3}{2}\epsilon^4\|\Delta_h\widehat{e}^{n+1}\|_2^2 +2\lambda(\lambda+\epsilon^p\eta)\|\widehat{e}^{n+1}\|_2^2 \notag \\
		&\leq \widehat{C}_1  \big(\|\widehat{e}^{n}\|_2^2 +\|\widehat{e}^{n+1}\|_2^2\big) +\|\widehat{e}^{n+1}\|_{-1,h}^2 +\|\widehat{\tau}^{n+1}\|_{-1,h}^2,
		\label{h-1-es1}
	\end{align}
	where $\widehat{C}_1>0$ depends on $\lambda$, $\epsilon$, $\eta$ and $\theta$.
	On the other hand, using the fact $\overline{\widehat{e}^{n+1}}=0$, the $H^2_h$-estimate in \cite[Proposition 2.2]{FengNMPDE2017} such that
	\begin{align}
		\|\widehat{e}^{n+1}\|_{H^2_h}\leq C \|\Delta_h \widehat{e}^{n+1}\|_2, \label{HH2hes}
	\end{align}
	and the interpolation inequality, we find
	\begin{equation}
		\begin{aligned}
			&\|\widehat{e}^{n+1}\|_2^2\leq C\|\Delta_h\widehat{e}^{n+1}\|_2^{\frac23} \|\widehat{e}^{n+1}\|_{-1,h}^{\frac43}.
		\end{aligned}\notag
	\end{equation}
	A similar estimate also holds for $\widehat{e}^{n}$. Applying the above estimates and Young's inequality, we can eventually write  \eqref{h-1-es1}   as
	\begin{align}
		&	\|\widehat{e}^{n+1}\|_{-1,h}^2- \|\widehat{e}^{n}\|_{-1,h}^2 + \Delta t\Big[\epsilon^4\|\Delta_h\widehat{e}^{n+1}\|_2^2
		+2\lambda(\lambda+\epsilon^p\eta) \|\widehat{e}^{n+1}\|_2^2\Big] \notag \\
		&\quad \leq  \widehat{C}_2 \Delta t \big(  \|\widehat{e}^{n+1}\|_{-1,h}^2 + \|\widehat{e}^{n}\|_{-1}^2\big)
		+\frac{\epsilon^4}{2}\Delta t \|\Delta_h\widehat{e}^{n}\|_2^2 + \Delta t \|\widehat{\tau}^{n+1}\|_{-1,h}^2,\notag
	\end{align}
	where $\widehat{C}_2>0$ depends on $\lambda$, $\epsilon$, $\eta$ and $\theta$.
	Thus, for $0<\Delta t\leq (2\widehat{C}_2)^{-1}$, using the fact $\widehat{e}^0=0$, \eqref{tau-n1} and the discrete Gronwall inequality, we can conclude that
	\begin{equation}\label{convergence analysis}
		\|\widehat{e}^{n+1}\|_{-1,h}^2 + \Delta t \sum_{k=1}^{n+1} \big[\epsilon^4 \|\Delta_h\widehat{e}^{k}\|_2^2 +2\lambda(\lambda+\epsilon^p\eta)\|\widehat{e}^{k}\|_2^2\big]
		\leq \widehat{C}_3(\Delta t^6+h^6),
	\end{equation}
	where $\widehat{C}_3>0$ depends on $\lambda$, $\epsilon$, $\eta$ and $\theta$ and $T$.
	
	It remains to recover the \emph{a priori} assumption \eqref{apri-es}. Obviously, it is valid for $n=0$. Then we apply an induction argument. Assume that \eqref{apri-es} holds at the previous time step corresponding to $n\in \mathbb{N}$. From \eqref{convergence analysis} and
the linear refinement constraint $\Delta t \leq C_1h$, we deduce that
	\begin{equation}
		\|\widehat{e}^{n+1}\|_2\leq h^{-1}\|\widehat{e}^{n+1}\|_{-1,h}\leq C(\Delta t^2+h^2)\leq \Delta t^\frac{5}{3}+ h^\frac{5}{3},
		\label{priori assumptionB}
	\end{equation}
	provided that $\Delta t$ and $h$ are sufficiently small. Hence, \eqref{apri-es} holds for all $n\in \mathbb{N}$.
	
	Now for any $m\in \mathbb{N}$, $m\leq M$, recalling the definition of the two error functions (see \eqref{old-err} and \eqref{new-err}), we have
	\begin{align*}
		\|e^m\|_{-1,h}&\leq \|\widehat{e}^m\|_{-1,h} +\|\mathcal{P}_h\Phi_N^m- \mathcal{P}_h\widehat{\Phi}^m\|_{-1,h}\\
		&\leq \|\widehat{e}^m\|_{-1,h}  + \Delta t \|\mathcal{P}_h(\mathcal{P}_N\Phi_{\Delta t,1}^m)\|_{-1,h}  + \Delta t^2 \|\mathcal{P}_h(\mathcal{P}_N\Phi_{\Delta t,2}^m)\|_{-1,h}  + h^2 \|\mathcal{P}_h(\mathcal{P}_N\Phi_{h,1}^m)\|_{-1,h},
	\end{align*}
	and in a similar manner,
	\begin{align*}
		\|e^m\|_2 &\leq \|\widehat{e}^m\|_2+ \Delta t \|\mathcal{P}_h(\mathcal{P}_N\Phi_{\Delta t,1}^m)\|_2+ \Delta t^2 \|\mathcal{P}_h(\mathcal{P}_N\Phi_{\Delta t,2}^m)\|_2+ h^2 \|\mathcal{P}_h(\mathcal{P}_N\Phi_{h,1}^m)\|_2,\\
		\|\Delta_h e^m\|_2 &\leq \|\Delta_h  \widehat{e}^m\|_2+ \Delta t \|\Delta_h\mathcal{P}_h(\mathcal{P}_N\Phi_{\Delta t,1}^m)\|_2+\Delta t^2\|\Delta_h\mathcal{P}_h(\mathcal{P}_N\Phi_{\Delta t,2}^m)\|_2+h^2\|\Delta_h\mathcal{P}_h(\mathcal{P}_N\Phi_{h,1}^m)\|_2.
	\end{align*}
	Thanks to the above observations, we can conclude the error estimate \eqref{des-h-1} from \eqref{convergence analysis}, \eqref{fourier projection} and the uniform boundedness of $\Phi_{\Delta t,1}$, $\Phi_{\Delta t,2}$, $\Phi_{h,1}$ and its derivatives.
\end{proof}

\subsection{Error estimate in $l^{\infty}(0,T;L_h^2)\cap l^2(0,T;H^3_h)$}

We are now in a position to prove Theorem \ref{thm1}.	
\begin{proof}
	First, from the lower order error estimate \eqref{convergence analysis} for the auxiliary numerical error $\widehat{e}^n$, we have obtained a rough $L_h^2$ error estimate \eqref{apri-es}, which together with the inverse inequality yields the strict separation property of the numerical solution $\phi^m$ for all $m\in \mathbb{N}$, $m\leq M$ (see \eqref{strict separation numerical}).
	
	Next, subtracting the numerical scheme \eqref{numerical phi}--\eqref{numerical mu} from \eqref{correction trucation error} and taking (discrete) inner product between the resultant and  $2\widehat{e}^{n+1}$, we get
	\begin{equation}\label{l2convergence}
		\frac{1}{\Delta t}\left(\|\widehat{e}^{n+1}\|_2^2-\|\widehat{e}^{n}\|_2^2 +\|\widehat{e}^{n+1}-\widehat{e}^{n}\|_2^2\right) + 2\epsilon^4\|\nabla_h\Delta_h\widehat{e}^{n+1}\|_2^2 +2\lambda(\lambda+\epsilon^p\eta)\|\nabla_h\widehat{e}^{n+1}\|_2^2 = \sum_{i=1}^6 J_i,
	\end{equation}
	where
	\begin{equation}
		\begin{aligned}
			&J_1 = 2\left<\left(\beta(\widehat{\Phi}^{n+1}) \beta'(\widehat{\Phi}^{n+1}) -\beta(\phi^{n+1})\beta'(\phi^{n+1})\right),\, \Delta_h\widehat{e}^{n+1}\right>_{\Omega},\\
			&J_2 = 2\epsilon^2\left<\sum_{\zeta=x,y,z} \left[\beta''(\widehat{\Phi}^{n+1}) a_{\zeta}\left(|D_{\zeta}\widehat{\Phi}^{n+1}|^2\right) -\beta''(\phi^{n+1}) a_{\zeta}\left(|D_{\zeta}\phi^{n+1}|^2\right)\right],\, \Delta_h\widehat{e}^{n+1}\right>_{\Omega},\\
			&J_3 = -4\epsilon^2\left<\sum_{\zeta=x,y,z}\left[d_{\zeta} \left(A_{\zeta} \big(\beta'(\widehat{\Phi}^{n+1})\big) D_{\zeta}\widehat{\Phi}^{n+1} -A_{\zeta}\big(\beta'(\phi^{n+1})\big) D_{\zeta}\phi^{n+1}\right)\right],\, \Delta_h\widehat{e}^{n+1}\right>_{\Omega},\\
			&J_4 = - 2 \left<\left(\lambda\widehat{\Phi}^n \beta'(\widehat{\Phi}^{n}) -\lambda\phi^n\beta'(\phi^{n}) +(\lambda+\epsilon^p\eta) \beta(\widehat{\Phi}^{n}) -(\lambda+\epsilon^p\eta)\beta(\phi^{n})\right),\, \Delta_h\widehat{e}^{n+1}\right>_{\Omega},\\
			&J_5 = 2\epsilon^2(2\lambda+\epsilon^p\eta)\left<\Delta_h \widehat{e}^n,\,\Delta_h\widehat{e}^{n+1}\right>_{\Omega},\\
			&J_6 = 2\left<\widehat{\tau}^{n+1},\,\widehat{e}^{n+1}\right>_{\Omega}.
		\end{aligned}
		\notag
	\end{equation}
	For  $J_1$, we can find some $\xi_1\in \mathcal{C}_{\mathrm{per}}$ with its values  $(\xi_1)_{i,j,k}$ staying between $\big(\mathcal{P}_h\widehat{\Phi}^{n+1}\big)_{i,j,k}$ and $\phi^{n+1}_{i,j,k}$, $1\leq i,j,k\leq N$, such that
	\begin{equation}
		\begin{aligned}
			J_1 &=2\left<\left(\beta'(\xi_1)^2 +\beta(\xi_1)\beta''(\xi_1)\right) \widehat{e}^{n+1},\, \Delta_h\widehat{e}^{n+1}\right>_{\Omega}\\
			&\leq C\|\widehat{e}^{n+1}\|_2\|\Delta_h\widehat{e}^{n+1}\|_2\\
			&\leq C\left(\|\Delta_h\widehat{e}^{n+1}\|_2^2+  \|\widehat{e}^{n+1}\|_2^2\right),
		\end{aligned}\notag
	\end{equation}
	where the constant $C>0$ only depends on $\theta$. By a similar argument, we can estimate $J_4$ as
	\begin{equation}
		\begin{aligned}
			J_4 &= -2 \left<\big[\lambda(\beta'(\xi_2) +\xi_3\beta''(\xi_2)) +(\lambda+\epsilon^p\eta)\beta'(\xi_2)\big] \widehat{e}^{n},\, \Delta_h\widehat{e}^{n+1}\right>_{\Omega}\\
			&\leq C \|\widehat{e}^{n}\|_2\|\Delta_h\widehat{e}^{n+1}\|_2\\
			&\leq C\left( \|\Delta_h\widehat{e}^{n+1}\|_2^2 +   \|\widehat{e}^{n}\|_2^2\right),
		\end{aligned}
		\notag
	\end{equation}
	where the constant $C>0$ only depends on $\lambda$, $\eta$, $\epsilon$ and $\theta$.
	
	Next, we handle $J_2$ by making the following decomposition:
	\begin{equation}
		\begin{aligned}
			J_2 & = 2\epsilon^2\left(\left<\sum_{\zeta=x,y,z}\left(\beta''(\widehat{\Phi}^{n+1}) -\beta''(\phi^{n+1})\right) a_{\zeta}\left(|D_{\zeta}\widehat{\Phi}^{n+1}|^2\right),\, \Delta_h\hat{e}^{n+1}\right>_{\Omega}\right.\\
			&\qquad+\left.\left<\sum_{\zeta=x,y,z}\beta''(\phi^{n+1}) a_{\zeta}\left(|D_{\zeta}\widehat{\Phi}^{n+1}|^2 -|D_{\zeta}\phi^{n+1}|^2\right),\, \Delta_h\widehat{e}^{n+1}\right>_{\Omega}\right)\\
			&=: 2\epsilon^2(J_{2a}+J_{2b}).
		\end{aligned}
		\notag
	\end{equation}
	Again, we can find some $\xi_3\in \mathcal{C}_{\mathrm{per}}$ with its values  $(\xi_3)_{i,j,k}$ staying between $\big(\mathcal{P}_h\widehat{\Phi}^{n+1}\big)_{i,j,k}$ and $\phi^{n+1}_{i,j,k}$, $1\leq i,j,k\leq N$, such that
	\begin{equation}
		\begin{aligned}
			J_{2a}&=\left<\sum_{\zeta=x,y,z}\beta'''(\xi_3)\widehat{e}^{n+1} a_{\zeta}\left(|D_{\zeta}\widehat{\Phi}^{n+1}|^2\right),\, \Delta_h\widehat{e}^{n+1}\right>_{\Omega}\\
			&\leq C \left\|\sum_{\zeta=x,y,z}a_{\zeta}\left(|D_{\zeta}\widehat{\Phi}^{n+1}|^2\right)\right\|_{\infty} \|\widehat{e}^{n+1}\|_2\|\Delta_h\widehat{e}^{n+1}\|_2\\
			&\leq C\|\widehat{e}^{n+1}\|_2 \|\Delta_h\widehat{e}^{n+1}\|_2,
		\end{aligned}
		\notag
	\end{equation}
	where $C>0$ depends on $\theta$ and norms of $\Phi$. For the second term, we infer from \eqref{strict separation numerical}, H\"{o}lder's inequality and the uniform $H_h^2$-bound \eqref{destH2} that
	\begin{equation}
		\begin{aligned}
			J_{2b}
			&\leq C \left|\left<\sum_{\zeta=x,y,z}a_{\zeta} \left(|D_{\zeta}(\widehat{\Phi}^{n+1}+\phi^{n+1})| |D_{\zeta}\widehat{e}^{n+1}|\right),\, \Delta_h\widehat{e}^{n+1}\right>_{\Omega}\right|\\
			&\leq C \left\|\sum_{\zeta=x,y,z}a_{\zeta} \left(|D_{\zeta}(\widehat{\Phi}^{n+1}+\phi^{n+1})|\right)\right\|_{L_h^4} \left\|\sum_{\zeta=x,y,z}\left(a_{\zeta}|D_{\zeta}\widehat{e}^{n+1}|\right)\right\|_{L_h^4} \left\|\Delta_h\widehat{e}^{n+1}\right\|_2\\
			&\leq C\left(\|\nabla_h\widehat{\Phi}^{n+1}\|_{L_h^4} +\|\nabla_h\phi^{n+1}\|_{L_h^4}\right) \left\|\nabla_h\widehat{e}^{n+1}\right\|_{L_h^4} \left\|\Delta_h\widehat{e}^{n+1}\right\|_2\\
			&\leq C\left(\|\widehat{\Phi}^{n+1}\|_{H^2_h} +\|\phi^{n+1}\|_{H^2_h} \right) \|\widehat{e}^{n+1}\|_{H^2_h} \|\Delta_h\widehat{e}^{n+1}\|_2\\
			&\leq C\|\widehat{e}^{n+1}\|_{H^2_h} \|\Delta_h\widehat{e}^{n+1}\|_2.
		\end{aligned}\notag
	\end{equation}
	Here, we have used the estimate  $\left\|\sum_{\zeta=x,y,z}a_{\zeta}|D_{\zeta}f|\right\|_{L_h^4} \leq C\left\|\nabla_hf\right\|_{L^4_h}$ for all $f\in\mathcal{C}_{\mathrm{per}}$  and discrete interpolation inequalities. Combining the estimates for
	$J_{2a}$ and $I_{2b}$, we deduce from \eqref{HH2hes} that
	\begin{equation}
		J_2\leq C\|\widehat{e}^{n+1}\|_{H^2_h} \|\Delta_h\widehat{e}^{n+1}\|_2
		\leq C \|\Delta_h\widehat{e}^{n+1}\|_2^2.
		\notag
	\end{equation}
	
	The term $J_3$ can be treated in a similar manner. Indeed, we have
	\begin{equation}
		\begin{aligned}
			J_3 &\leq 4\epsilon^2\left(\left|\left<\sum_{\zeta=x,y,z} d_{\zeta}\left(A_{\zeta}\big(\beta'(\widehat{\Phi}^{n+1}) -\beta'(\phi^{n+1})\big)D_{\zeta}\widehat{\Phi}^{n+1}\right),\, \Delta_h\widehat{e}^{n+1}\right>_{\Omega}\right|\right.\\
			&\qquad+
			\left.\left|\left<\sum_{\zeta=x,y,z}d_{\zeta} \big(A_{\zeta}\beta'(\phi^{n+1})D_{\zeta}\widehat{e}^{n+1}\big),\, \Delta_h\widehat{e}^{n+1}\right>_{\Omega}\right|\right)\\
			&=:4\epsilon^2\left(J_{3a}+J_{3b}\right).
		\end{aligned}
		\notag
	\end{equation}
	Concerning $J_{3a}$, there exists some $\xi_4\in \mathcal{C}_{\mathrm{per}}$ with its values  $(\xi_4)_{i,j,k}$ staying between $\big(\mathcal{P}_h\widehat{\Phi}^{n+1}\big)_{i,j,k}$ and $\phi^{n+1}_{i,j,k}$, $1\leq i,j,k\leq N$, such that
	\begin{equation}
		\begin{aligned}
			J_{3a}&=\left|\left<\sum_{\zeta=x,y,z}d_{\zeta} \left(A_{\zeta}\big(\beta''(\xi_4)\widehat{e}^{n+1}\big) D_{\zeta}\widehat{\Phi}^{n+1}\right),\, \Delta_h\widehat{e}^{n+1}\right>_{\Omega}\right|\\
			&\leq C\sum_{\zeta=x,y,z} \|\beta''(\xi_4)\|_\infty \|D_{\zeta}\widehat{\Phi}^{n+1}\|_{\infty}\|\widehat{e}^{n+1}\|_2 \|\nabla_h\Delta_h\widehat{e}^{n+1}\|_2\\
			&\leq C\left\|\widehat{e}^{n+1}\right\|_2 \left\|\nabla_h\Delta_h\widehat{e}^{n+1}\right\|_2,
		\end{aligned}
		\notag
	\end{equation}
	where we have used Lemma \ref{lem:sumbyparts} and H\"{o}lder's inequality. Similarly, for $J_{3b}$, we have
	$$		
	J_{3b}\leq C\left\|\nabla_h \widehat{e}^{n+1}\right\|_2 \left\|\nabla_h\Delta_h\widehat{e}^{n+1}\right\|_2.
	$$
	As a consequence, it holds
	\begin{align*}
		J_3 &\leq C\left\|\widehat{e}^{n+1}\right\|_{H^1_h} \left\|\nabla_h\Delta_h\widehat{e}^{n+1}\right\|_2\\
		&\leq  \epsilon^4\|\nabla_h \Delta_h \widehat{e}^{n+1}\|_2^2 +C \left\|\widehat{e}^{n+1}\right\|_{H^1_h}^2\\
		&\leq \epsilon^4\|\nabla_h \Delta_h \widehat{e}^{n+1}\|_2^2 + C \left\|\Delta_h \widehat{e}^{n+1}\right\|_2^2.
	\end{align*}
	For the remaining two terms, an application of Cauchy-Schwarz inequality  yields that
	\begin{align*}
		J_5&\leq 2\epsilon^2(2\lambda+\epsilon^p\eta) \|\Delta_h \widehat{e}^n\|_2 \|\Delta_h \widehat{e}^{n+1}\|_2
		\leq \epsilon^2(2\lambda+\epsilon^p\eta)\big(\| \Delta_h \widehat{e}^{n+1}\|_2^2 +  \|\Delta_h \widehat{e}^n\|_2^2\big), \\
		J_6&\leq\|\widehat{e}^{n+1}\|_2^2+\|\widehat{\tau}^{n+1}\|_2^2.
	\end{align*}
	
	Combining the above estimates, we can deduce from \eqref{HH2hes} and \eqref{l2convergence} that
	\begin{align}
		&\frac{1}{\Delta t}\left(\|\widehat{e}^{n+1}\|_2^2-\|\widehat{e}^{n}\|_2^2\right) + \epsilon^4\|\nabla_h\Delta_h\widehat{e}^{n+1}\|_2^2 +2\lambda(\lambda+\epsilon^p\eta) \|\nabla_h\widehat{e}^{n+1}\|_2^2 \notag \\
		&\qquad \leq C\left( \|\Delta_h\widehat{e}^{n+1}\|_2^2 + \|\Delta_h\widehat{e}^{n}\|_2^2\right) +\|\widehat{\tau}^{n+1}\|_2^2,
		\label{diff1}
	\end{align}
	which together with the fact $\widehat{e}^{0}=0$ leads to the following inequality
	\begin{align}
		&\|\widehat{e}^{n+1}\|_2^2 + \Delta t\sum_{k=1}^{n+1}\left[\epsilon^4\big\|\nabla_h\Delta_h \widehat{e}^{k}\big\|_2^2+2\lambda(\lambda+\epsilon^p\eta) \|\nabla_h\widehat{e}^{k}\|_2^2\right]  \notag\\
		&\quad \leq  C\Delta t \sum_{k=1}^{n+1} \|\Delta_h\widehat{e}^{k}\|_2^2
		+ \Delta t \sum_{k=0}^{n} \big\|\widehat{\tau}^{k+1}\big\|_2^2.
		\label{diff2}
	\end{align}
	Thanks to the truncation error \eqref{tau-n1} and the lower order error estimate \eqref{convergence analysis}, we deduce from \eqref{diff2} that
	\begin{equation}\label{conclusion}
		\left\|\widehat{e}^{n+1}\right\|_2^2 + \Delta t\sum_{k=1}^{n+1}\left[\epsilon^4\big\|\nabla_h\Delta_h \widehat{e}^{k}\big\|_2^2+2\lambda(\lambda+\epsilon^p\eta) \|\nabla_h\widehat{e}^{k}\|_2^2\right] \leq C(\Delta t^6+ h^6).
	\end{equation}
	
	Finally, by a comparison between the two error functions $\widehat{e}^{n+1}$ and $e^{n+1}$ (using \eqref{old-err}, \eqref{pertubation expansion} and \eqref{new-err}), from the projection estimate \eqref{fourier projection}, the uniform boundedness of $\Phi_{\Delta t,1}$, $\Phi_{\Delta t,2}$, $\Phi_{h,1}$, the estimates \eqref{convergence analysis}, \eqref{conclusion}, we can achieve the conclusion  \eqref{L2-est}.

This completes the proof of Theorem \ref{thm1}.
\end{proof}

\section{Numerical Experiments} \label{sec:numerical results}
\setcounter{equation}{0}

In this section, we implement a preconditioned steepest descent (PSD) algorithm to solve the fully discrete numerical scheme \eqref{numerical phi}--\eqref{numerical mu}, following the framework proposed in \cite{feng2017}. This algorithm has been applied to the FCH equation \eqref{phi} with a regular potential \eqref{regF} in \cite{feng2018}. Further applications can be found in \cite{zhang2021} and the references cited therein.

Observe that our scheme \eqref{numerical phi}--\eqref{numerical mu} can be recast as a minimization problem of the discrete energy (\ref{discrete energy}). This is equivalent to solve the zero point of the discrete variation of (\ref{discrete energy}):
$$
\mathcal{N}_h(\phi) = f^n,
$$
where
$$
\mathcal{N}_h(\phi) = \frac{\phi}{\Delta t } - \Delta_h\delta_{\phi}E_h^c(\phi)\quad \text{and}\quad
f^n= \frac{\phi^{n}}{\Delta t } + \Delta_h\delta_{\phi}E_h^e(\phi^{n}).
$$
The essential idea of the PSD solver is to use a linearized version of the nonlinear operator as a pre-conditioner. Here, the preconditioner $\mathcal{L}_h:\mathring{\mathcal{C}}_{\mathrm{per}} \rightarrow\mathring{\mathcal{C}}_{\mathrm{per}}$ is defined as
\begin{equation*}
	\mathcal{L}_h: =	\begin{aligned}
		&\frac{1}{\Delta t} - \epsilon^4\Delta_h^3 + \epsilon^2\theta_1\Delta_h^2 - (\lambda^2+\lambda\epsilon^p\eta+\theta_2)\Delta_h ,
	\end{aligned}
\end{equation*}
where $\theta_1$, $\theta_2>0$ are two parameters that can be adjusted. Obviously, $\mathcal{L}_h$ is a positive, symmetric operator. In particular, this metric can be used to find an appropriate search direction for the steepest descent solver. Given the current iterate $\phi^n$, we obtain the next iterate direction by finding a  $d_n\in\mathcal{C}_{\mathrm{per}}$ such that
$$\mathcal{L}_h[d_n] = r_n - \overline{r_n},\qquad \overline{r_n}:=f^n - \mathcal{N}_h(\phi^n),$$
where $r_n$ is the nonlinear residual of the $n^{\mathrm{th}}$ iterate $\phi^n$. Using this linear operator, we then solve the equation by the Fast Fourier Transform (FFT). Consequently, the next iterate is obtained as
$$\phi^{n+1} := \phi^n + \alpha_n d_n,$$
where $\alpha_n\in\mathbb{R}$ is the unique solution to the steepest descent line minimization problem
$$\alpha_n:=\mathop{\rm{argmin}}\limits_{\alpha\in\mathbb{R}}E_h(\phi^n + \alpha d_n)= \mathop{\rm{argzero}}\limits_{\alpha\in\mathbb{R}}\left<\mathcal{N}_h(\phi^n+\alpha d_n)-f^n,d_n\right>_{\Omega}.$$

\subsection{Convergence test}

In this subsection we carry out an accuracy check for the numerical scheme (\ref{numerical phi})--(\ref{numerical mu}). For simplicity, the computational domain is assumed to be the unit square $\Omega = (0,1)^2$, and the phase variable is given as
$$
\Phi(x,y,t) = \frac{1}{\pi}\sin(2\pi x)\cos(2\pi y)\cos(t).
$$
In particular, we see that $1+\Phi$ and $1-\Phi$ stays positive at a pointwise level, so no singularity occurs during the computation. In order to make $\Phi$ become the solution of the PDE system \eqref{phi}, we have to add an artificial, time dependent forcing term, then we can continue to finish the convergence test.

We first give the parameters as $\epsilon = 0.5$, $\eta = 1$, $\lambda = 3$ and $p=2$. We compute the $L_h^2$ numerical errors with grid size $N = 16,\, 32,\, 64,\, 128,\, 256,\, 512$. To verify that our numerical scheme has first order in time and second order in space accuracy, we set the time step size $\Delta t=16h^2$ and $\Delta t=h$ respectively, and plot $\ln |e|$ versus $\ln N$ to demonstrate the temporal convergence order. In the first case $\Delta t=16h^2$, the expected temporal numerical accuracy assumption $e = C(\Delta t + h^2)$ indicates that $\ln|e| = \ln C - 2\ln N$. This is consistent with the red lines shown in Figure \ref{fig1}, with its slope approximately being $-2.0498$. In the second case $\Delta t=h$, the numerical error gradually approach the line with slope $-1$, as depicted with the purple dotted line.


\begin{figure}[!t]
	\centering
	\includegraphics[scale=0.5]{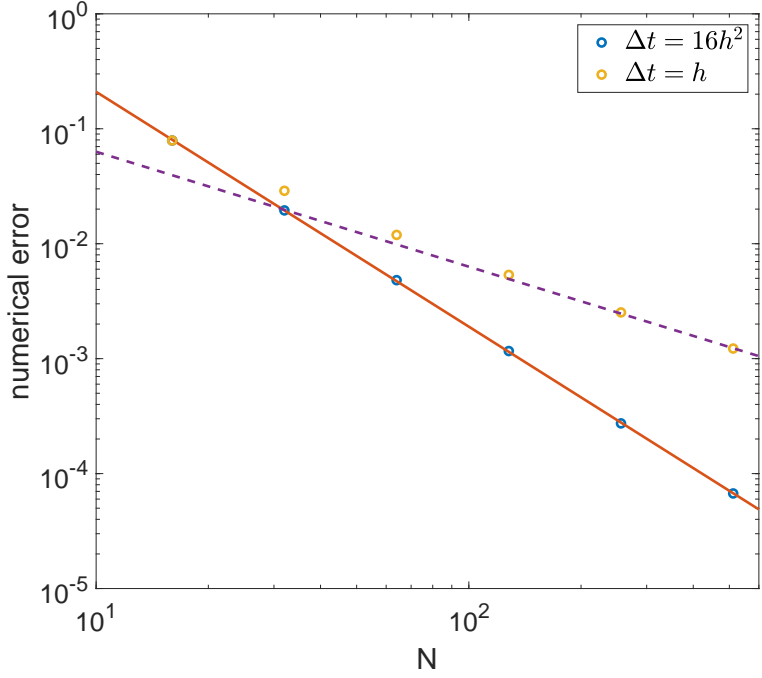}
	\caption{The discrete $L_h^2$ numerical error versus spatial resolution $N$ for $N = 16,\,32,\,64,\,128,\,256,\,512$. The blue dots represent the numerical error in the case $\Delta t = 16h^2$. The data are roughly on curves $CN^{-2}$, as is demonstrated by the red fitted line. Besides, the orange dots represent the numerical error in the case $\Delta t = h$. The data have an asymptote with a slope $-1$ (plotted by the purple dotted line), which verifies the first order in time accuracy.}
	\label{fig1}
\end{figure}

\subsection{Pearling bifurcation and  meandering instability}

\subsubsection{Adaptive time adjust strategy}

Hereafter, we apply an adaptive time strategy to perform our simulation. We set an upper bound and lower bound for both the energy decay rate and the phase change rate. If either of the two factors changes rapidly, that is, exceeding the upper bound, we shorten the time step, and if both factors change too slow, we increase the time step. The maximal time step is given as $2e-3$, and the upper bound and lower bound are set to be $1e-1$ and $1e-3$, respectively.

\subsubsection{The pearling bifurcation and sensitivity of the initial value}
\label{sec:pearl}

The pearling bifurcation is an important physical phenomenon that has been investigated in \cite{Doelman2014,Kraitzman2015,Promislow2015}. Several numerical simulations have been implemented in recent works \cite{zhang2020numerical,Zhang2021highly,Zhang2021analysis}, always with a regular potential. In our simulation, the following initial condition will be used:
\begin{equation}\label{pearlinitial}
	\phi(x,y,0) = 1.8\Big/\cosh\left(\frac{\sqrt{(x-0.5)^2+(y-0.5)^2}-0.42}{\ell\epsilon}\right)-0.9,
\end{equation}
where the parameter $\ell>0$ controls the interface width of the initial datum. Here we set $h=1/256$, $\epsilon=0.03$, $\eta=4$, $p=1$ and $\lambda = \ln(19)/0.9$ so that the minima $\pm\, r_*$ of $F$ are given with $r_*=0.9$. As a consequence, Figure \ref{figpearl035}, Figure \ref{figpearl039} and Figure \ref{figpearl05} show the time snapshots of simulations under different initial thickness of the ring. The pearl bifurcation appears when $\ell=0.35$ and $\ell=0.39$, but fails to exist when $\ell=0.5$. From this numerical experiment we can see that, the morphological structure of interfaces is highly sensitive to the initial datum. A slight change may lead to different intermediate states and long-time equilibria. The energy plot in Figure \ref{pearlenergy} shows a fast exponential decay at the beginning of the simulation, and when the pearl structure appears, the energy changes fast as well. In Figure \ref{totalenergy} we observe that in the above three cases for pearling bifurcation, the FCH energy, the Cahn-Hilliard energy $E_{\mathrm{CH}}$ and the phase-field Willmore (PFW) energy decrease at the same time.

\begin{figure}[!t]
	\centering
	\subfigure[$t=0$]{\includegraphics[scale=0.23,trim={30 0 30 0},clip]{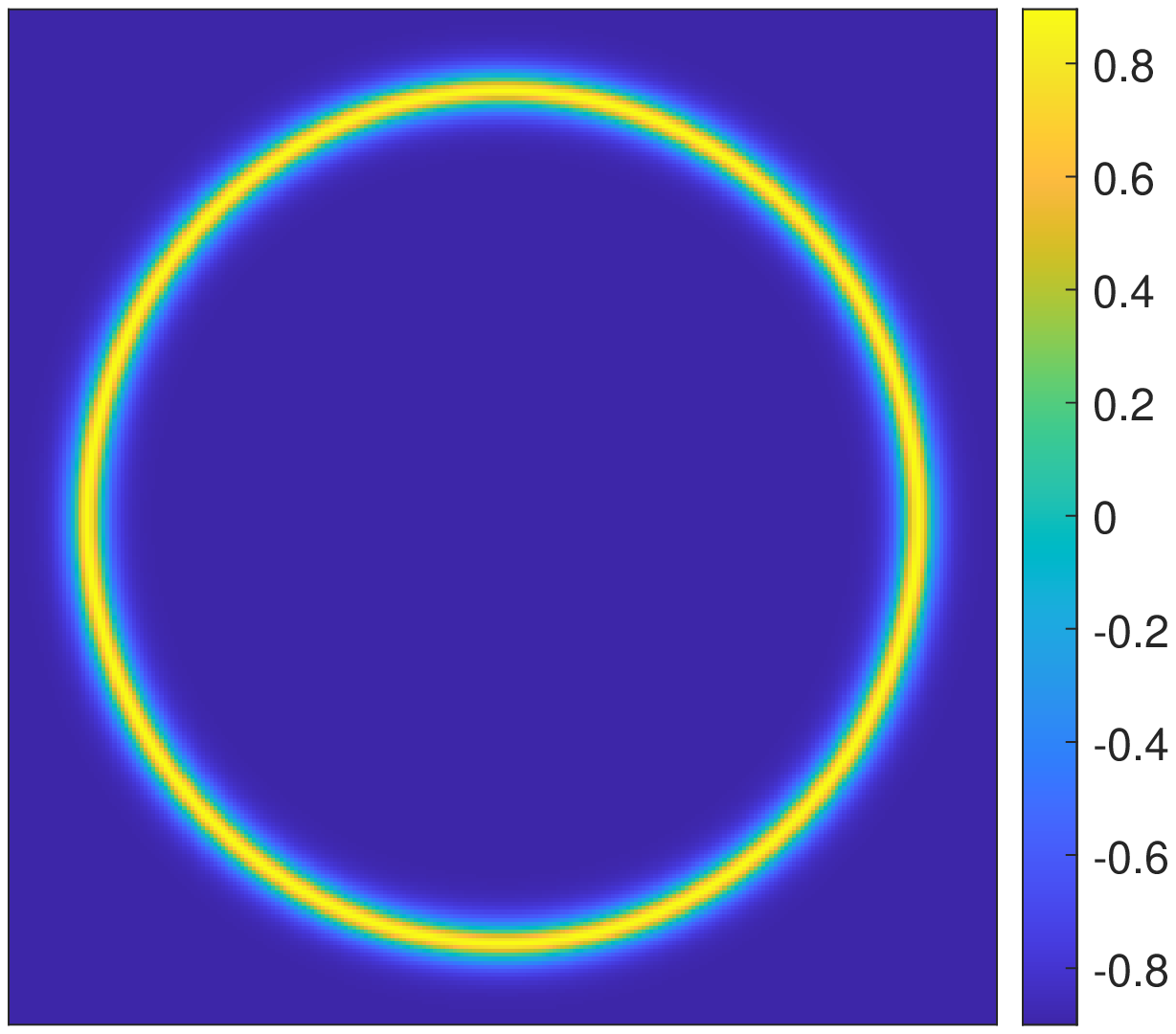}}
	\subfigure[$t=0.3$]{\includegraphics[scale=0.23,trim={30 0 30 0},clip]{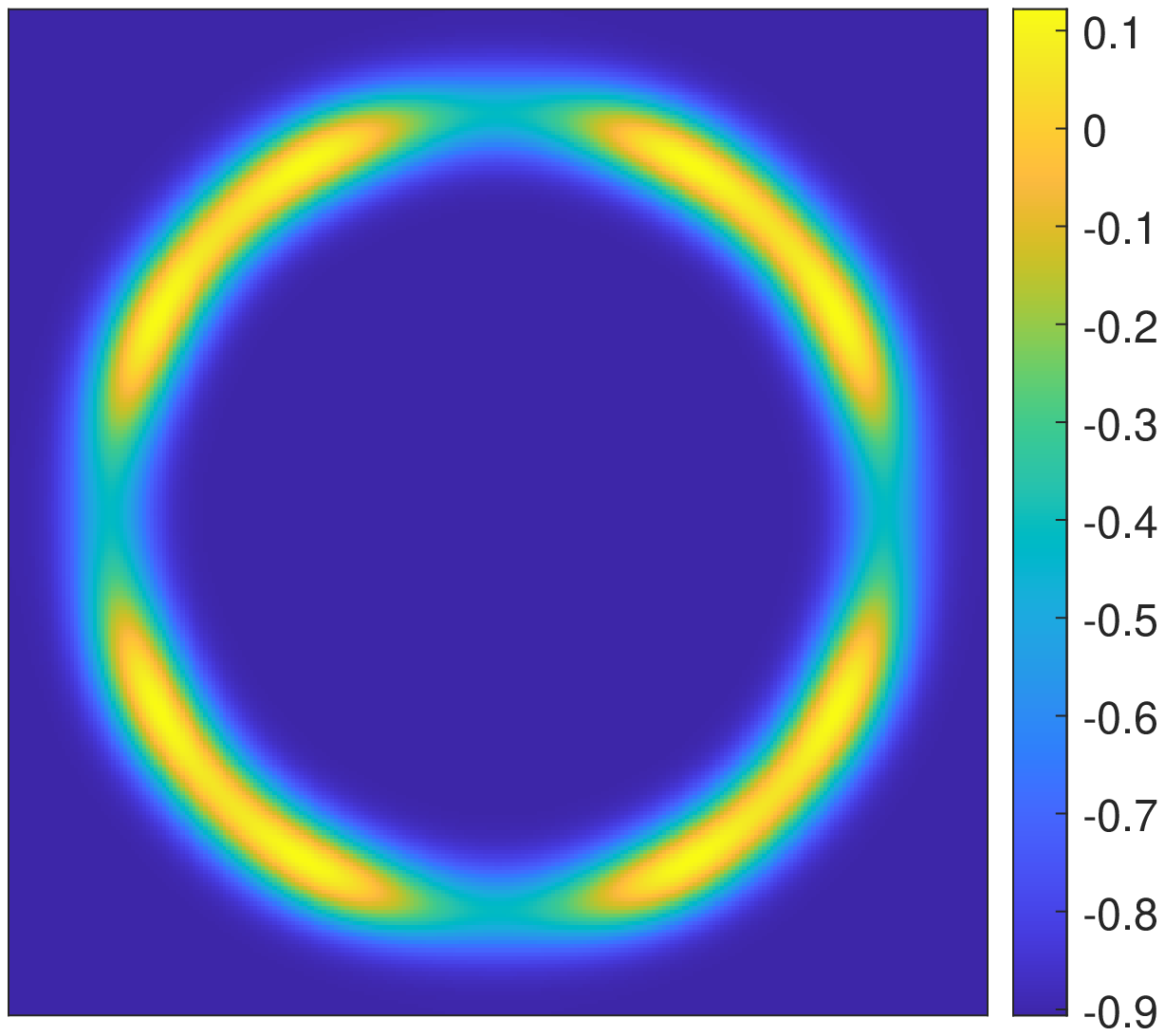}}
	\subfigure[$t=0.5$]{\includegraphics[scale=0.23,trim={30 0 30 0},clip]{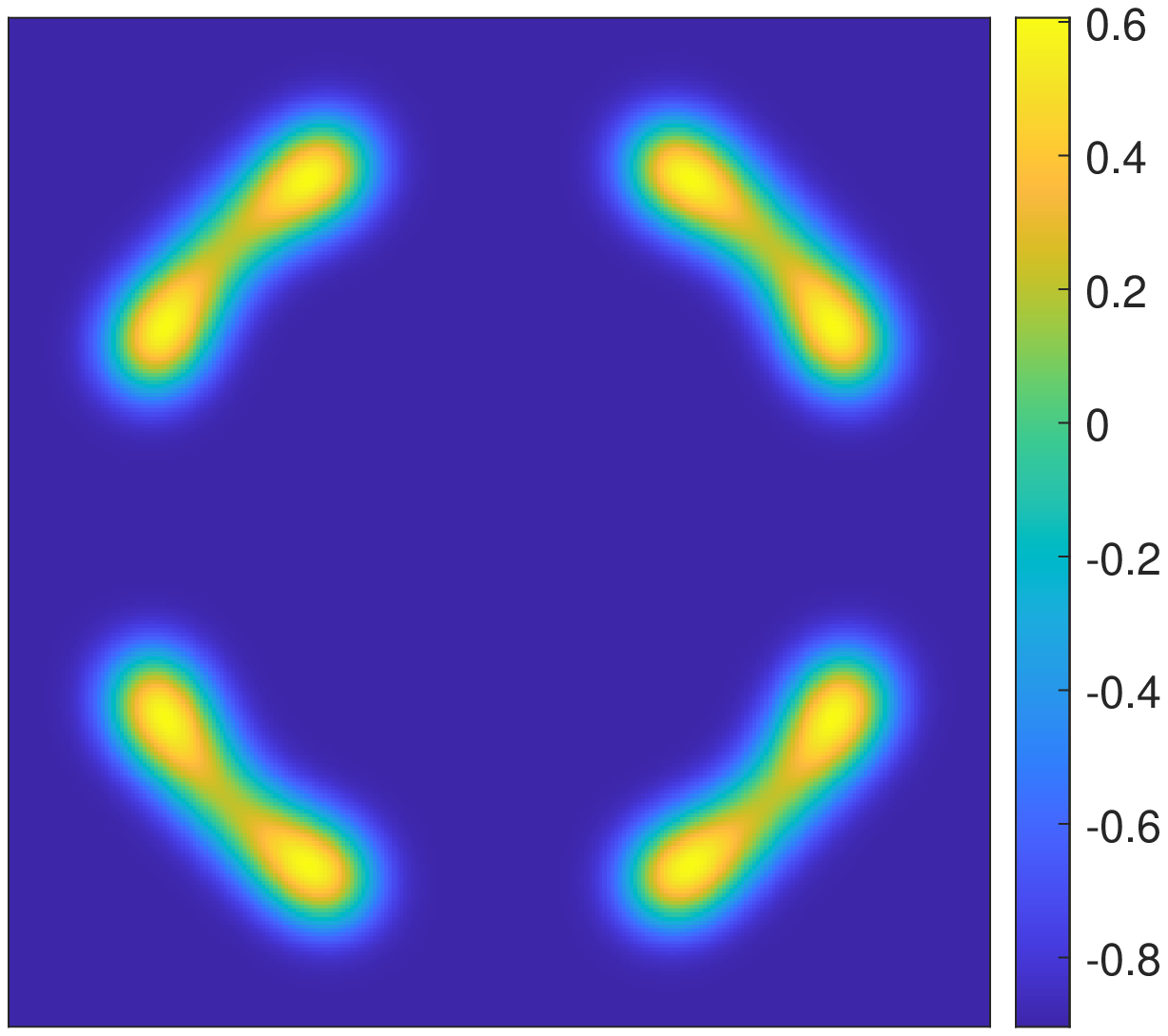}}
	\subfigure[$t=0.8$]{\includegraphics[scale=0.23,trim={30 0 30 0},clip]{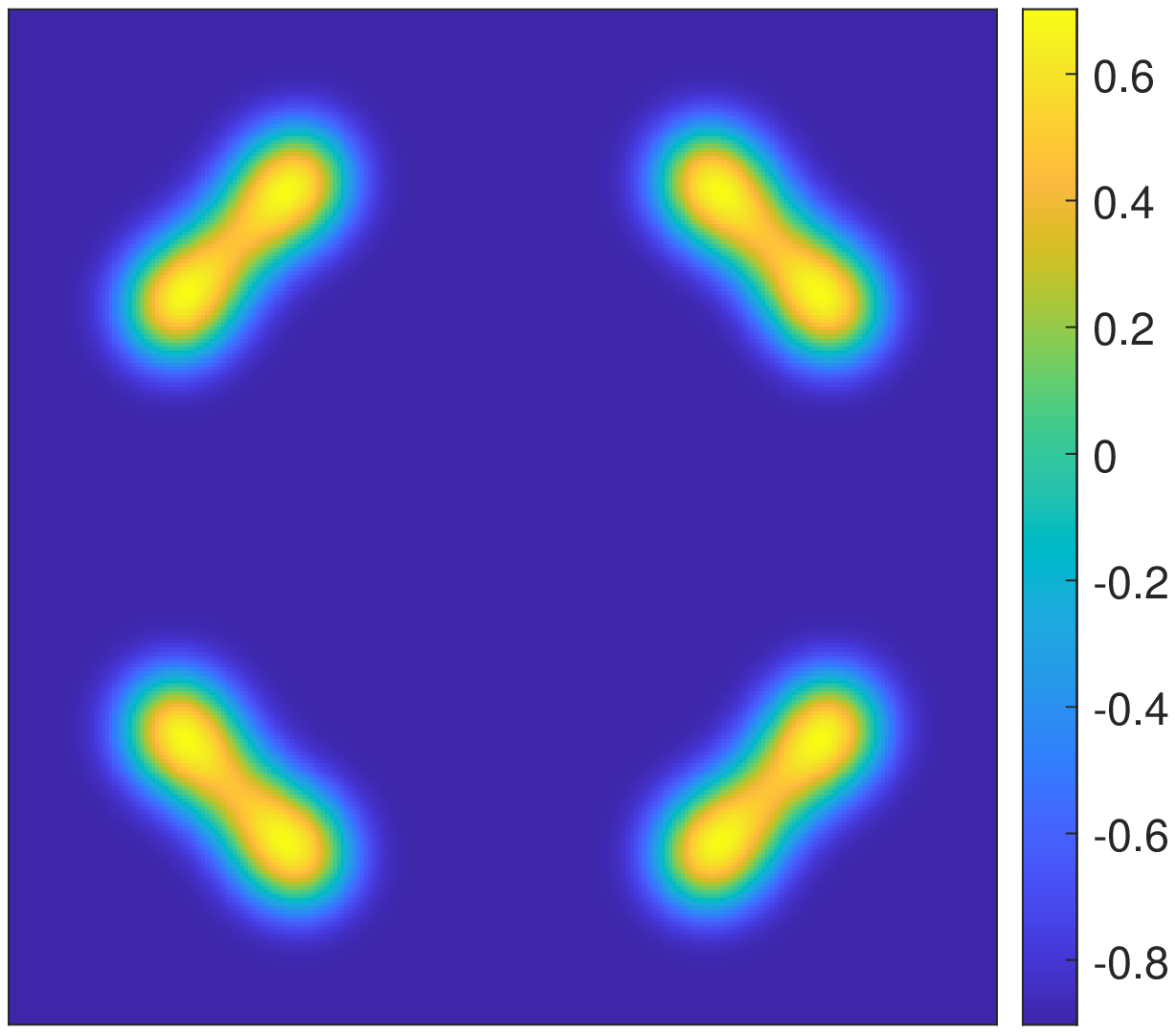}}
	\subfigure[$t=10$]{\includegraphics[scale=0.23,trim={30 0 30 0},clip]{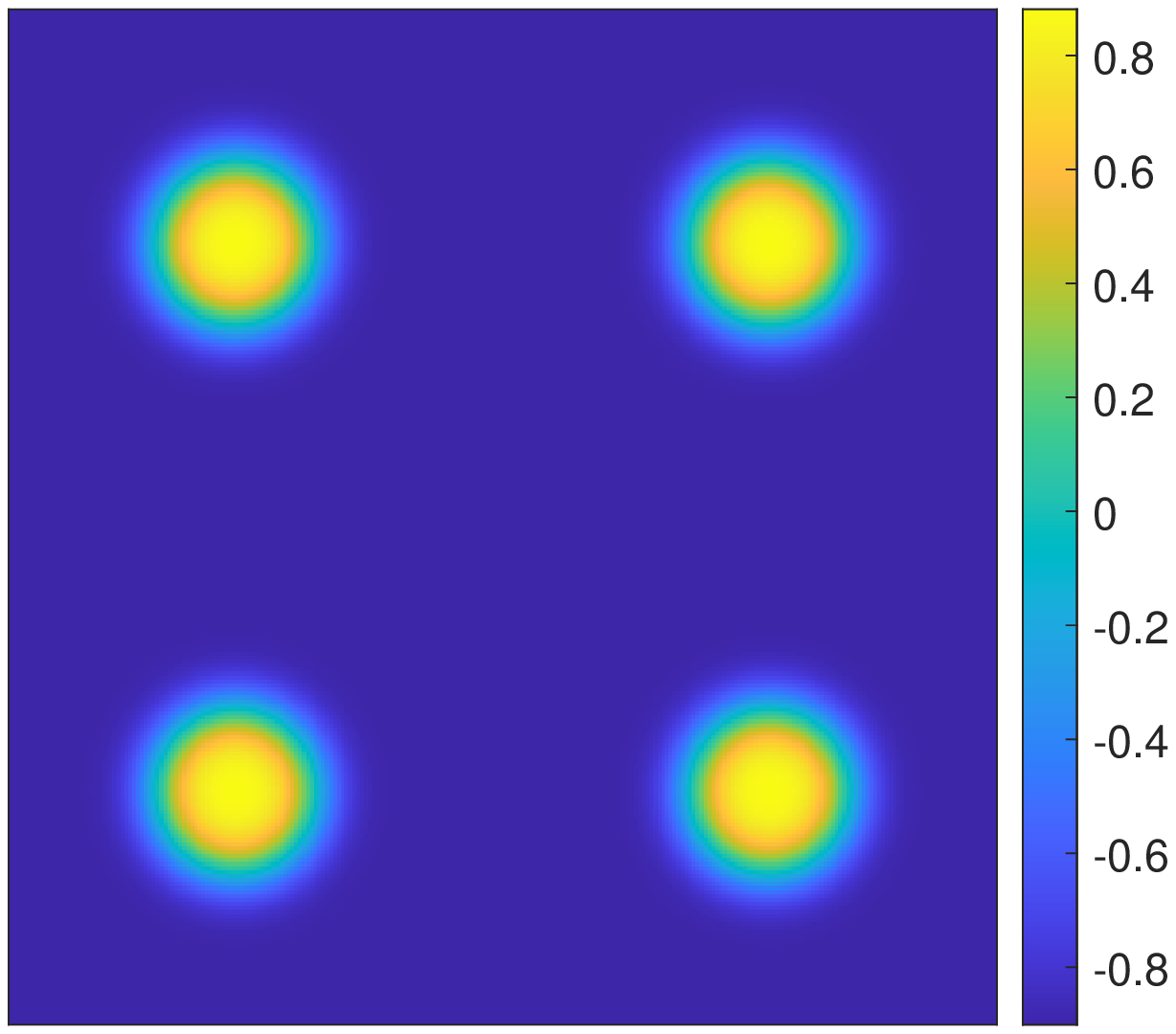}}
	\caption{Snapshots with initial condition \eqref{pearlinitial} when $\ell=0.35$. The ring splits near $t=0.3$ and converges to 4 circles.}
	\label{figpearl035}
\end{figure}

\begin{figure}[!t]
	\begin{center}
		\subfigure[$t=0$]{\includegraphics[scale=0.23,trim={30 0 30 0},clip]{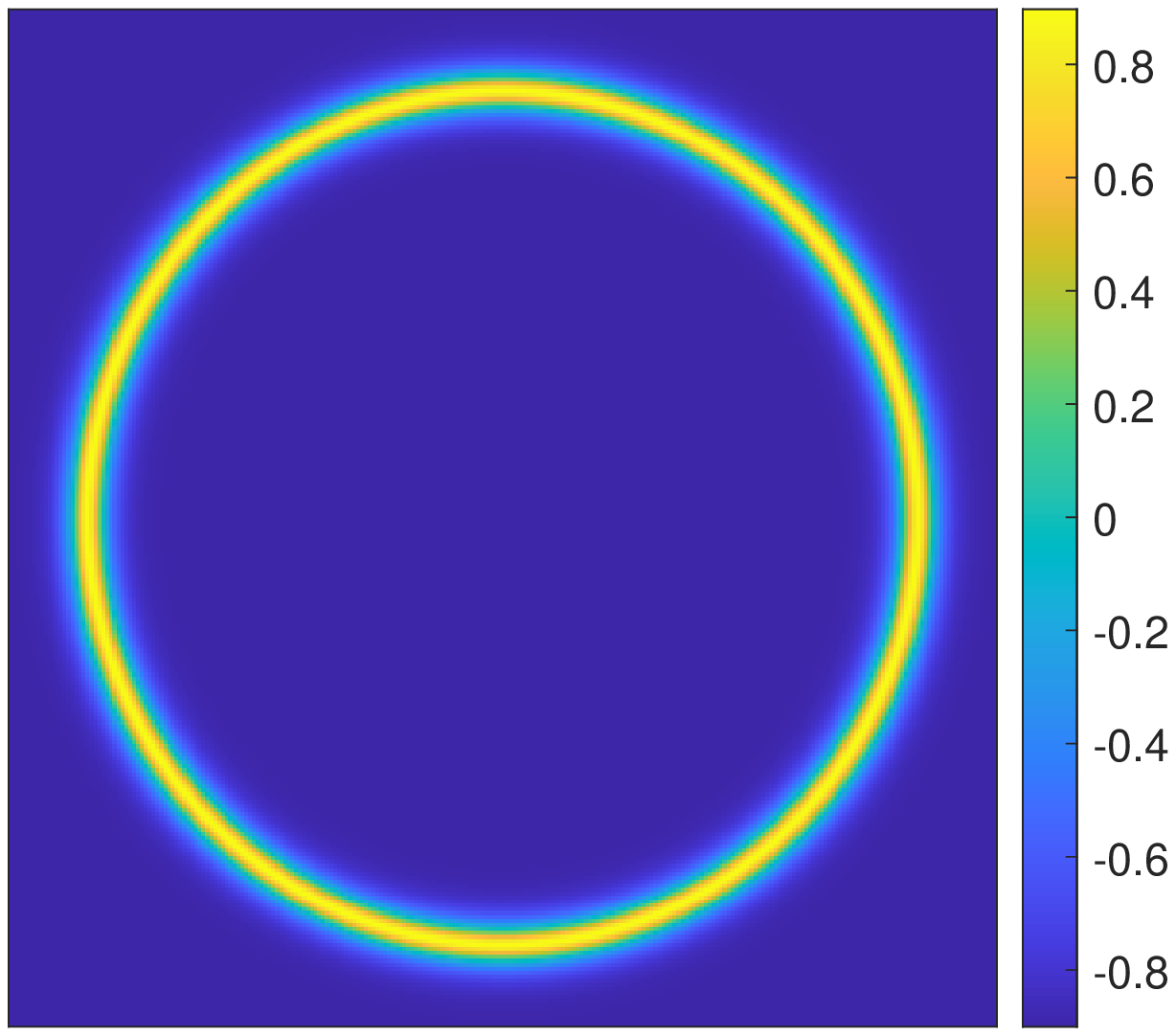}}
		\subfigure[$t=0.3$]{\includegraphics[scale=0.23,trim={30 0 30 0},clip]{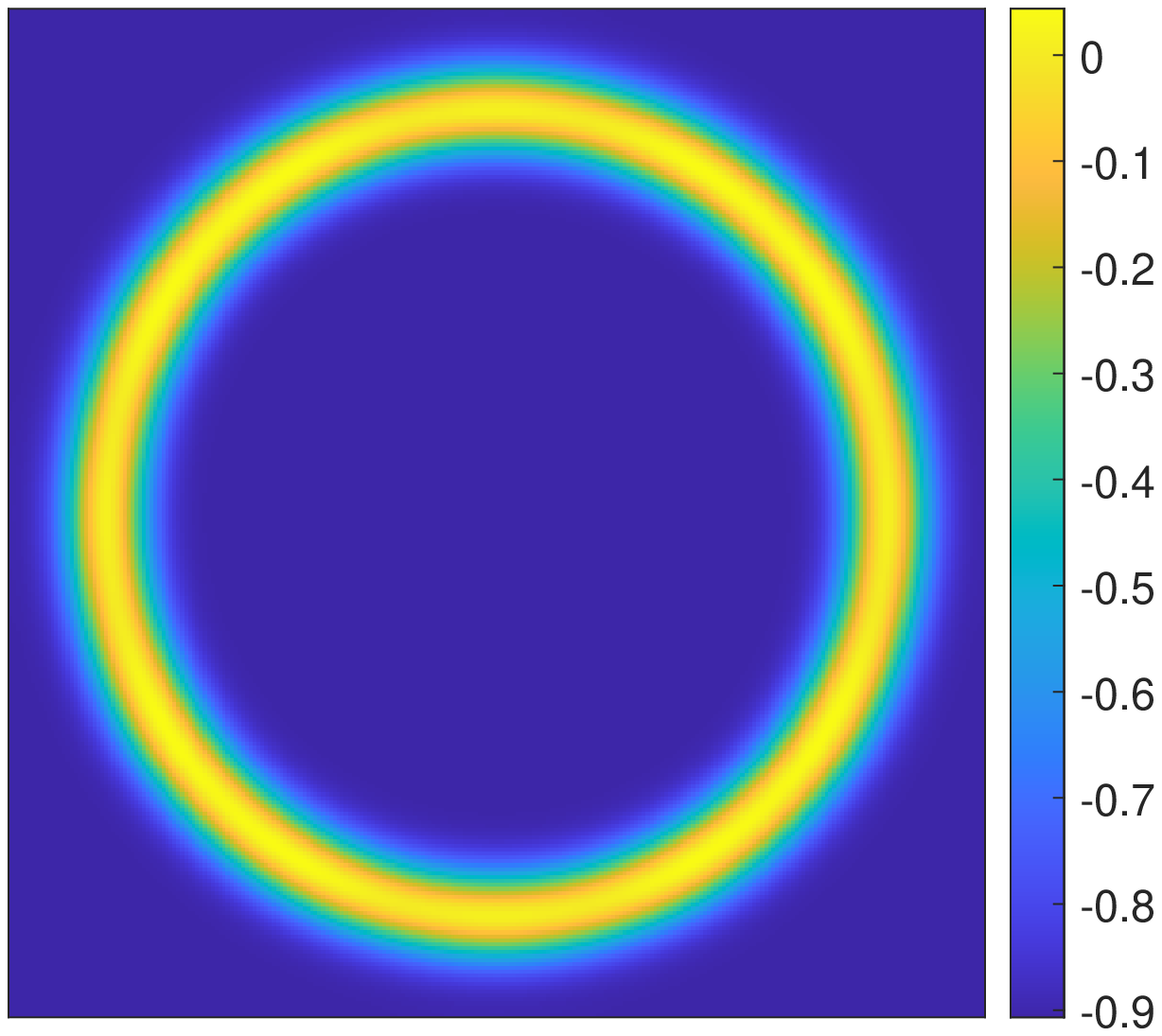}}
		\subfigure[$t=0.5$]{\includegraphics[scale=0.23,trim={30 0 30 0},clip]{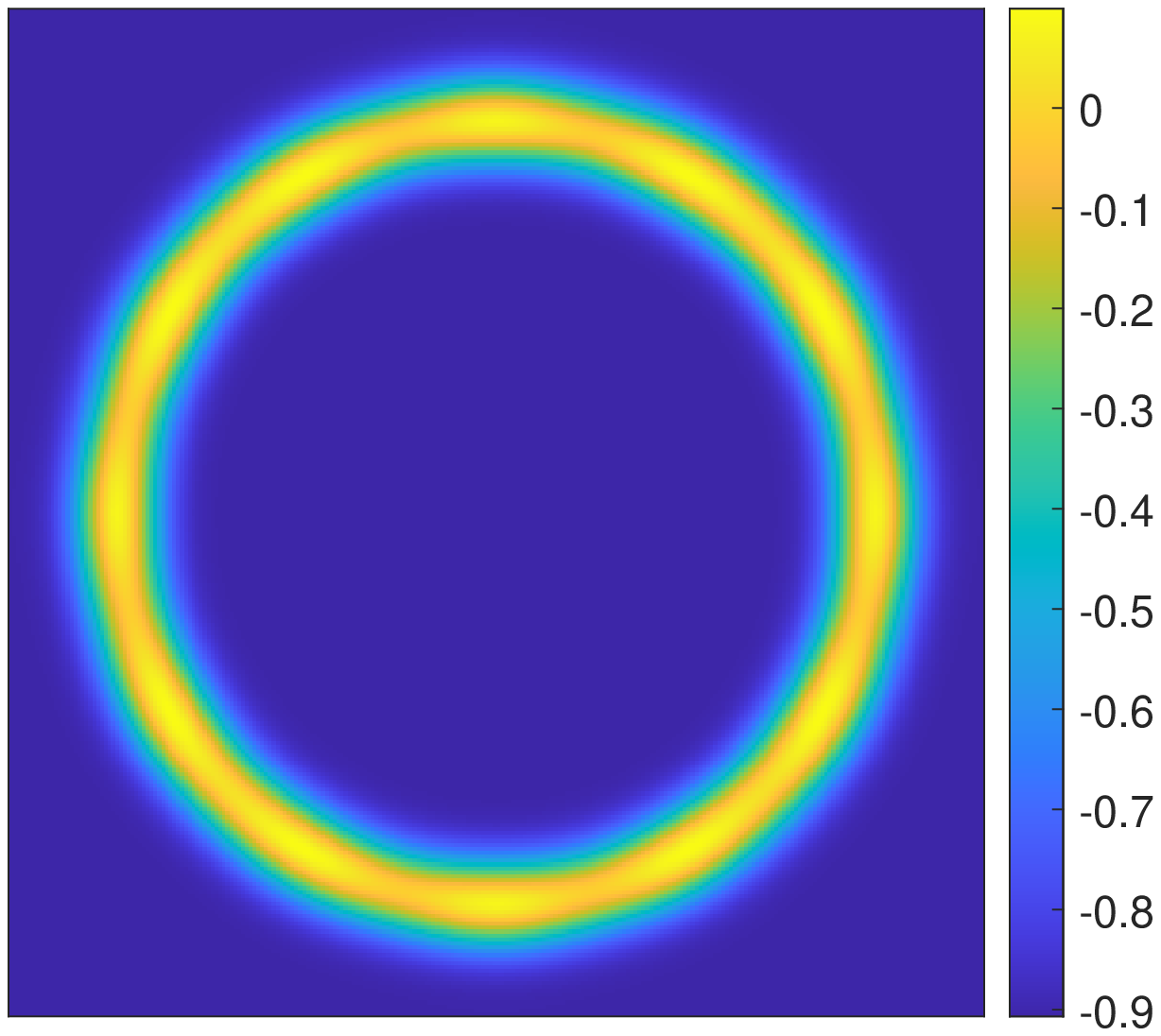}}
		\subfigure[$t=0.8$]{\includegraphics[scale=0.23,trim={30 0 30 0},clip]{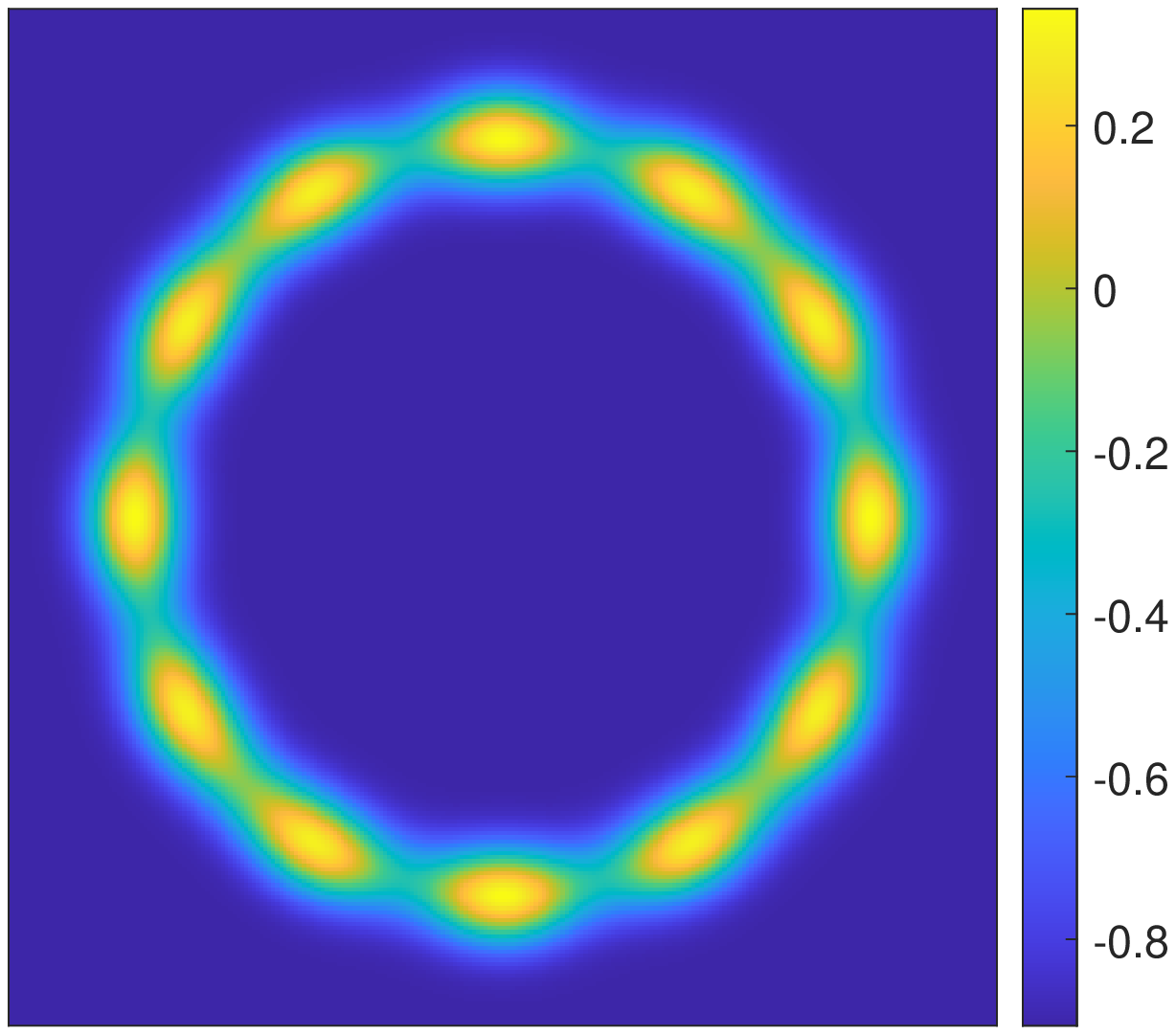}}
		\subfigure[$t=10$]{\includegraphics[scale=0.23,trim={30 0 30 0},clip]{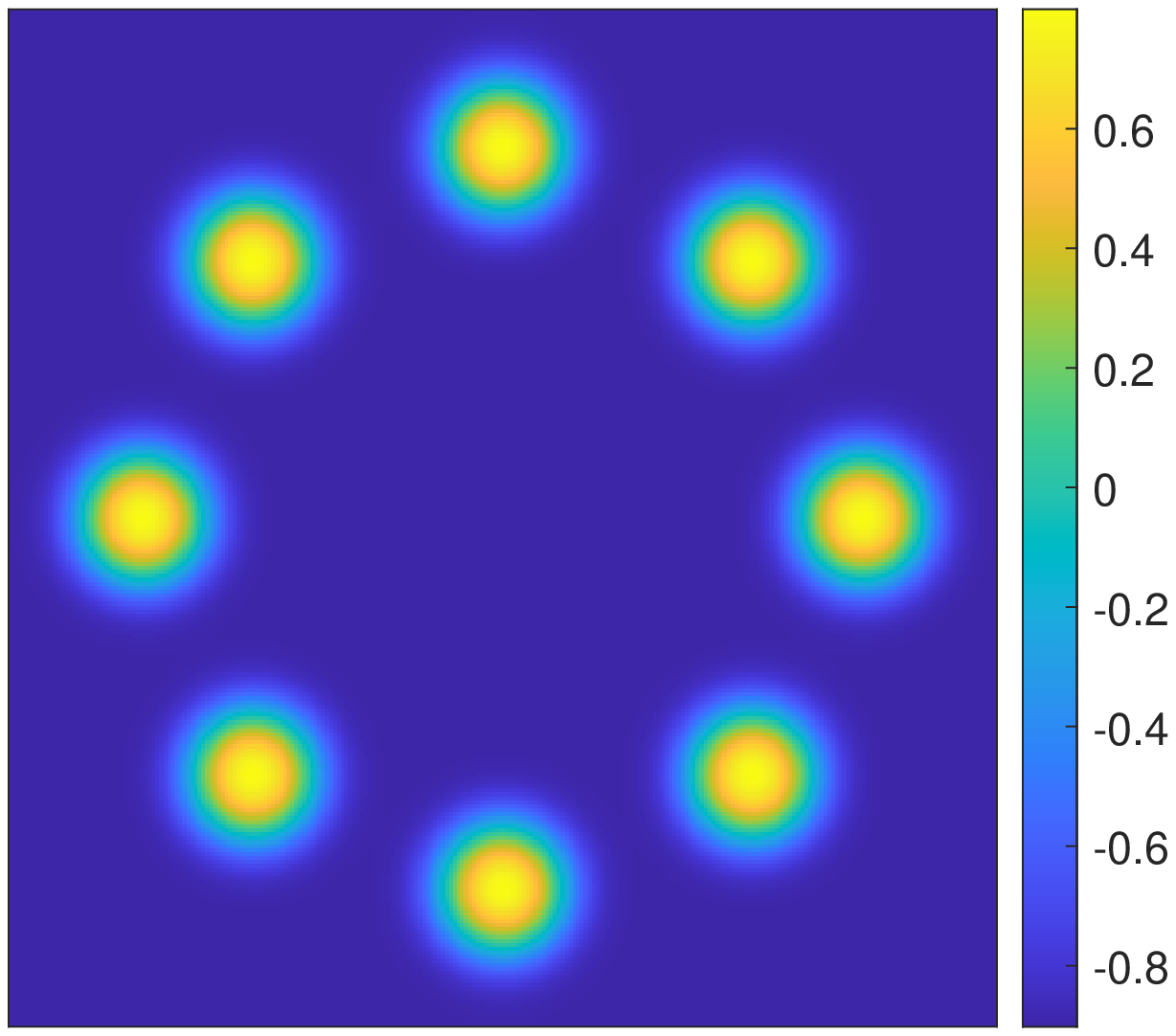}}
	\end{center}
	\caption{Snapshots with initial condition \eqref{pearlinitial} when $\ell=0.39$. The pearl bifurcation appears near $t=0.8$, and converges to 8 circles.}
	\label{figpearl039}
\end{figure}

\begin{figure}[!t]
	\centering
	\subfigure[$t=0$]{\includegraphics[scale=0.23,trim={30 0 30 0},clip]{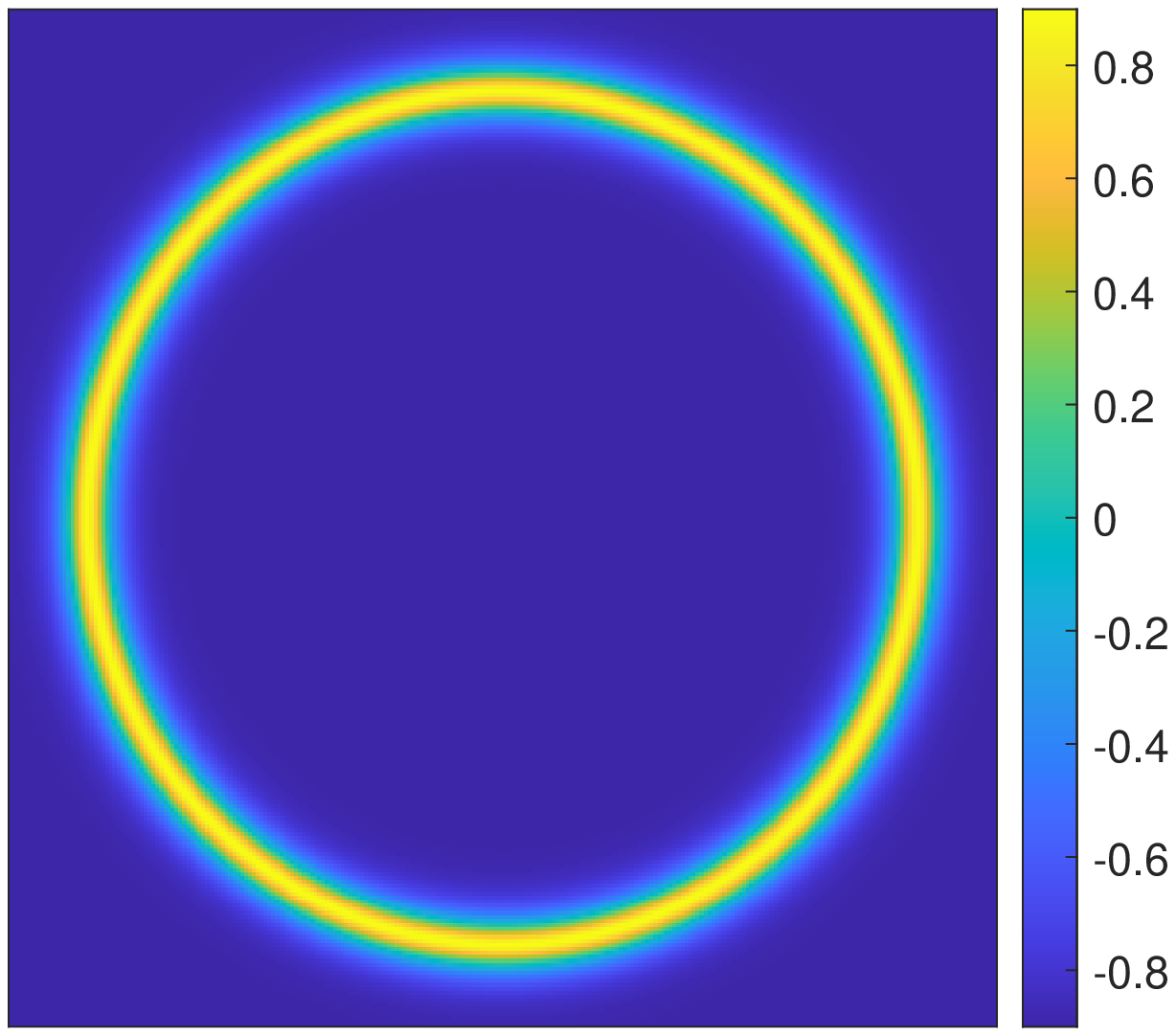}}
	\subfigure[$t=0.3$]{\includegraphics[scale=0.23,trim={30 0 30 0},clip]{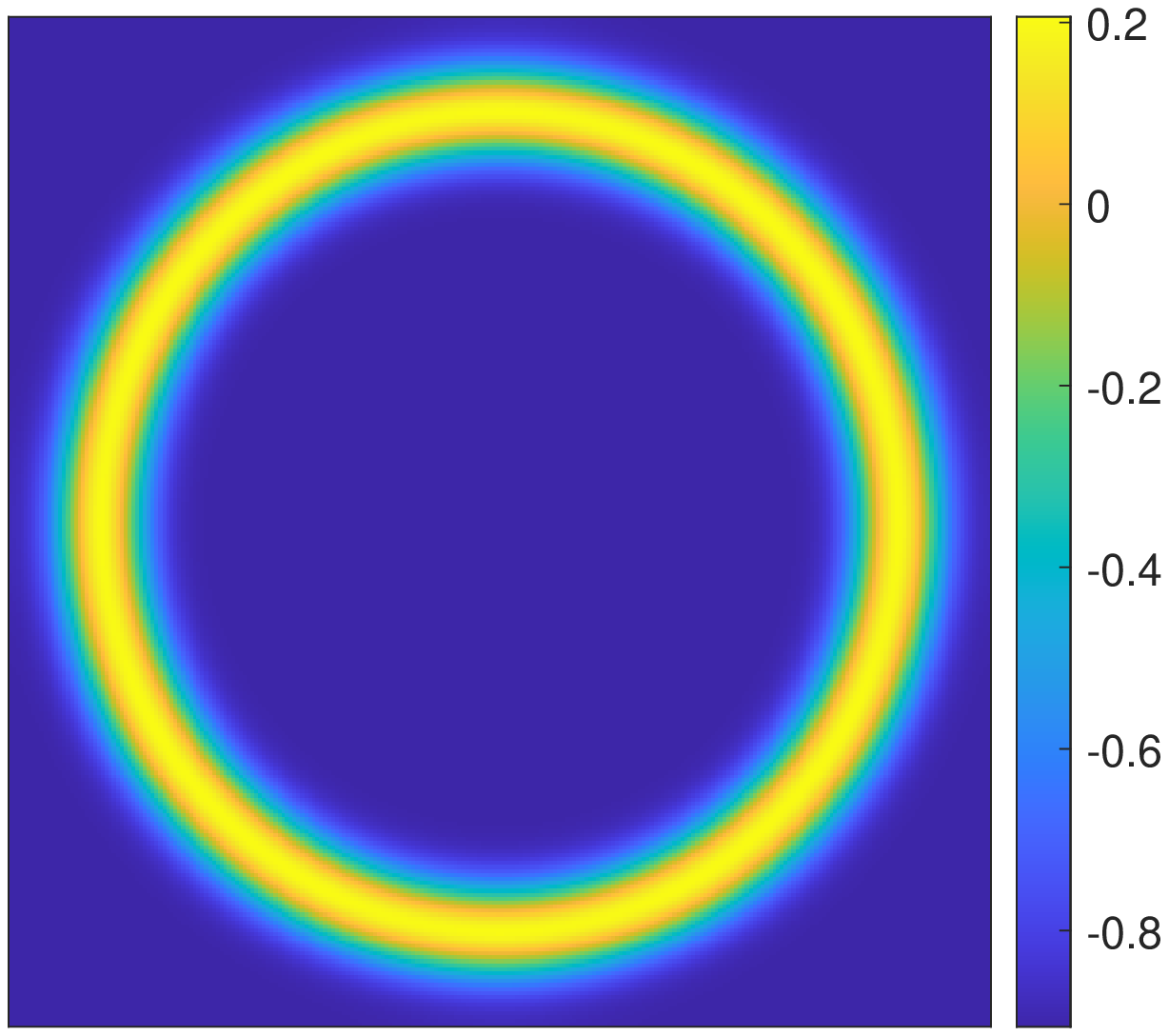}}
	\subfigure[$t=0.5$]{\includegraphics[scale=0.23,trim={30 0 30 0},clip]{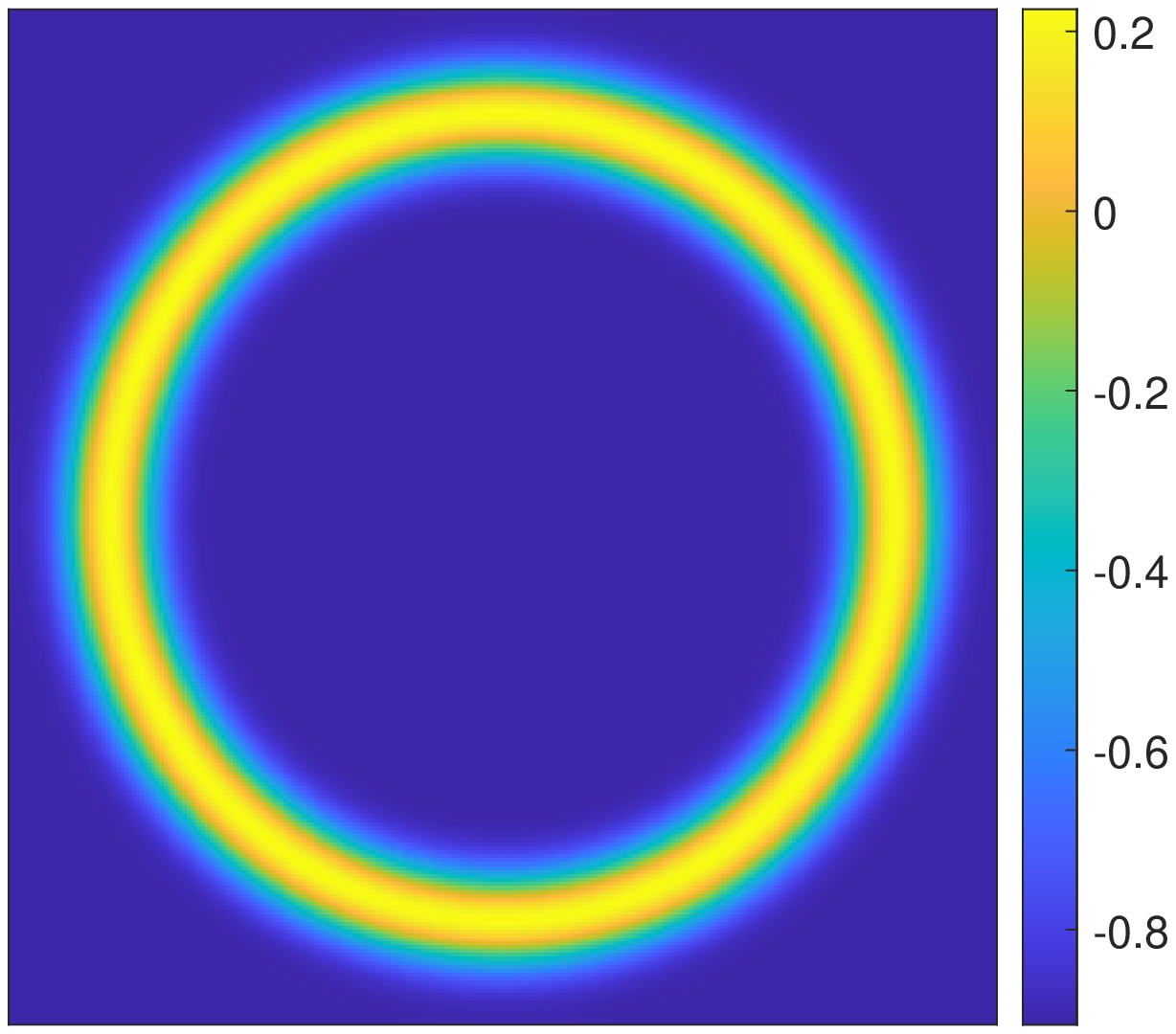}}
	\subfigure[$t=0.8$]{\includegraphics[scale=0.23,trim={30 0 30 0},clip]{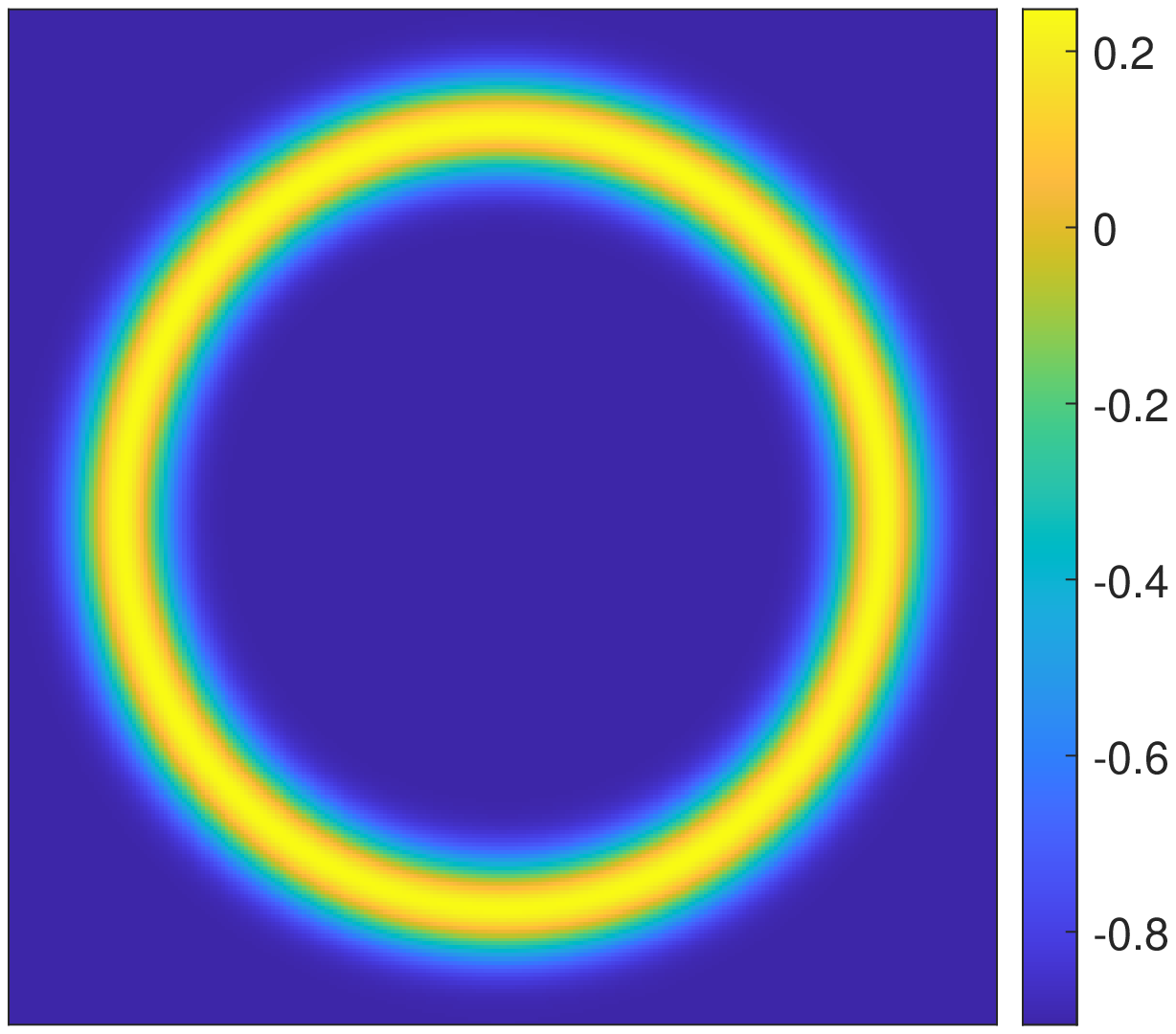}}
	\subfigure[$t=10$]{\includegraphics[scale=0.23,trim={30 0 30 0},clip]{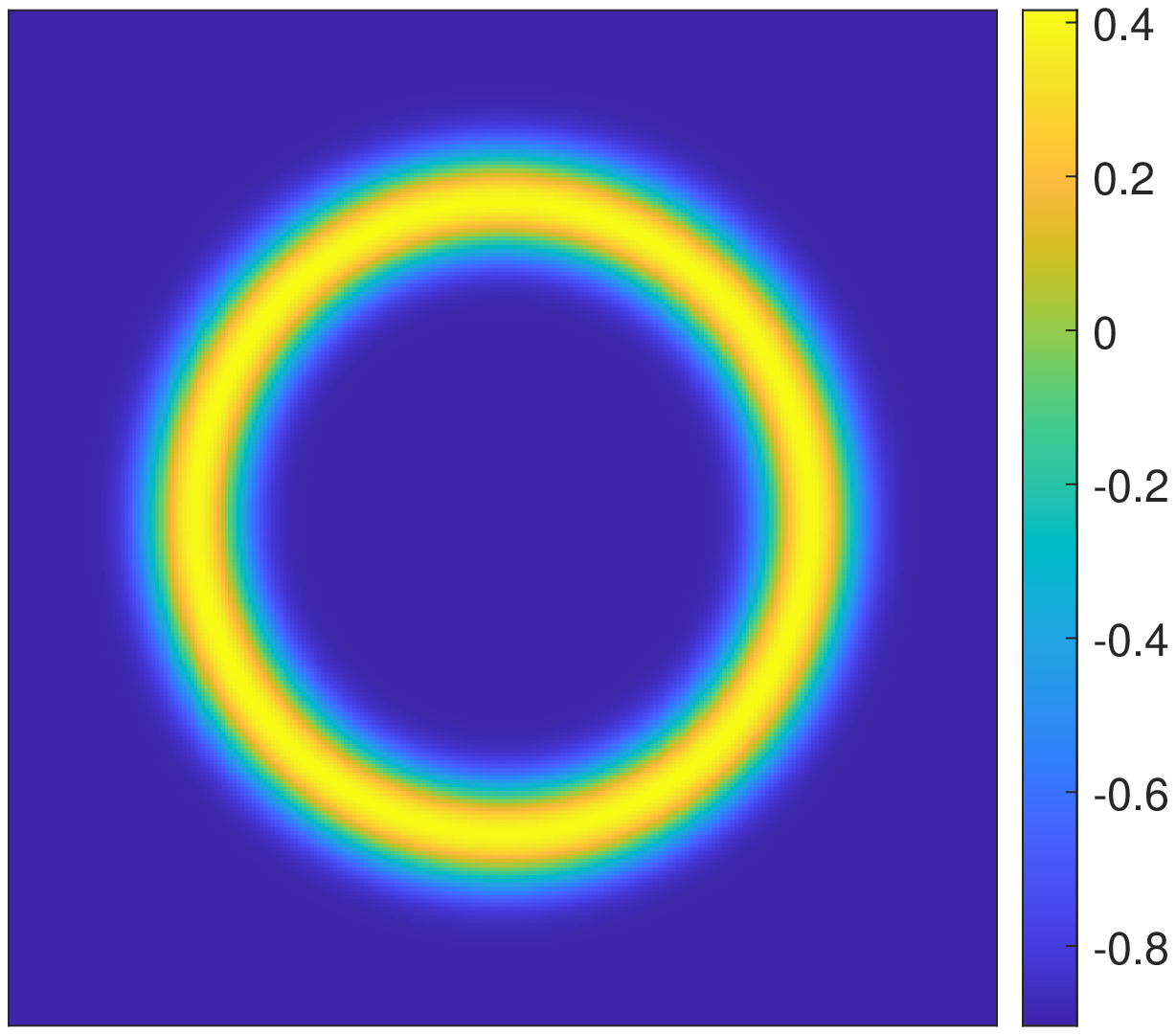}}
	\caption{Snapshots with initial condition \eqref{pearlinitial} when $\ell=0.5$. The ring doesn't split but only shrink and thicken. The pearl structure fails to exist the whole time.}
	\label{figpearl05}
\end{figure}

\begin{figure}[!t]
	\begin{center}
		\begin{overpic}[scale=.45]{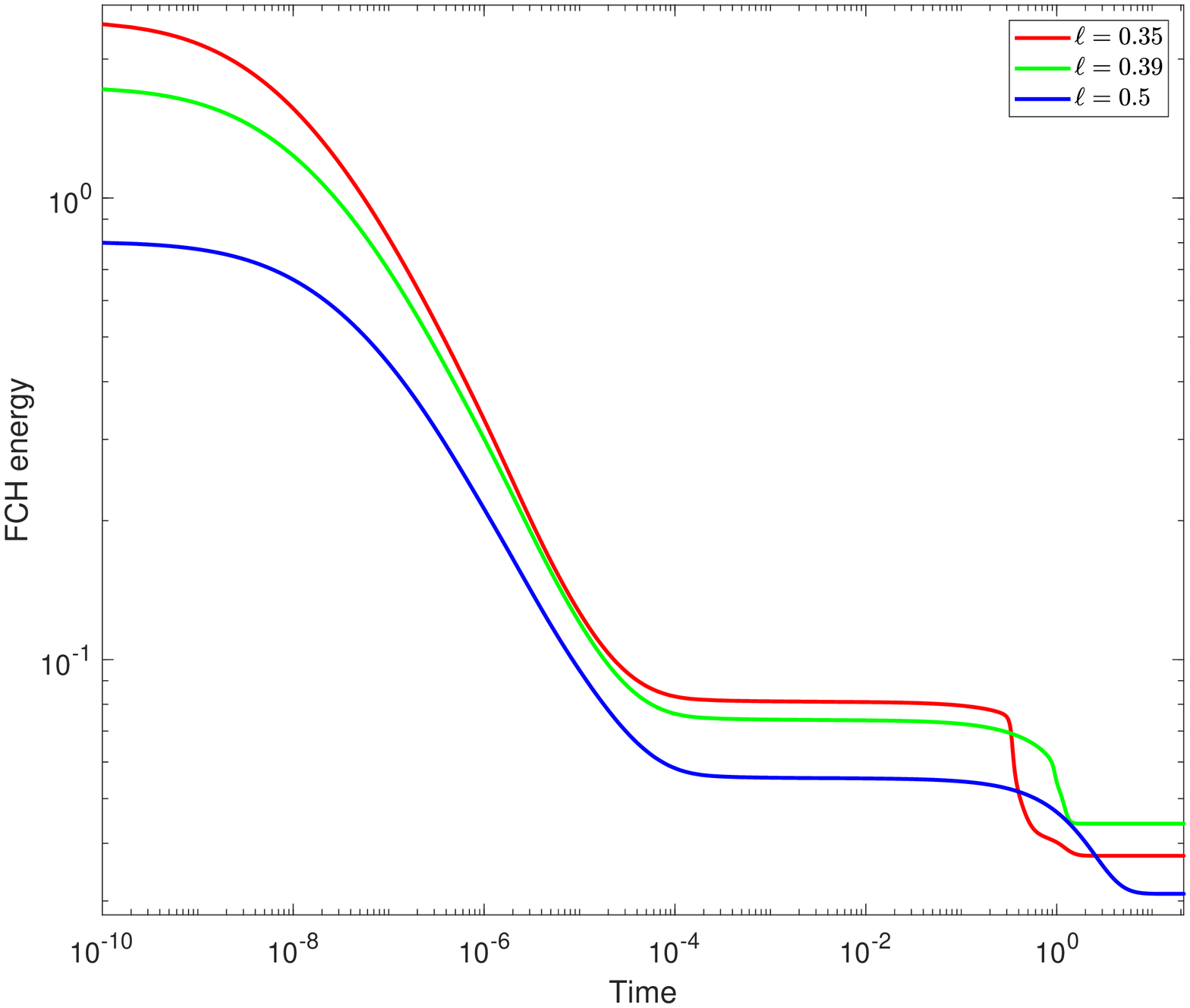}
			\put(48,30){\includegraphics[scale=.2,trim={30 20 30 20},clip]%
				{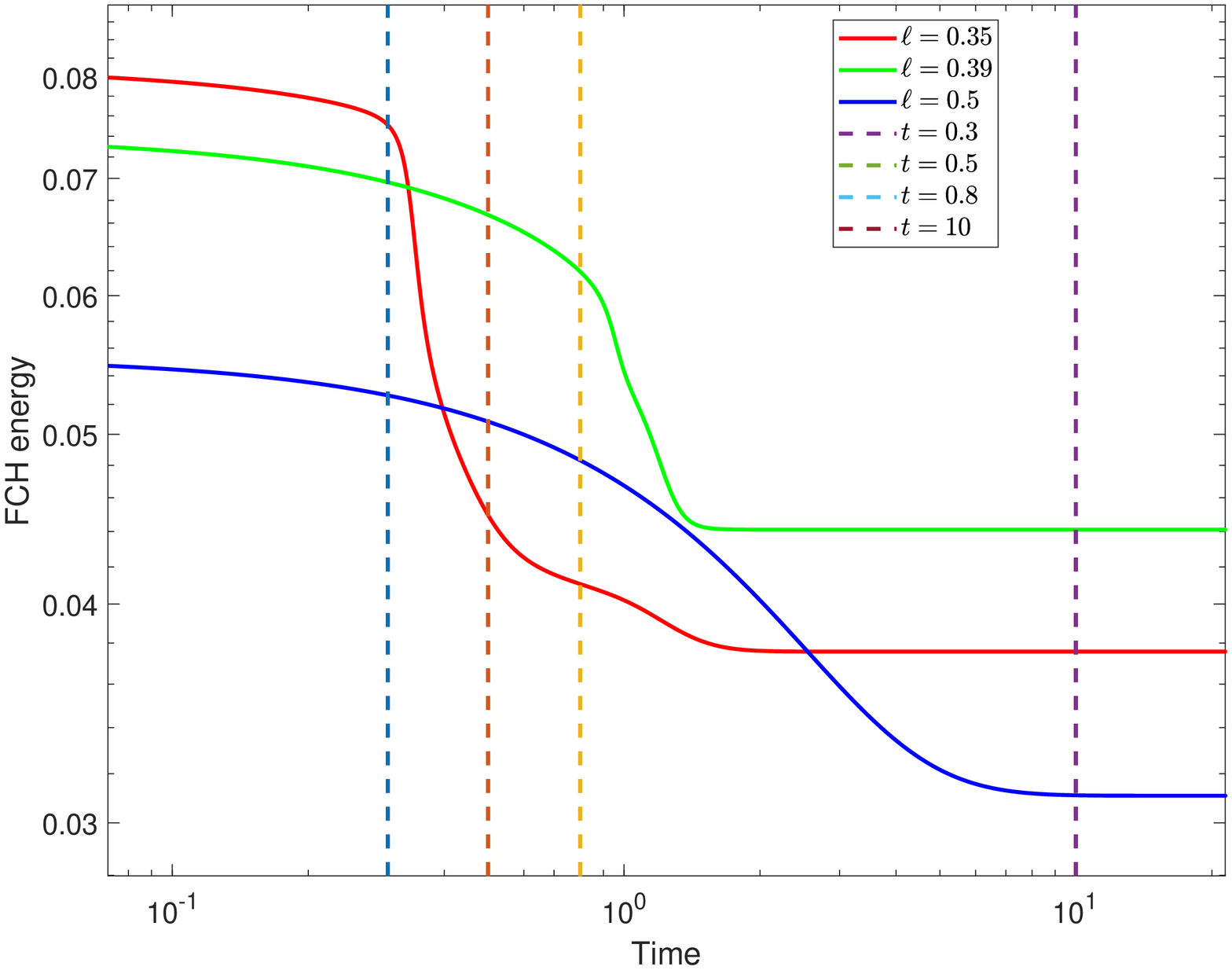}}
		\end{overpic}
	\end{center}
	\caption{Loglog plot of the FCH energy in three different cases. Four time ticks of the snapshots in Figure \ref{figpearl035}-\ref{figpearl05} are marked by the dotted lines. All cases encounter a fast exponential energy decay at the beginning of the simulation. When the pearl bifurcation appears, the energy also changes rapidly. }
	\label{pearlenergy}
\end{figure}

\begin{figure}[!t]
	\begin{flushleft}
		\begin{overpic}[scale=.4]{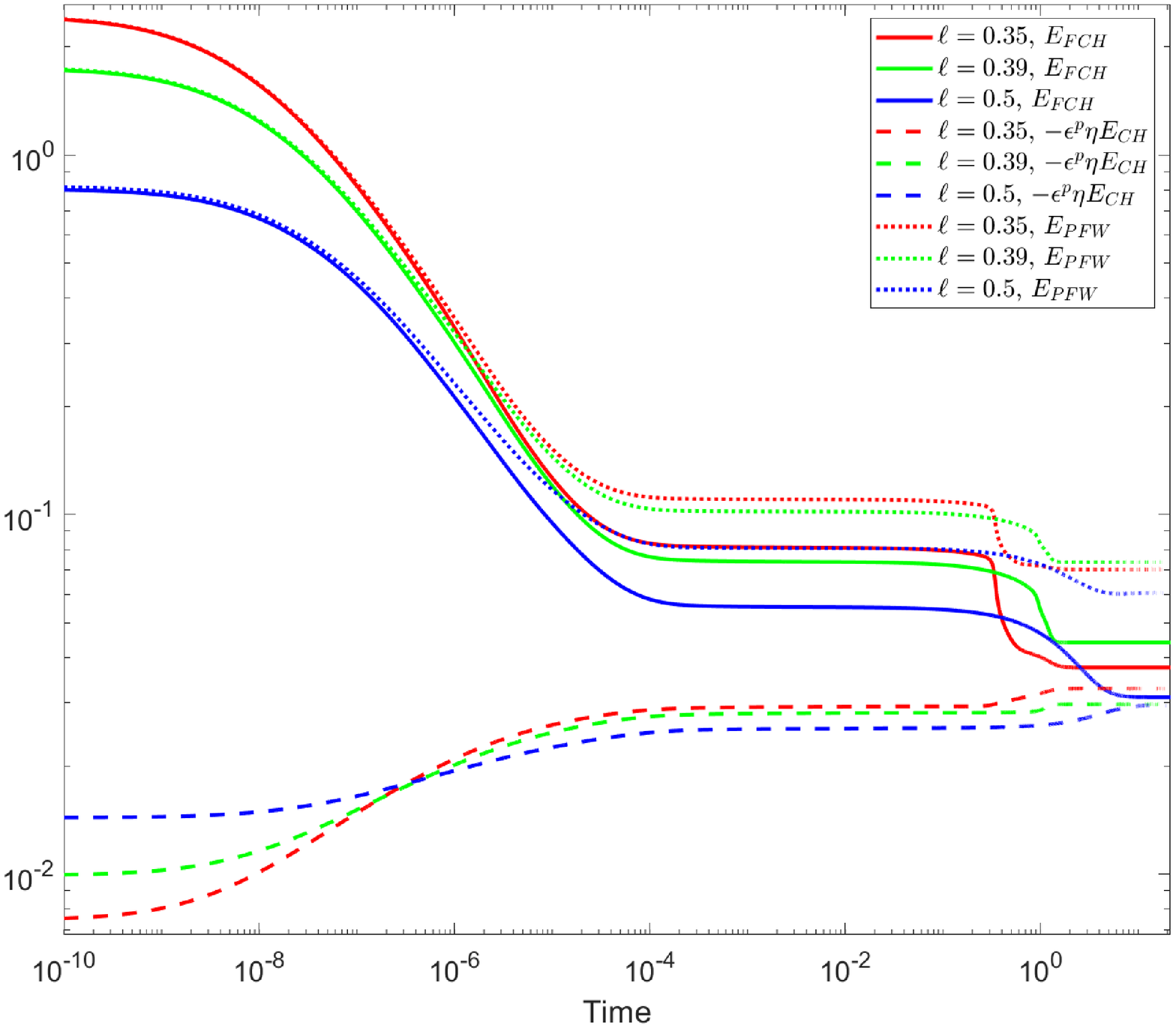}
			\put(91,12){\includegraphics[scale=.3, trim={0 0 210 0}, clip]%
				{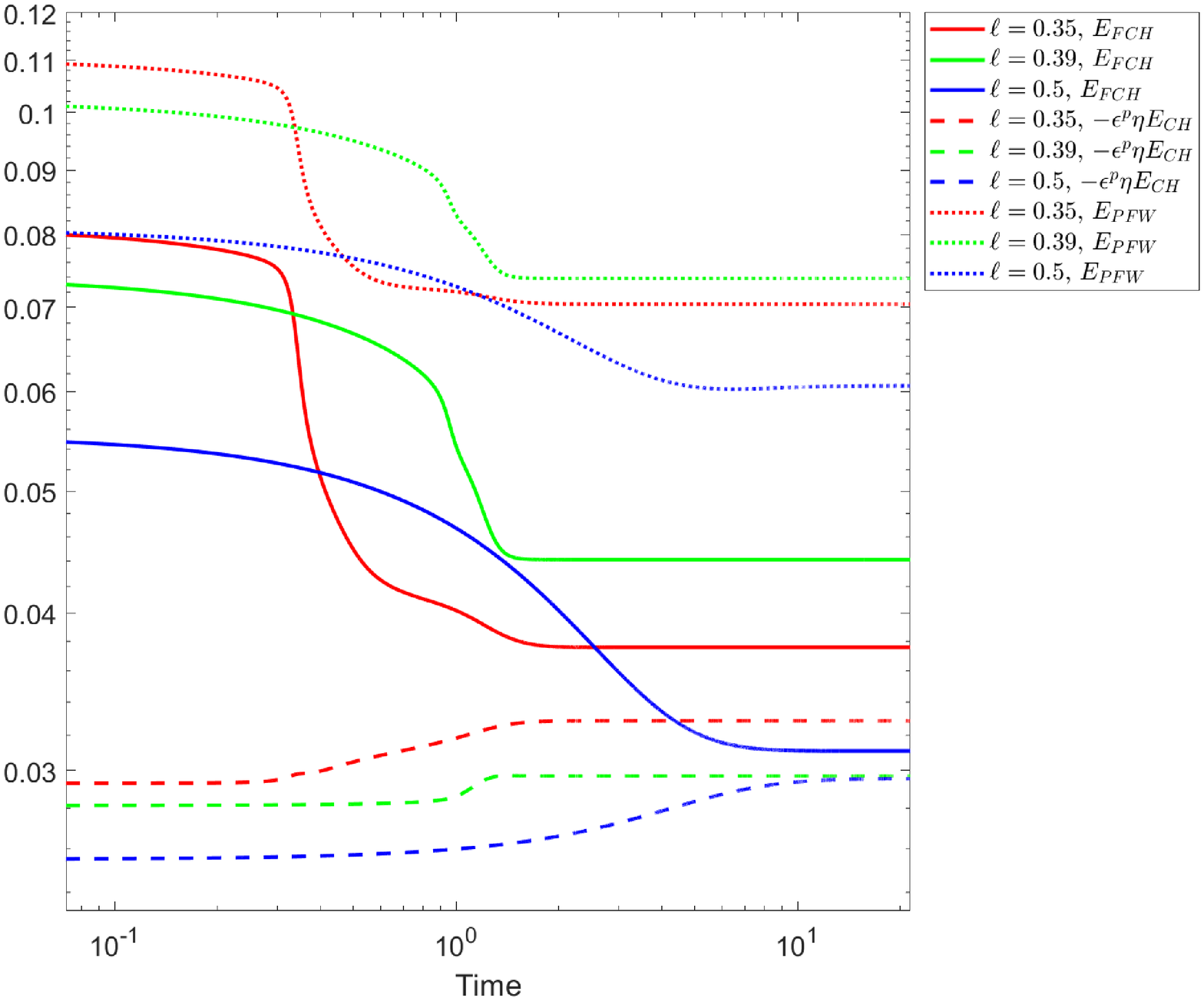}}
		\end{overpic}
	\end{flushleft}
	\caption{Loglog plot of the FCH energy, the minus Cahn-Hilliard energy and the phase-field Willmore (PFW) energy in three different cases. We observe that both the PFW energy and the Cahn-Hilliard energy decrease during time evolution for pearling bifurcation. }
	\label{totalenergy}
\end{figure}

\subsubsection{The meandering instability}

The meandering instability is another interesting phenomenon investigated in the literature, see e.g.,  \cite{Doelman2014}. Some numerical tests can be found in \cite{feng2018,Zhang2021analysis} with a regular potential. To proceed with our simulation, we consider the following initial condition:
\begin{equation}\label{meanderinitial}
	\phi(x,y,0)=\left\{
	\begin{aligned}
		&-0.9,\quad \mbox{if}\ y>0.5\sin\left(\frac{4\pi x}{12}\right) + 6.4,\\
		&-0.9,\quad \mbox{if}\ y<0.5\sin\left(\frac{4\pi x}{15}\right)-5.6,\\
		&0.9,\qquad\  \mbox{otherwise}.
	\end{aligned}\right.
\end{equation}
Besides, we set $\eta=10$, $\epsilon=0.01$, $p=1$, $h=1/256$ and $\lambda= \ln(19)/0.9$. Figure \ref{meander} shows the time snapshots of the meandering instability, and Figure \ref{meanderenergy1} presents the FCH and Cahn-Hilliard energy change of this experiment. The FCH energy still decreases over time, however, the Cahn-Hilliard energy increases after the initial decay. This reflects the balancing between the two parts (PFW vs. Cahn-Hilliard) in the FCH energy.

\begin{figure}[!t]
	\centering
	\subfigure[$t=0$]{\includegraphics[scale=0.25,trim={30 0 30 0},clip]{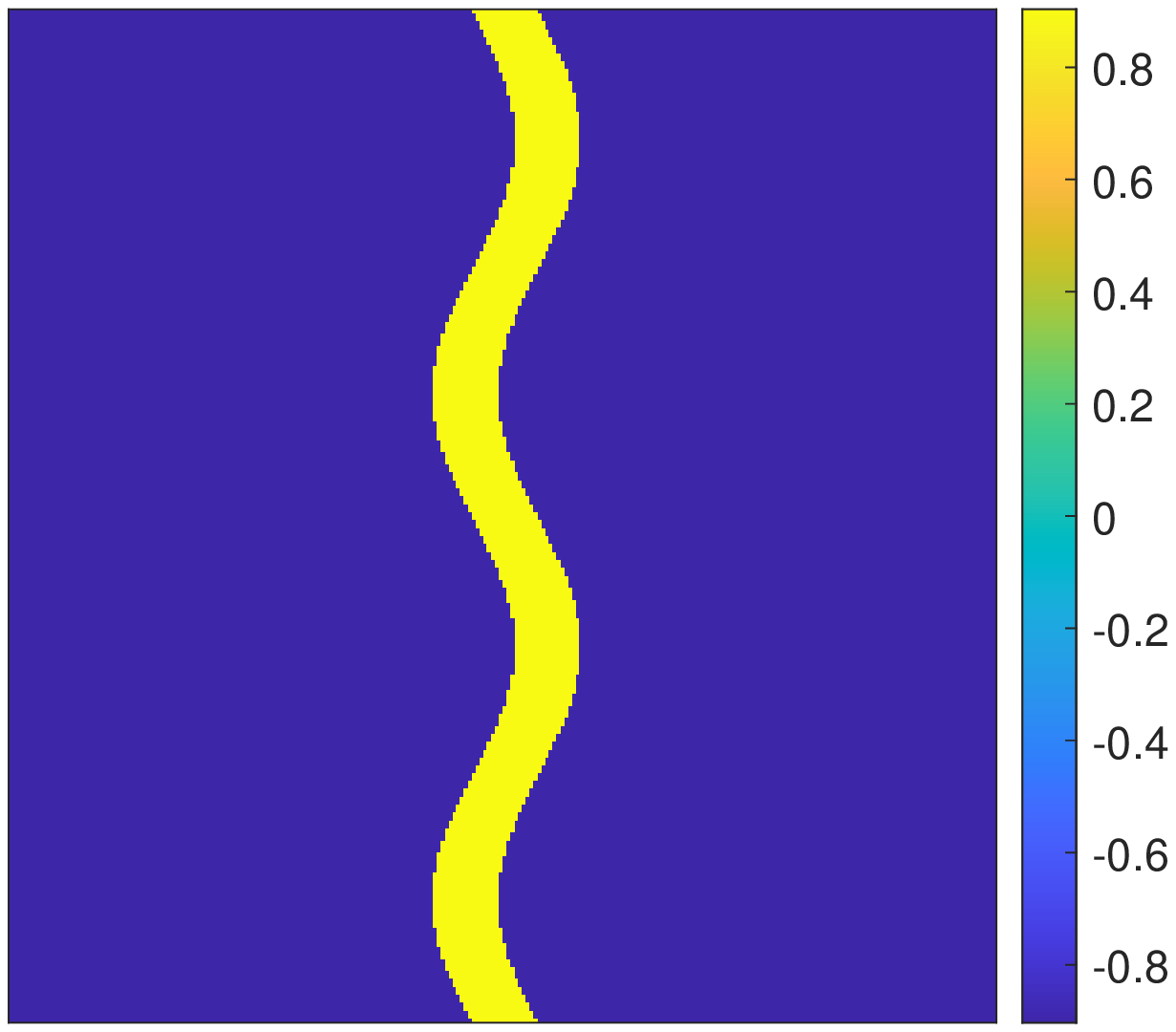}}
	\subfigure[$t=1$]{\includegraphics[scale=0.25,trim={30 0 30 0},clip]{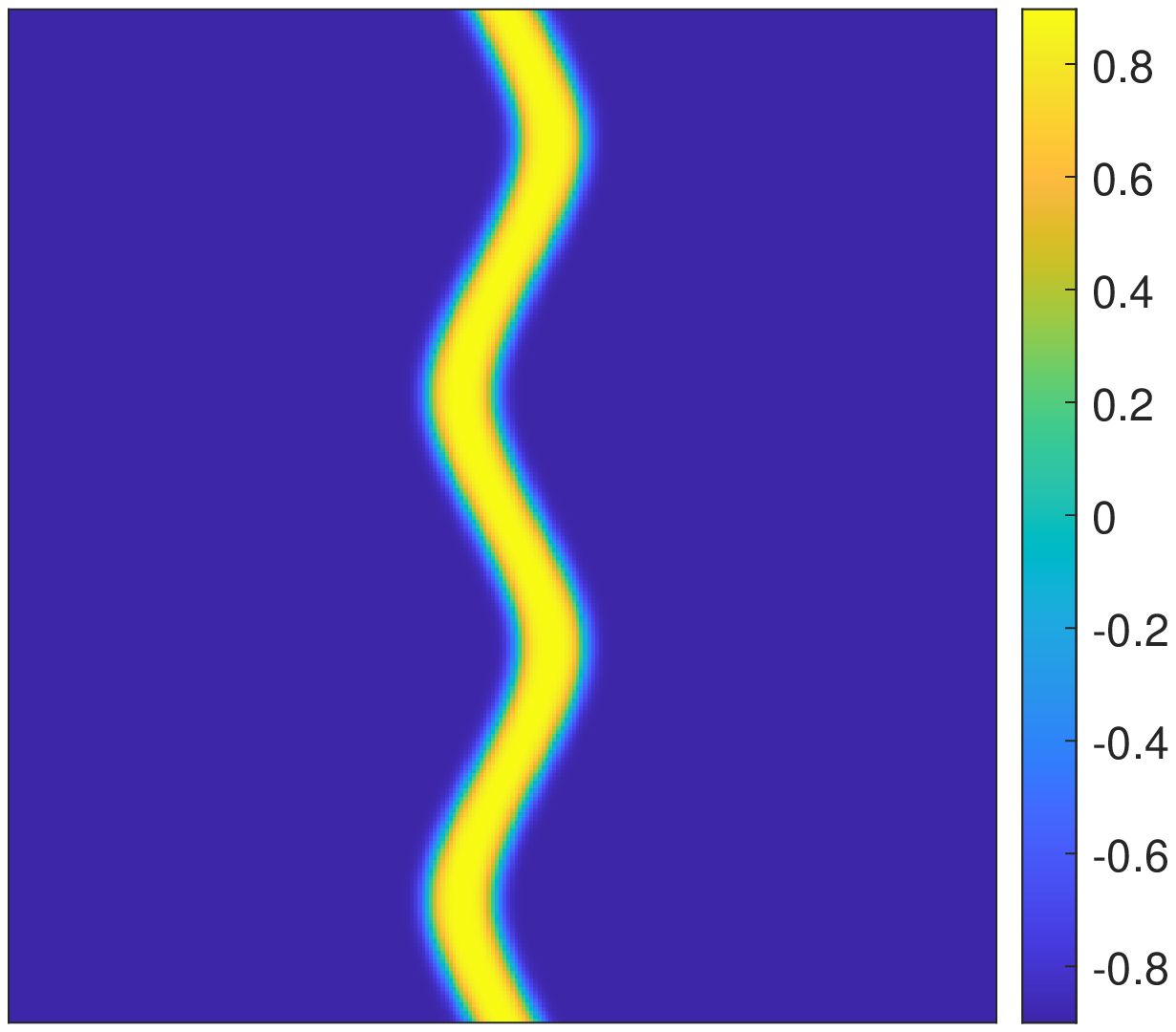}}
	\subfigure[$t=5$]{\includegraphics[scale=0.25,trim={30 0 30 0},clip]{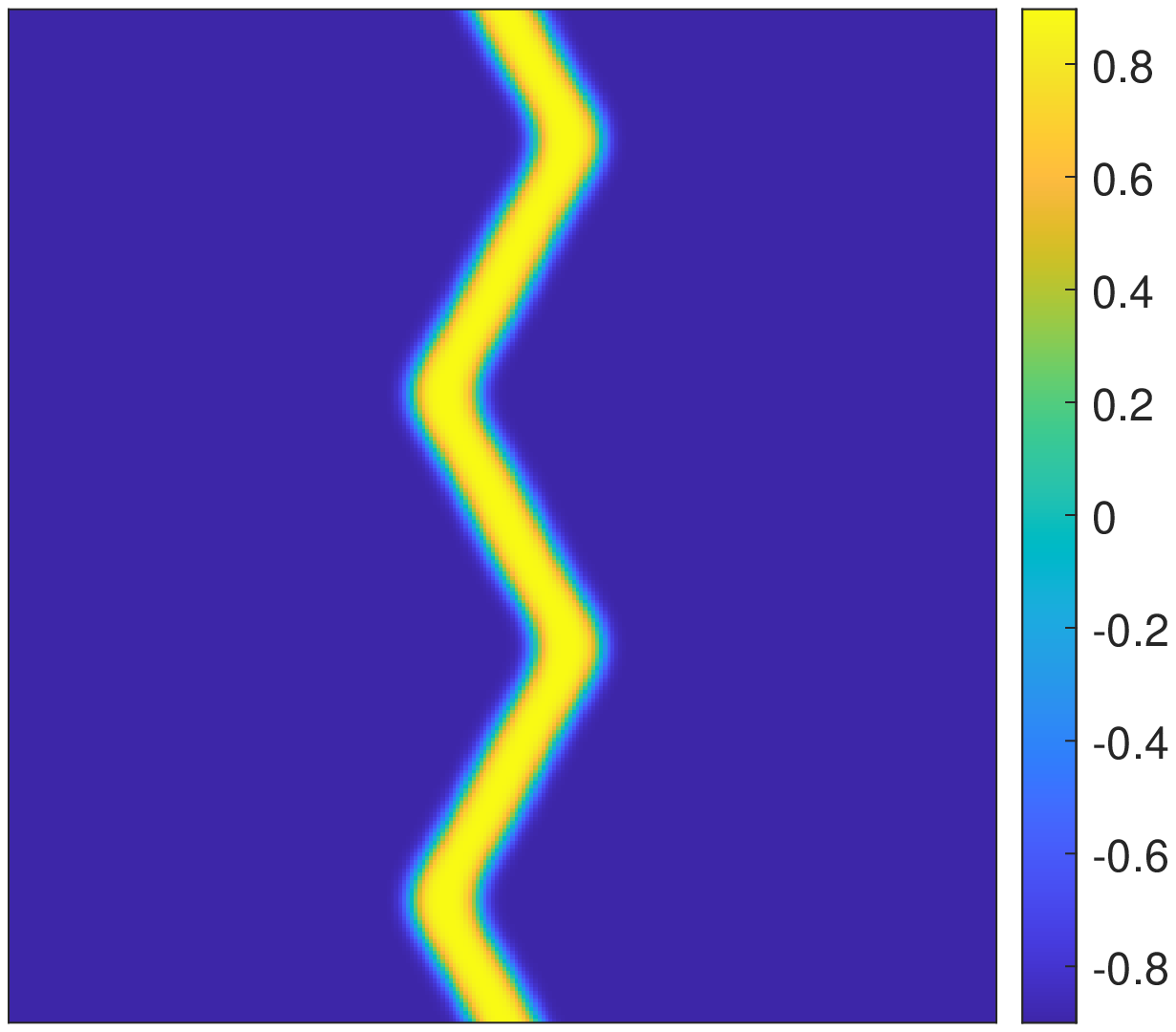}}
	\subfigure[$t=15$]{\includegraphics[scale=0.25,trim={30 0 30 0},clip]{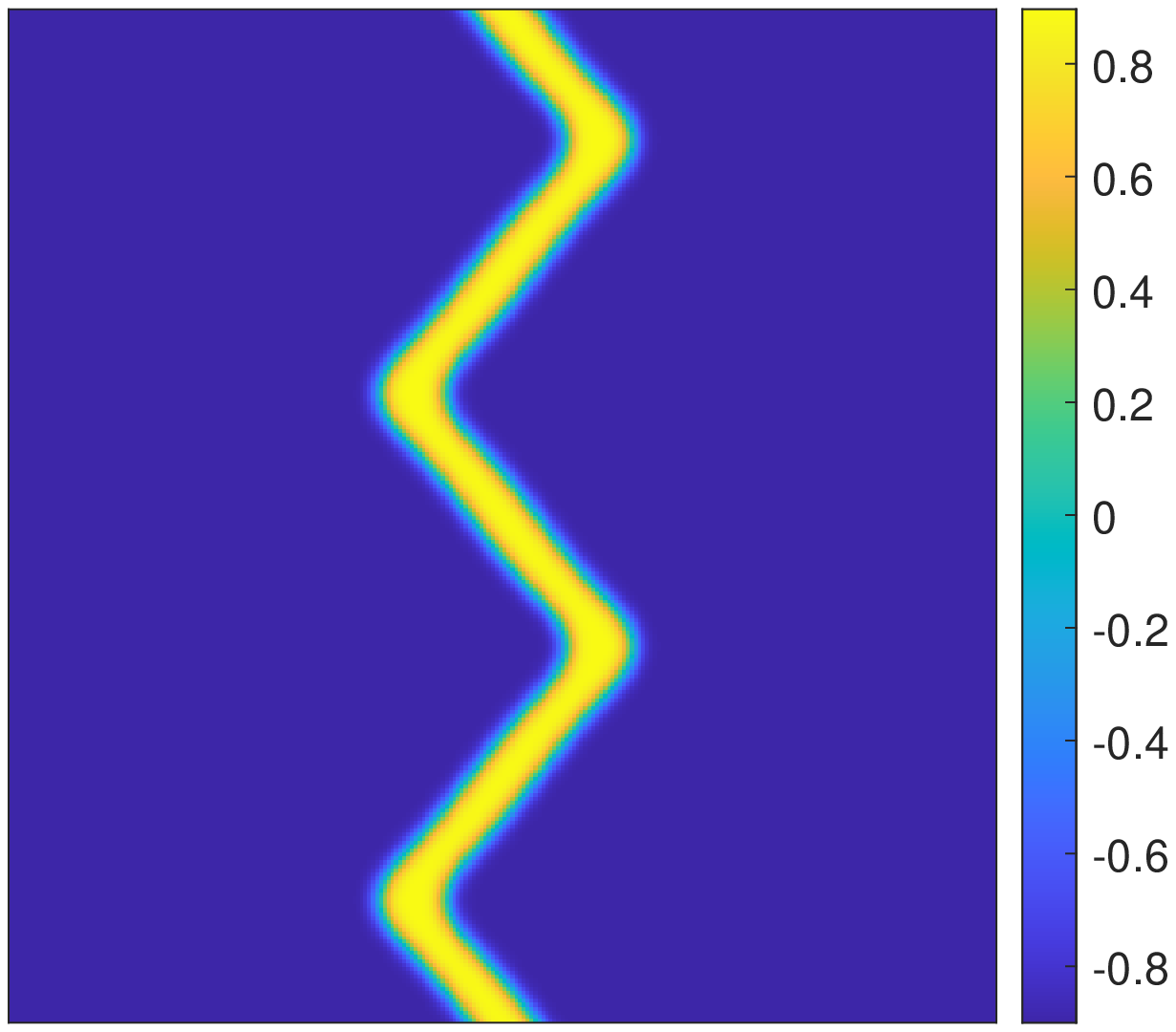}}
	\subfigure[$t=30$]{\includegraphics[scale=0.25,trim={30 0 30 0},clip]{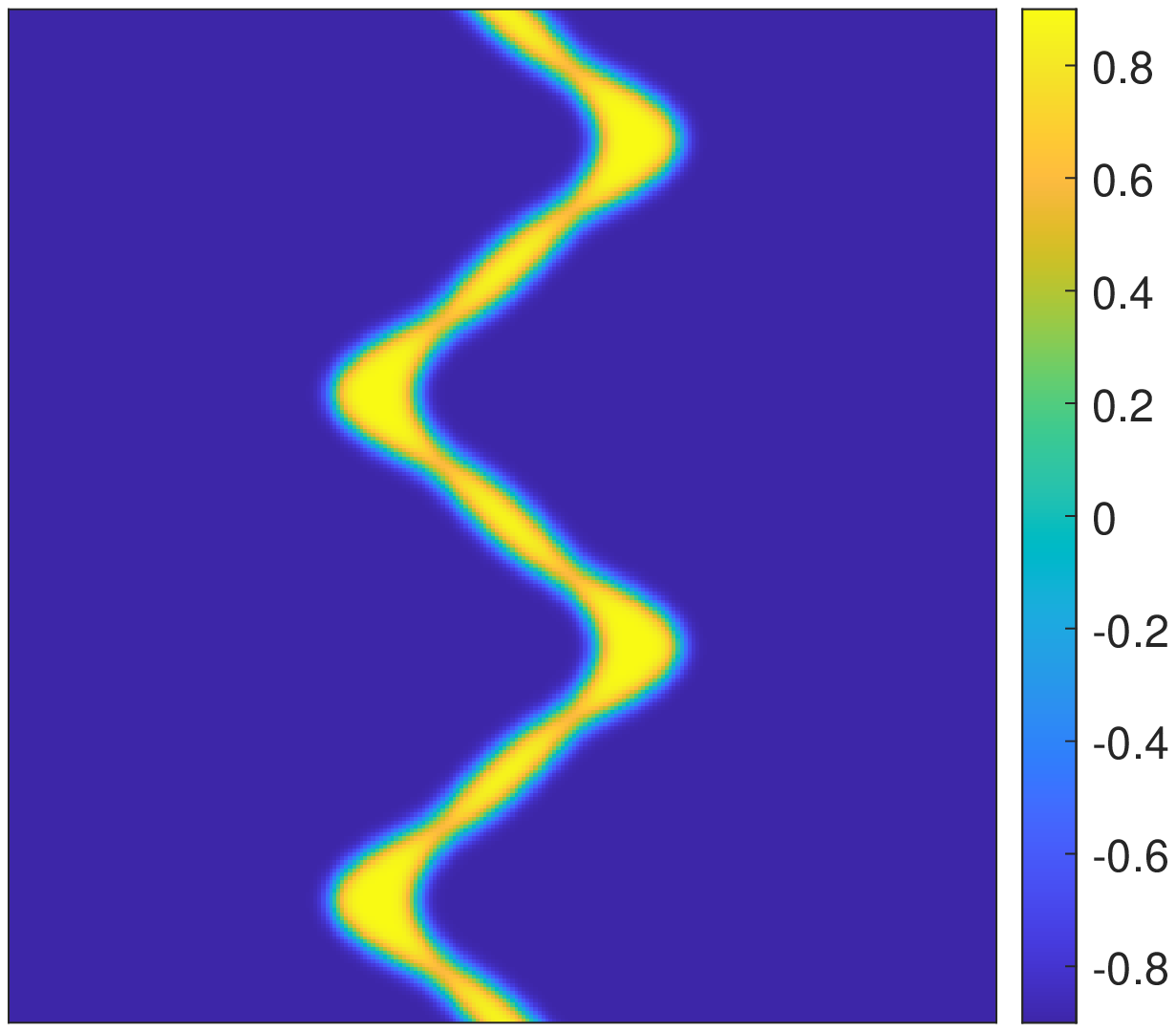}}
	\subfigure[$t=50$]{\includegraphics[scale=0.25,trim={30 0 30 0},clip]{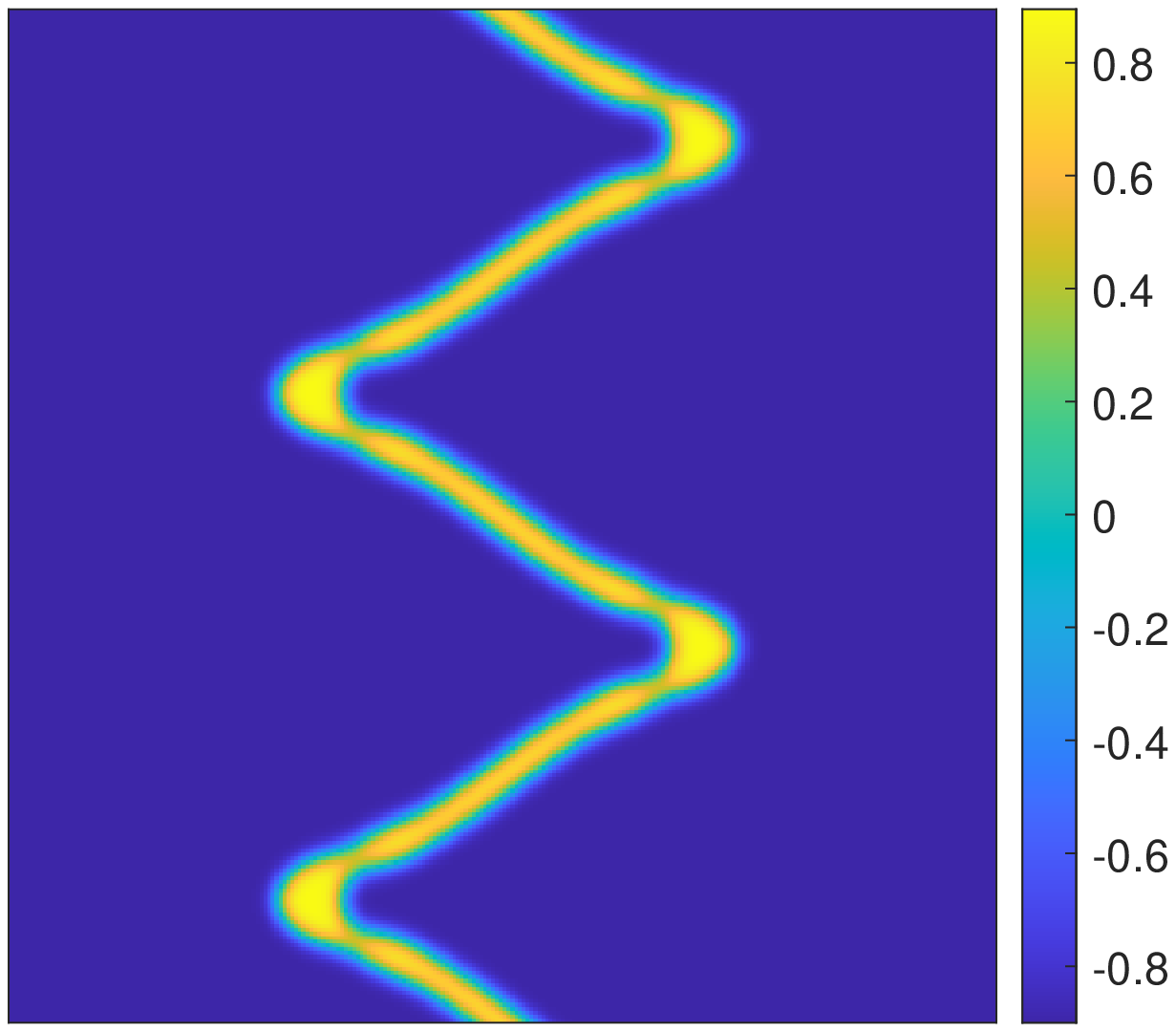}}
	\subfigure[$t=60$]{\includegraphics[scale=0.25,trim={30 0 30 0},clip]{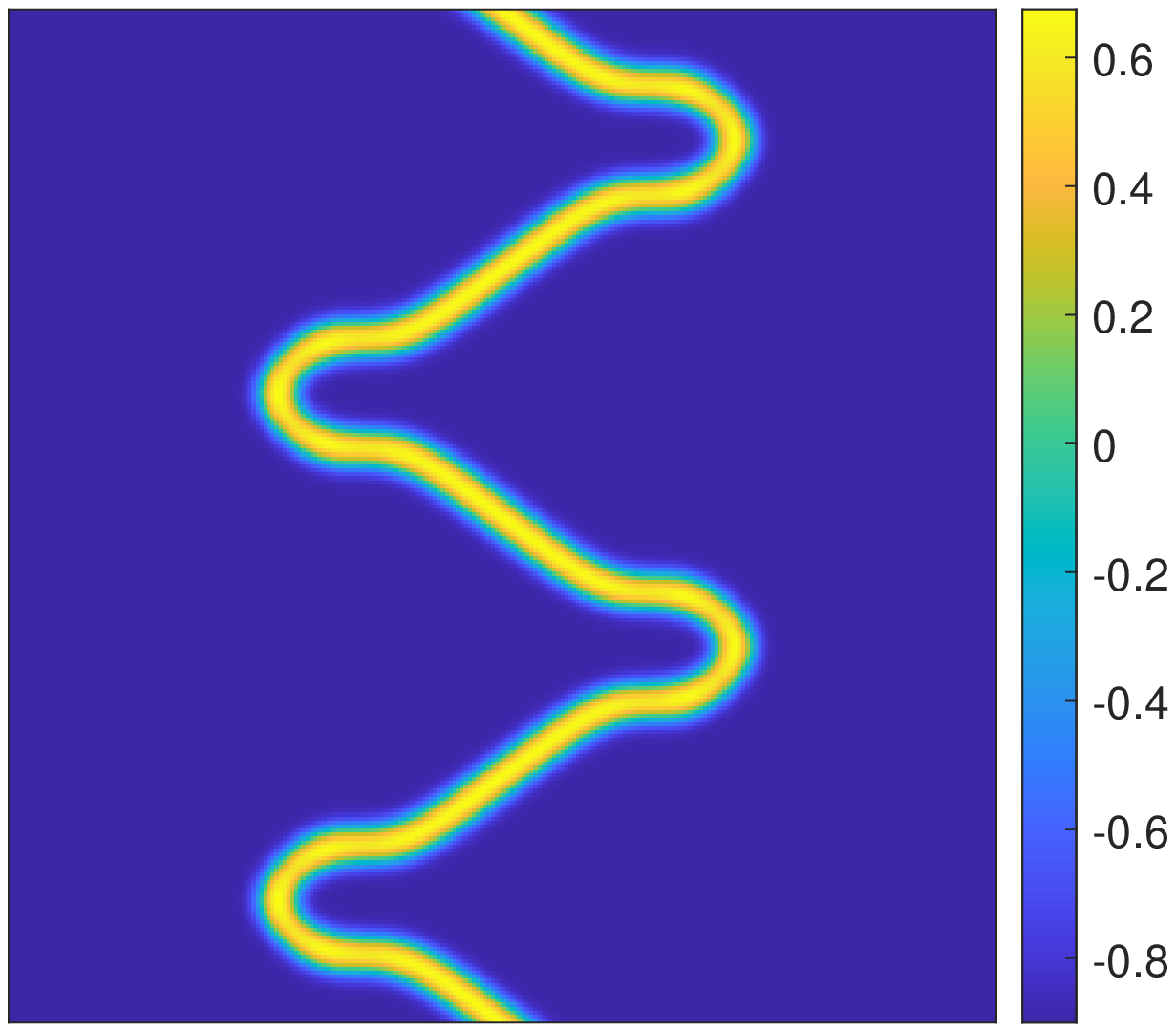}}
	\subfigure[$t=100$]{\includegraphics[scale=0.25,trim={30 0 30 0},clip]{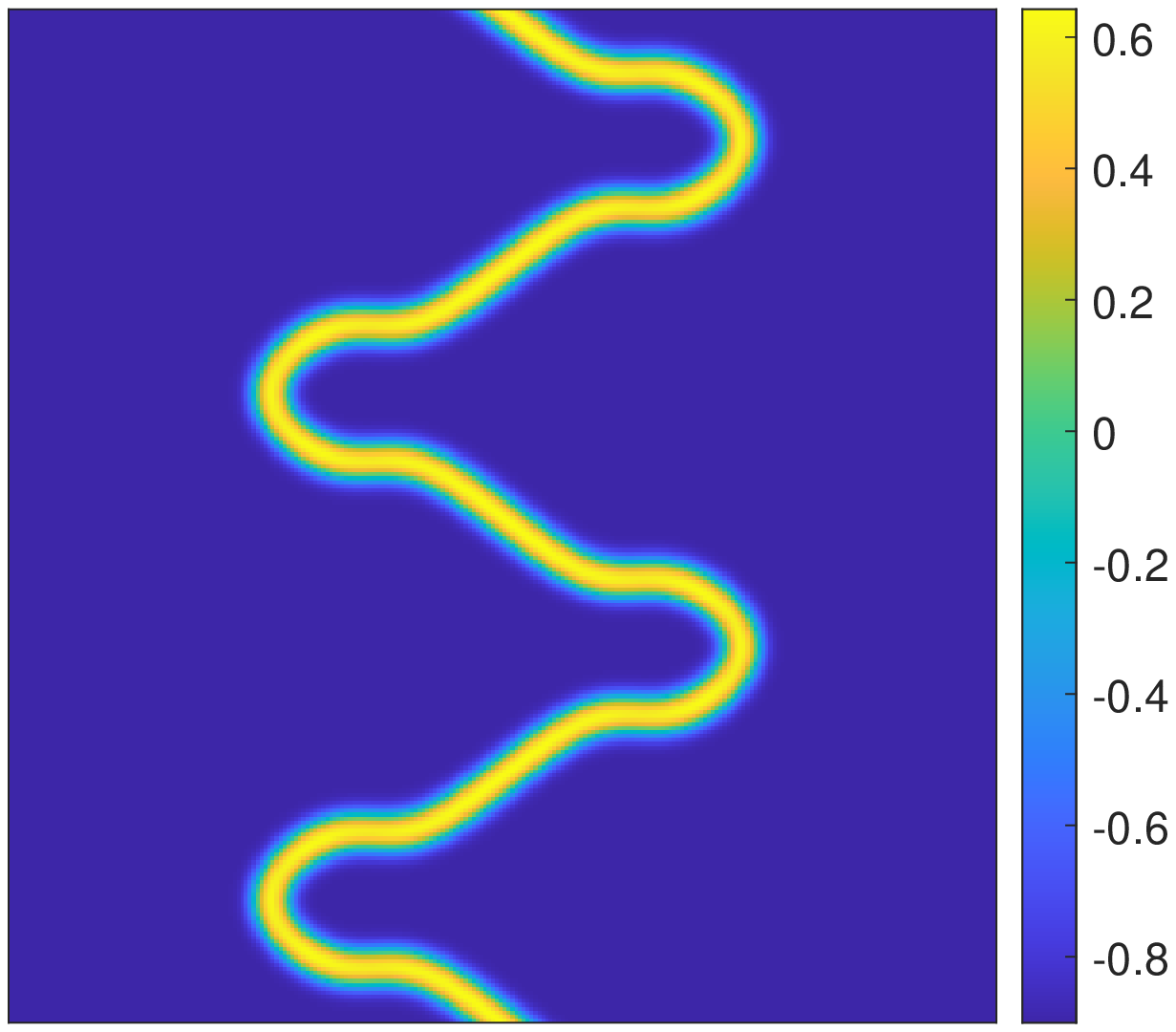}}
	
	\caption{Snapshots of the meandering instability with initial condition \eqref{meanderinitial}. Parameters are set to be as $\eta=10,\ \epsilon=0.01,\ p=1$, $h=1/256$ and $\lambda=\ln(19)/0.9$.}
	\label{meander}
\end{figure}

\begin{figure}[!t]
	\flushleft
	\subfigure		{\begin{overpic}[scale=.33, trim = {30 0 30 0}, clip]{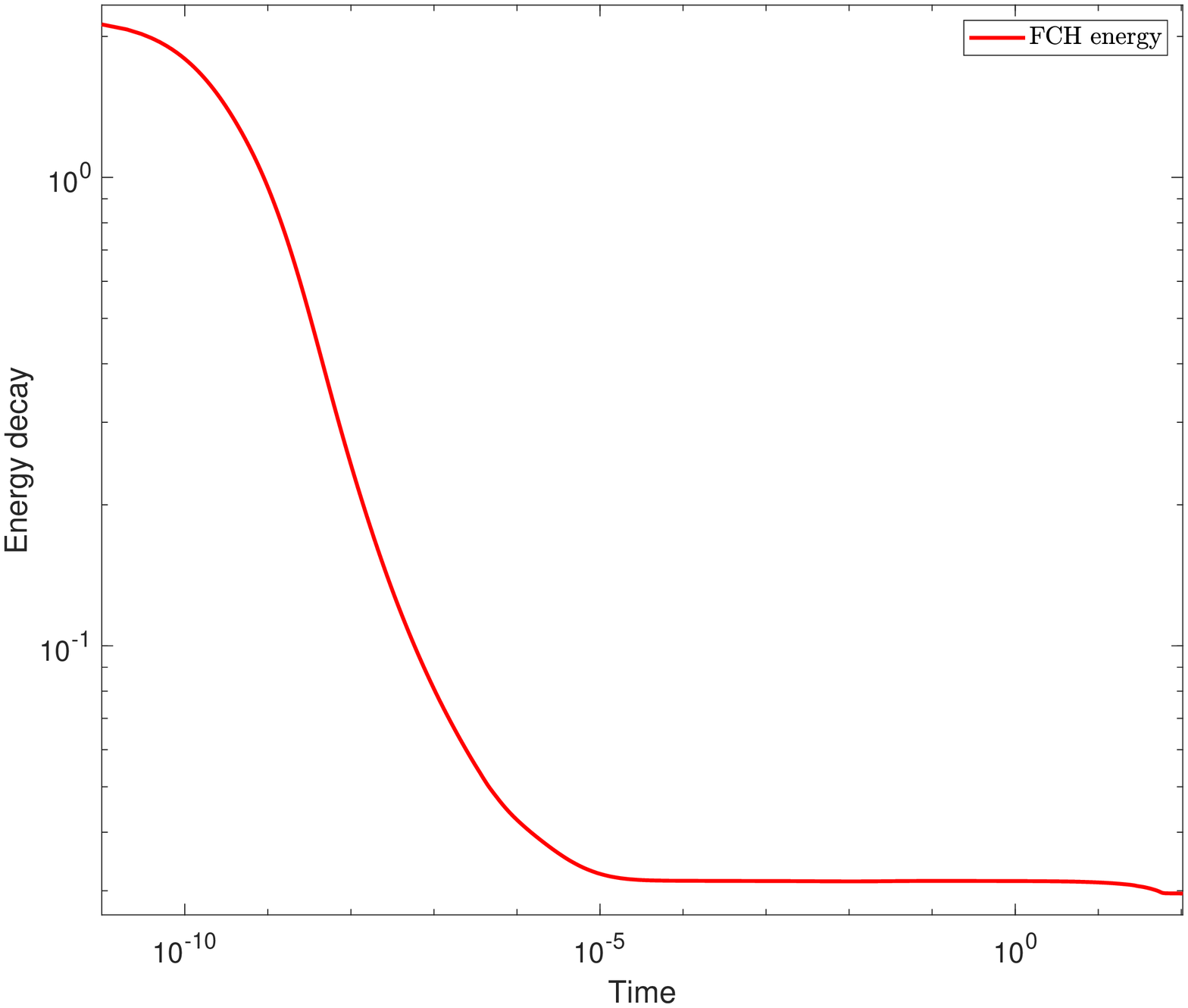}
			\put(37,26){\includegraphics[scale=.17, trim = {30 0 30 0}, clip]%
				{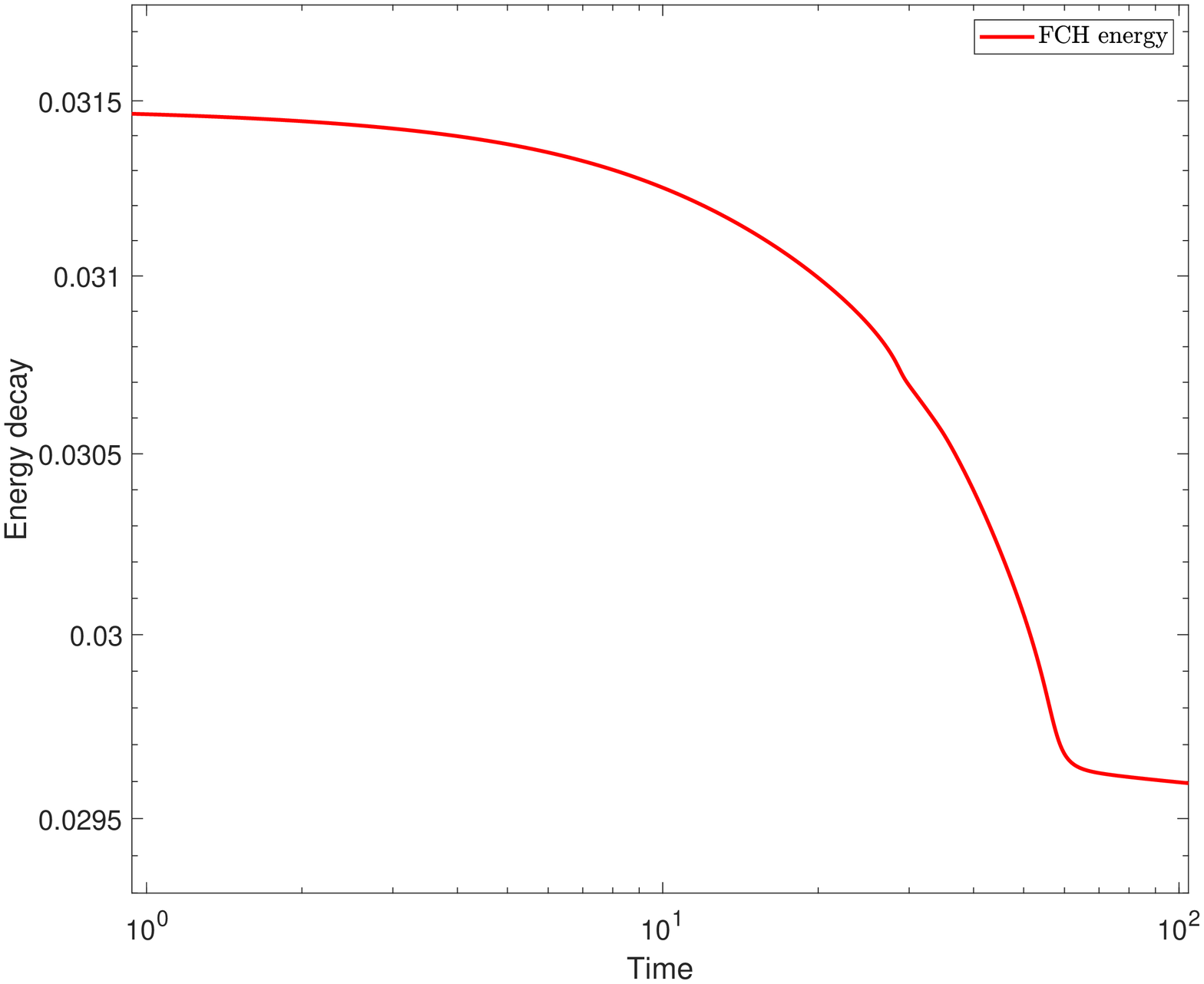}}
	\end{overpic}}
	\subfigure		{\begin{overpic}[scale=.33, trim = {30 0 30 0}, clip]{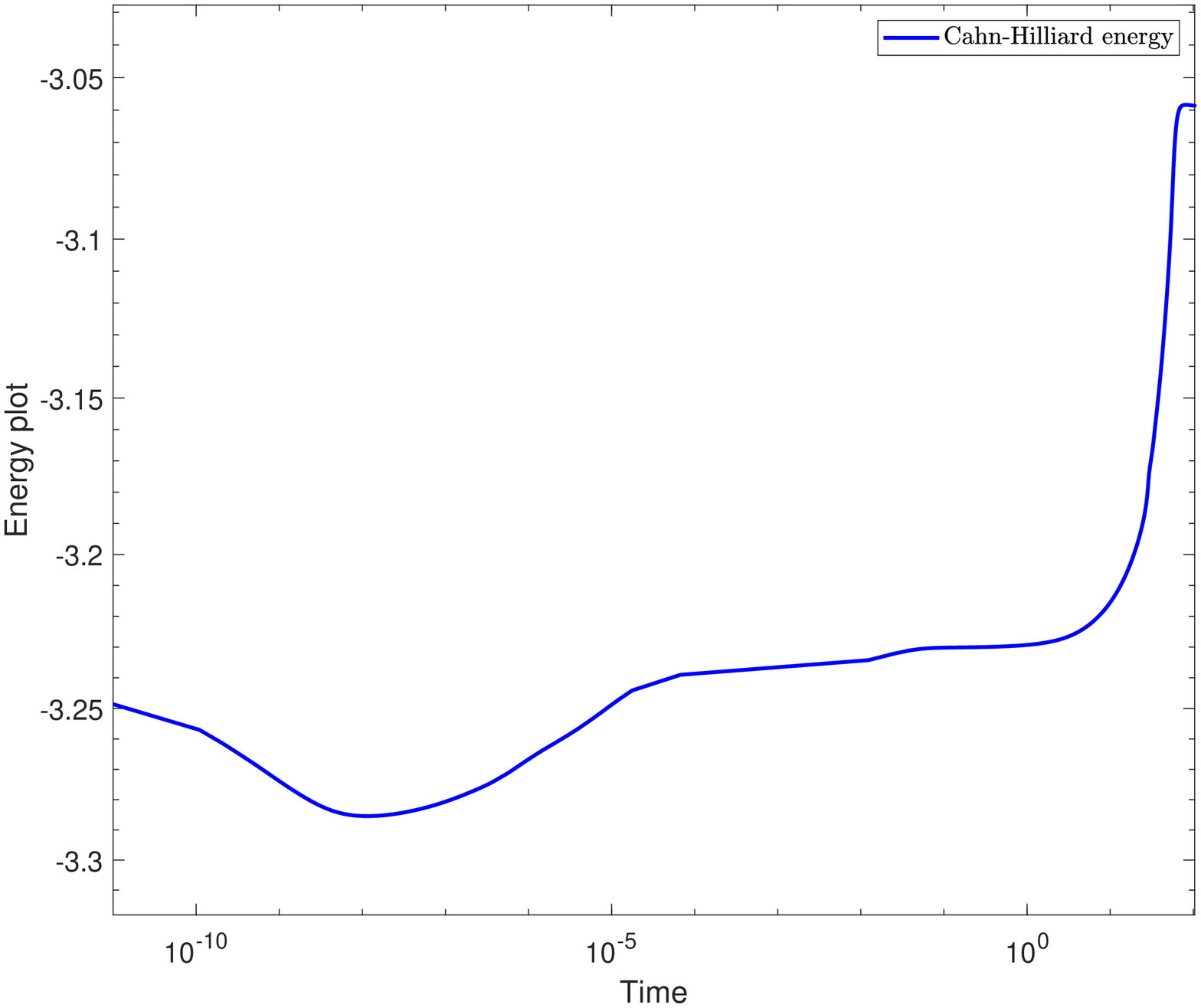}
			\put(16,31){\includegraphics[scale=.17,  trim = {30 0 30 0}, clip]%
				{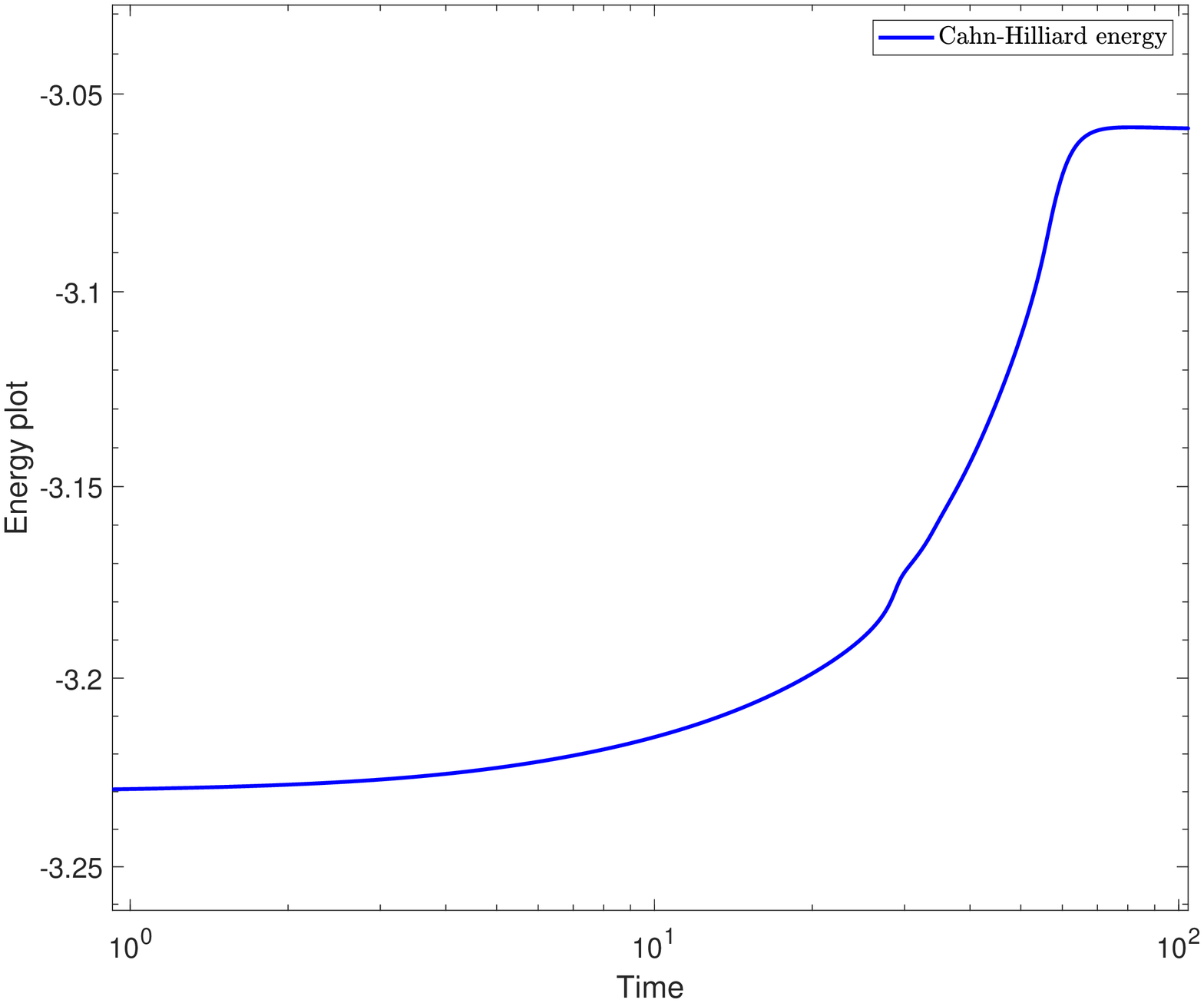}}
	\end{overpic}}
	\caption{Change of the FCH Energy (left) and the Cahn-Hilliard Energy (right) of the meandering instability.}
	\label{meanderenergy1}
\end{figure}

\subsection{Phase separation}

In this subsection we simulate the phase separation process, namely, the spinodal decomposition of a binary mixture with possible amphiphilic structure. We start with the following random initial condition (cf. \cite{feng2018}):
\begin{equation}\label{random initial data}
	\phi(x,y,0) = 0.5 + 0.01(2r-1),
\end{equation}
where $r$ are uniformly distributed random numbers in $[0, 1]$, and give the parameters as $\epsilon=0.008$, $\eta=8$, $p=1$, $h = 1/256$ and $\lambda= \ln(19)/0.9$.

Snapshots of the phase separation process can be found in Figure \ref{phasesepa}. The minimum value and maximum value of the phase variable stays between $-0.8852$ and $0.9035$ along time evolution, which keep a safe distance to the pure states $-1$ and $1$. This is consistent with the strict separation property proved in this paper (for the numerical scheme) and in \cite{SCHIMPERNA2020} (for the partial differential equation). Meanwhile, the average mass has a $10^{-12}$ error after $10^5$ iteration, which comes from possible cancellation errors when rounding off values at machine precision, as is shown in Figure \ref{randomresult}. The energy plot show a similar property to that for the previous simulation on meandering instability, in which the Cahn-Hilliard energy also increases after the initial decay. Figure~\ref{special structure} shows some special structures that we capture in the simulation, and these can be also found in the experiment of \cite{jain2003origins}.

\begin{figure}[!t]
	\centering
	\subfigure[$t=0$]{\includegraphics[scale=0.25,trim={20 0 20 0},clip]{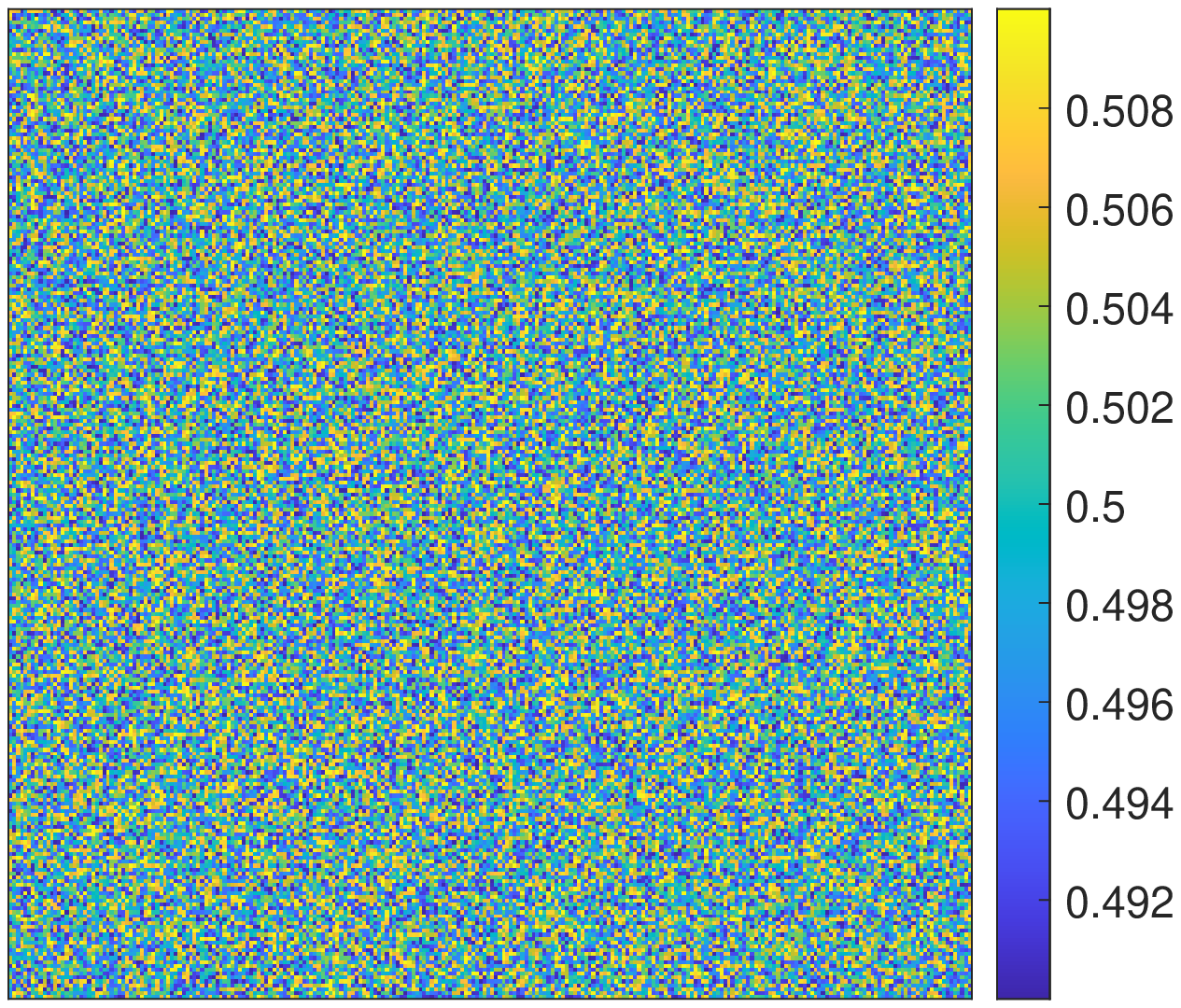}}
	\subfigure[$t=0.01$]{\includegraphics[scale=0.25,trim={20 0 20 0},clip]{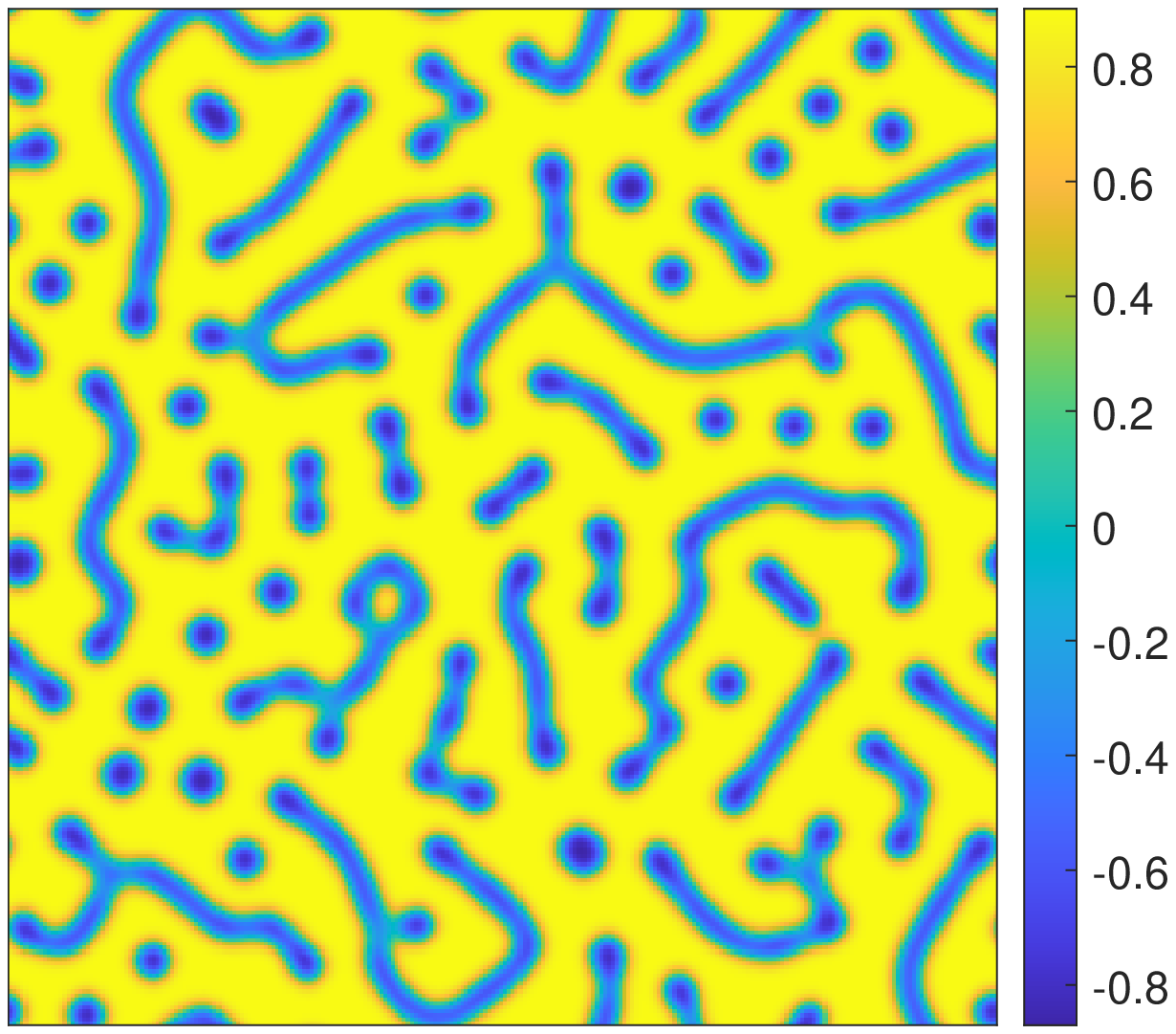}}
	\subfigure[$t=0.1$]{\includegraphics[scale=0.25,trim={20 0 20 0},clip]{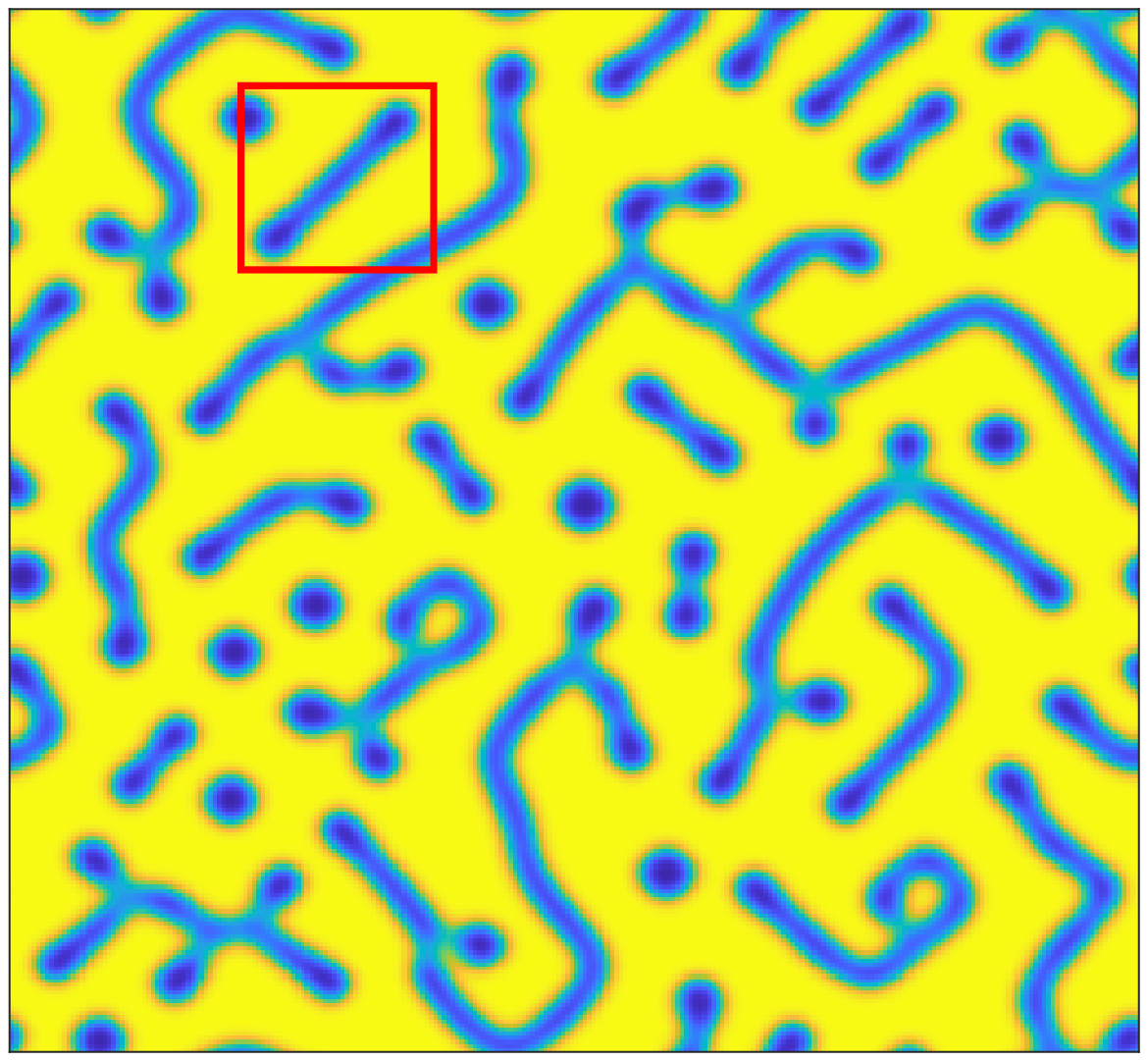}}
	\subfigure[$t=1$]{\includegraphics[scale=0.25,trim={20 0 20 0},clip]{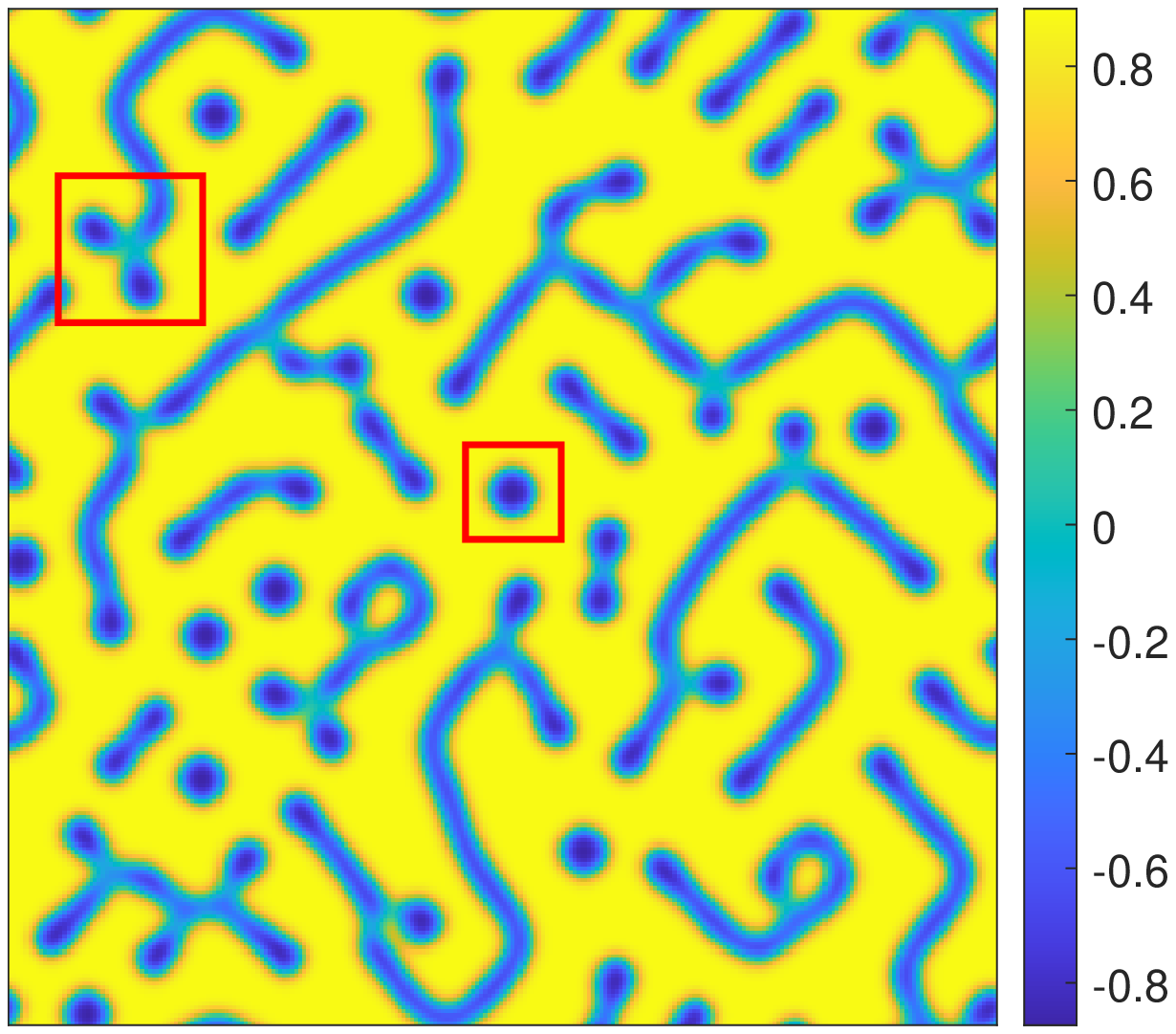}}
	\subfigure[$t=10$]{\includegraphics[scale=0.25,trim={20 0 20 0},clip]{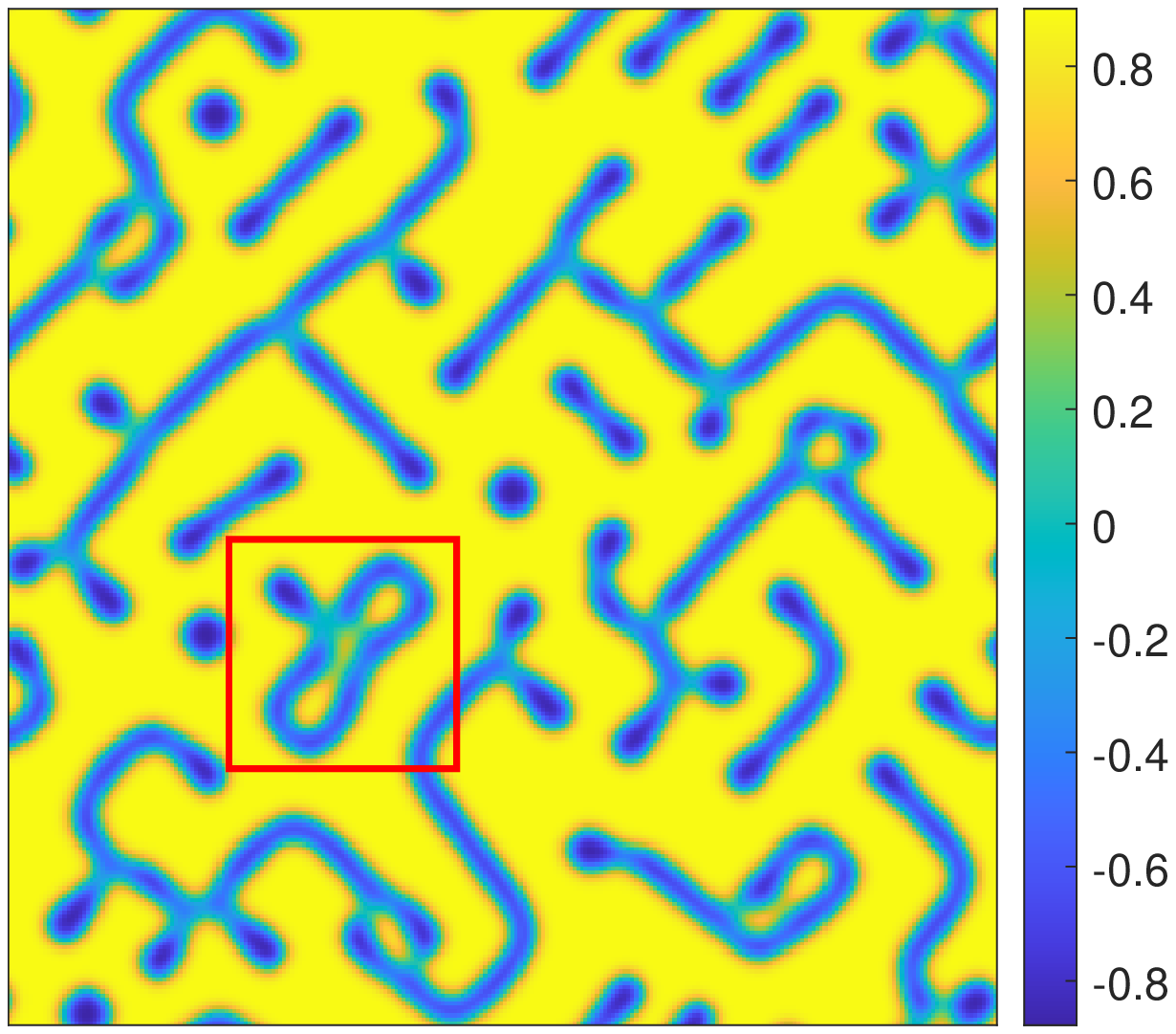}}
	\subfigure[$t=50$]{\includegraphics[scale=0.25,trim={20 0 20 0},clip]{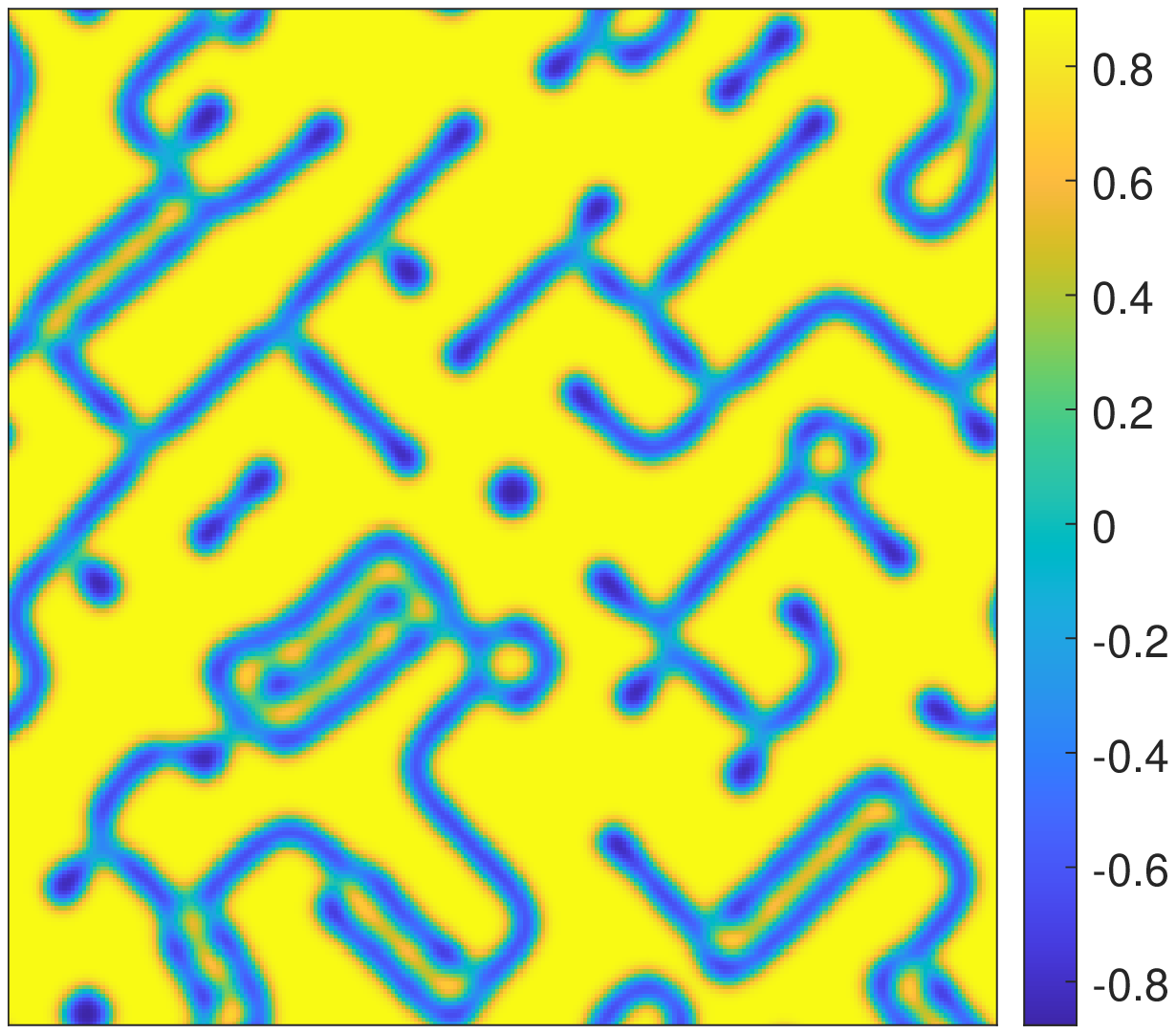}}
	\subfigure[$t=100$]{\includegraphics[scale=0.25,trim={20 0 20 0},,clip]{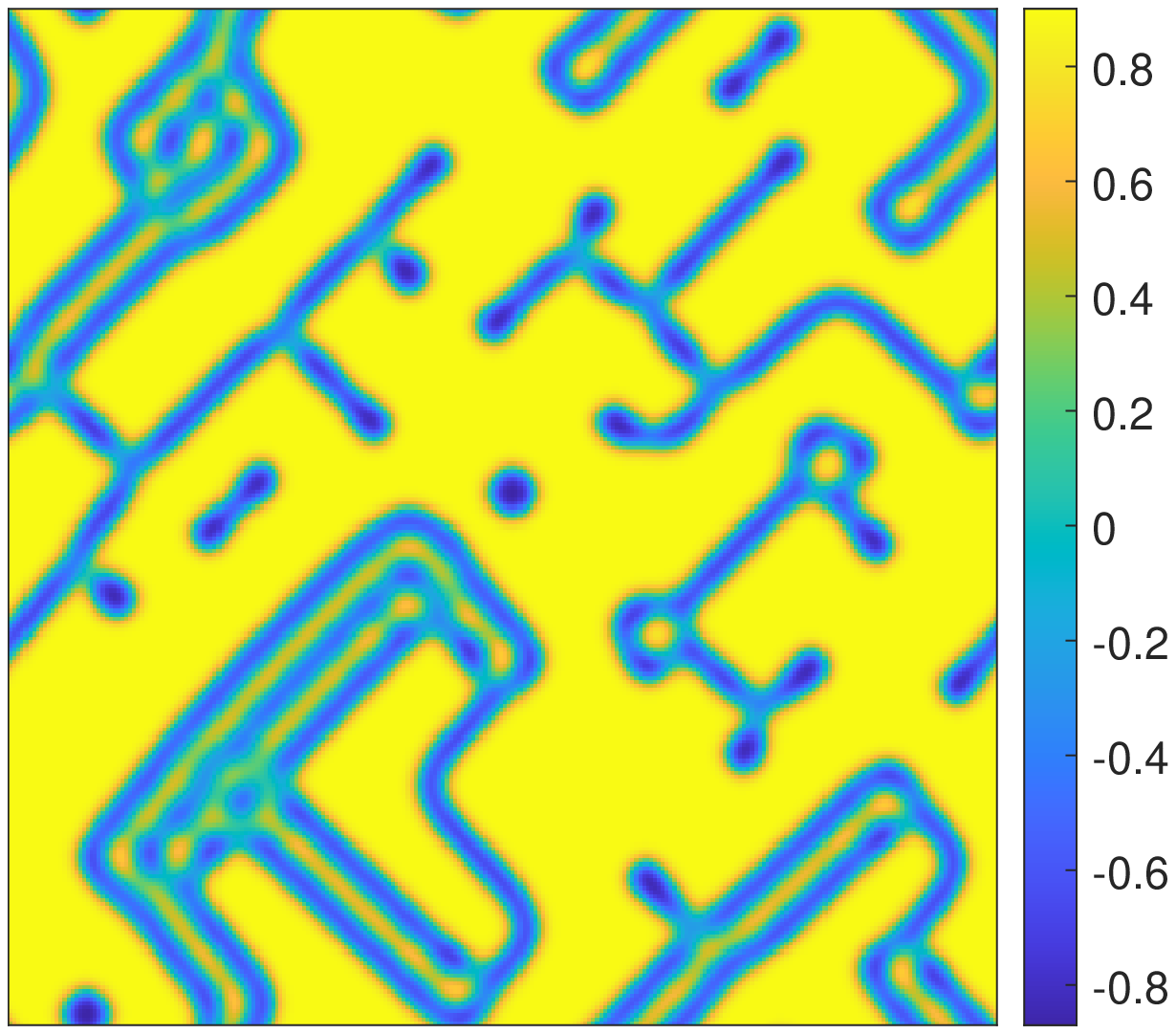}}
	\subfigure[$t=500$]{\includegraphics[scale=0.25,trim={20 0 20 0},clip]{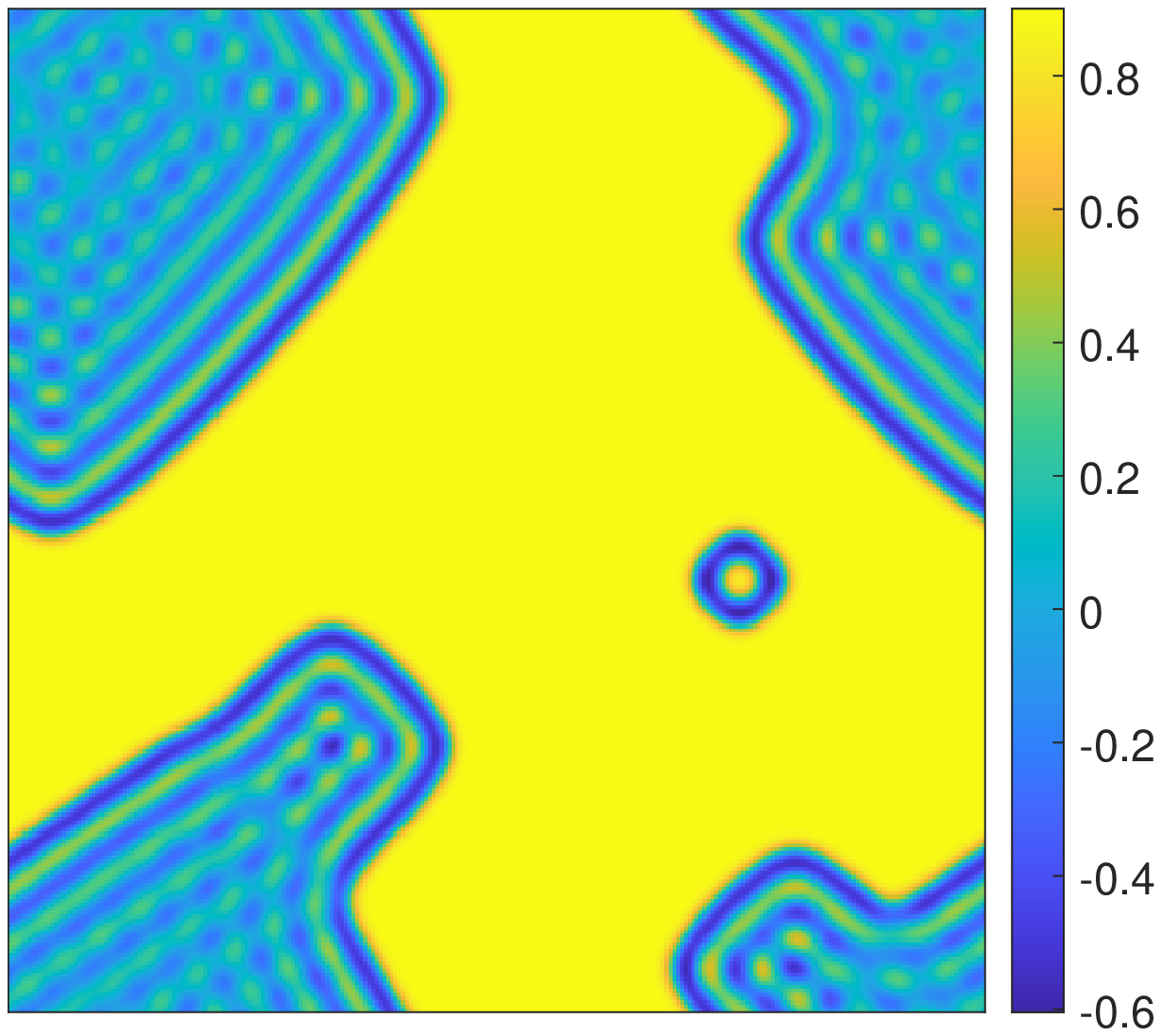}}
	\caption{Time snapshots of phase separation with the random initial datum \eqref{random initial data}.}
	\label{phasesepa}
\end{figure}
\begin{figure}[!t]
	\centering
	\subfigure[Y-juntion]{\includegraphics[scale=0.25,trim={20 0 20 0},clip]{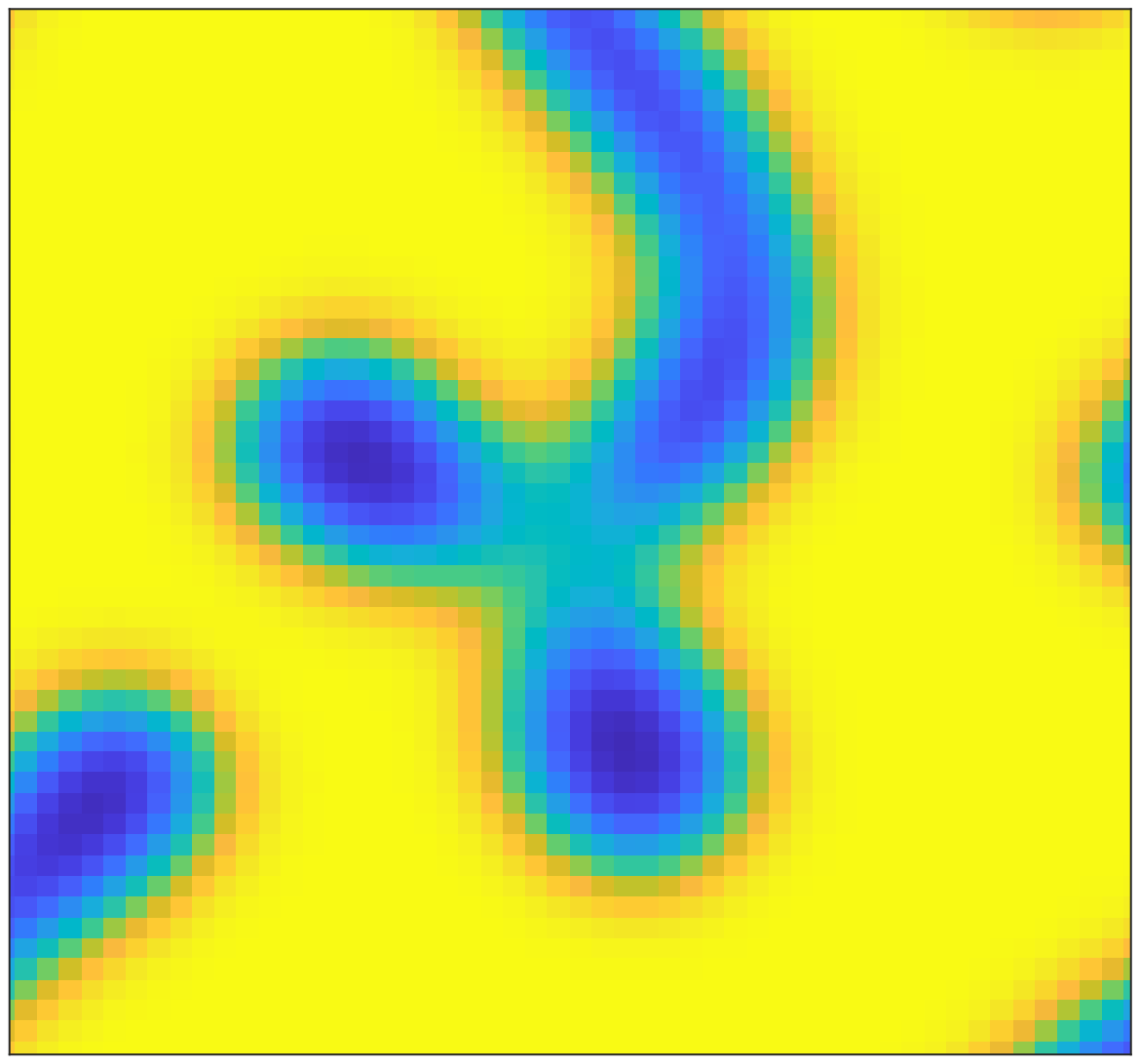}}
	\subfigure[cylinder]{\includegraphics[scale=0.25,trim={20 0 20 0},clip]{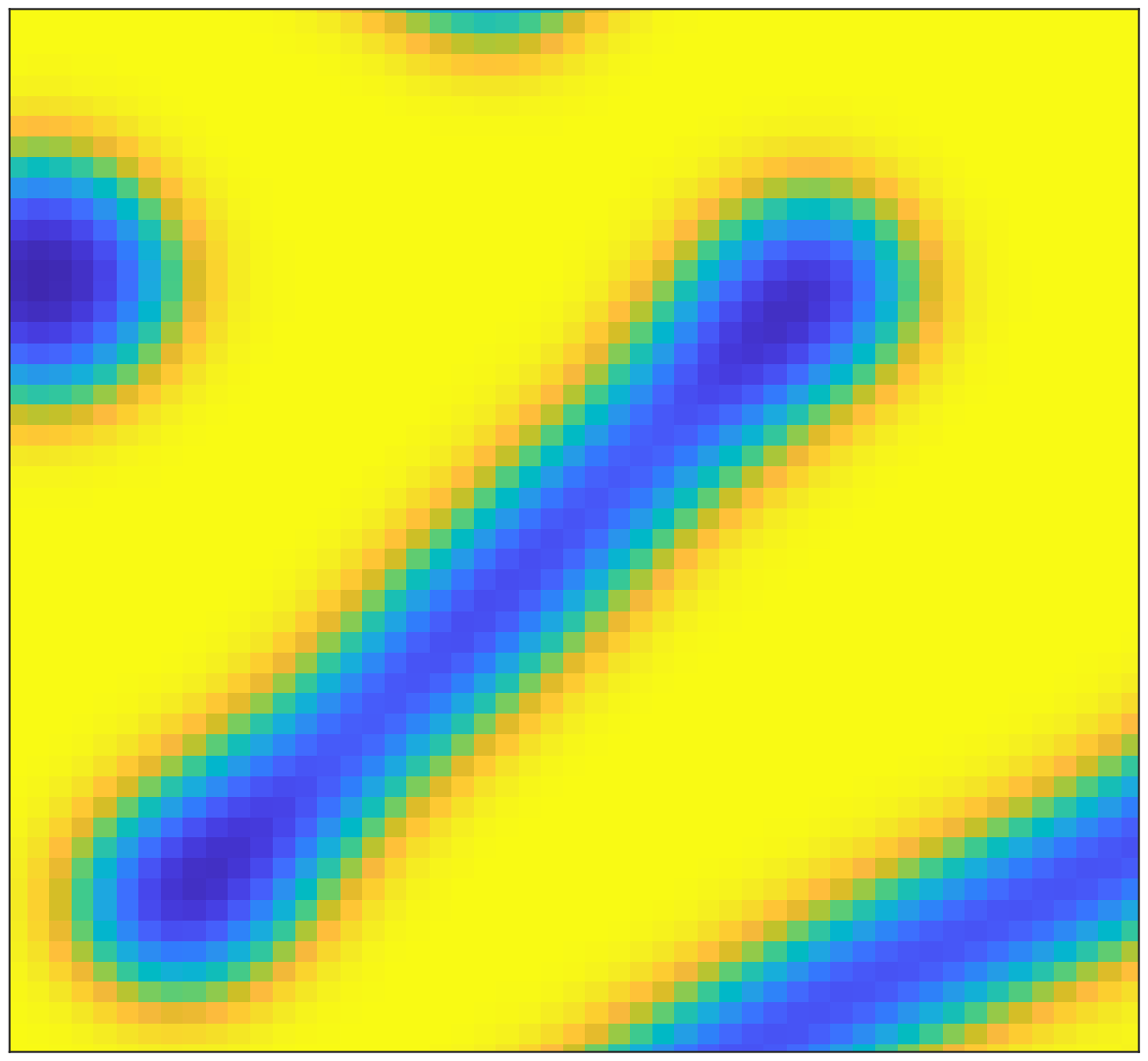}}
	\subfigure[sphere]{\includegraphics[scale=0.25,trim={20 0 20 0},clip]{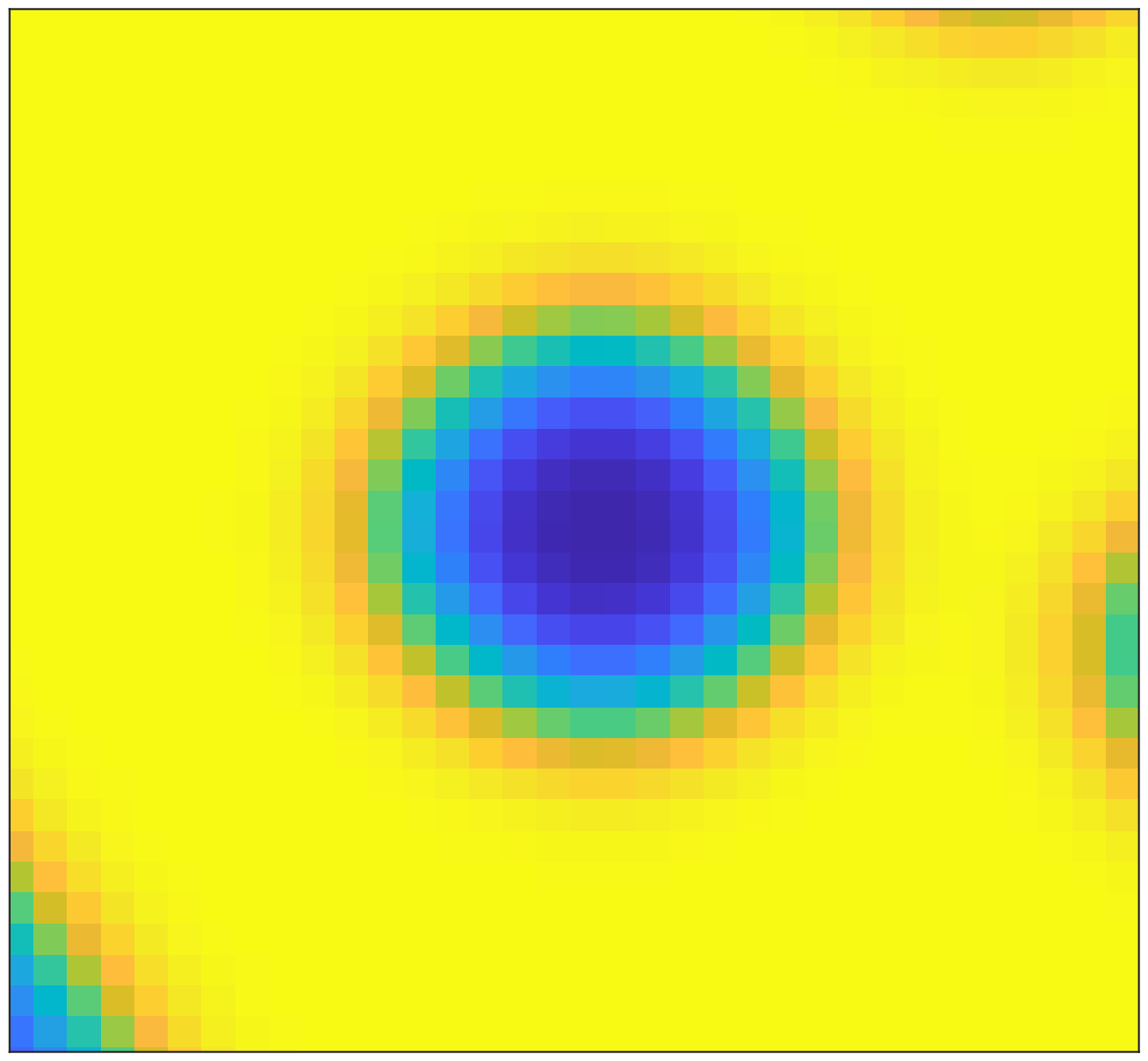}}
	\subfigure[cylindrical loop]{\includegraphics[scale=0.25,trim={20 0 20 0},clip]{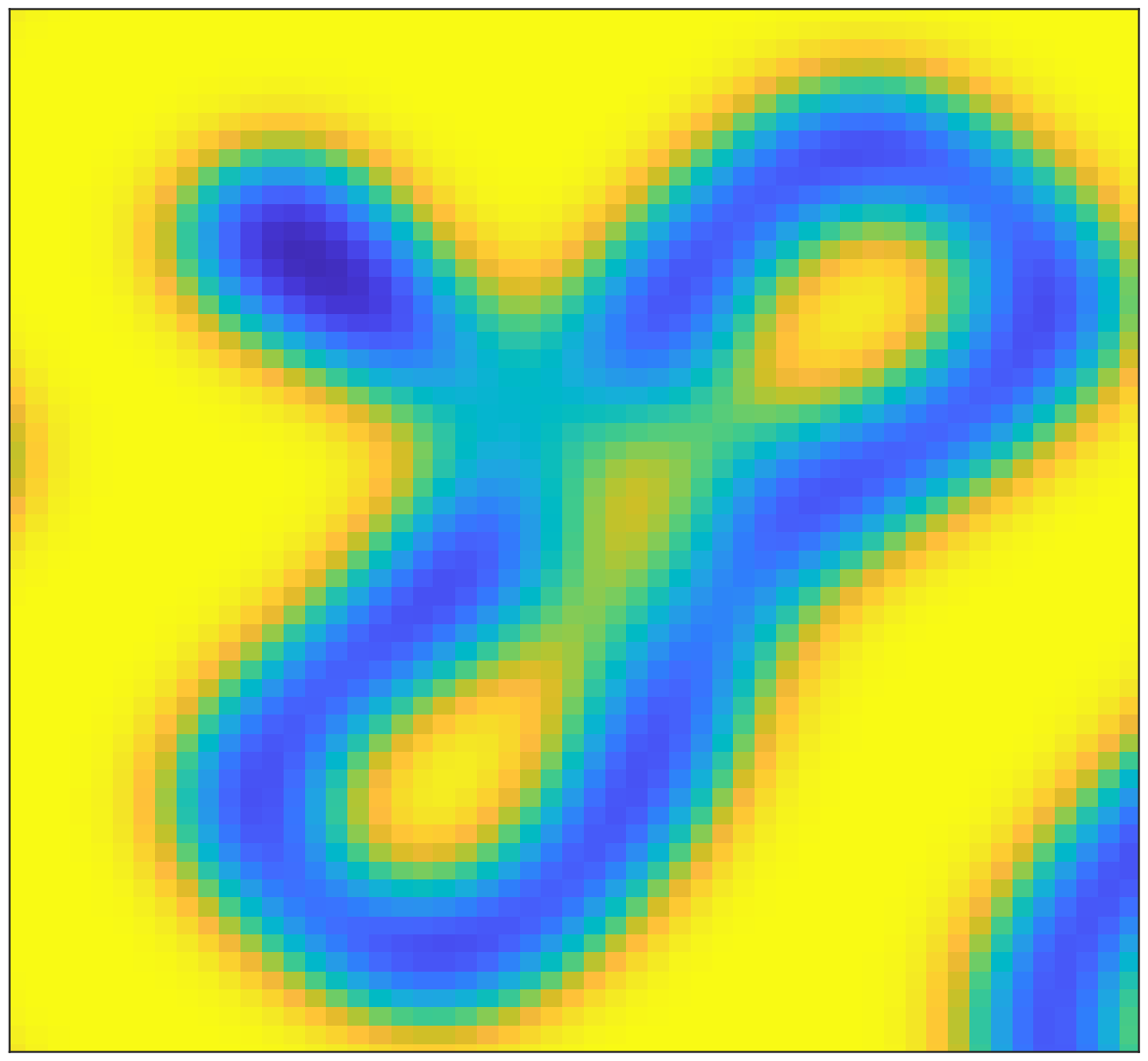}}
	\caption{Special structures like Y-junction, cylinder, spheres and cylindrical loops shown in the phase separation, which is consistent with Figure 1 and Figure 3 in \cite{jain2003origins}.}
	\label{special structure}
\end{figure}
\begin{figure}[!t]
	\centering
	\subfigure{\includegraphics[scale=0.3,trim={0 0 0 0},clip]{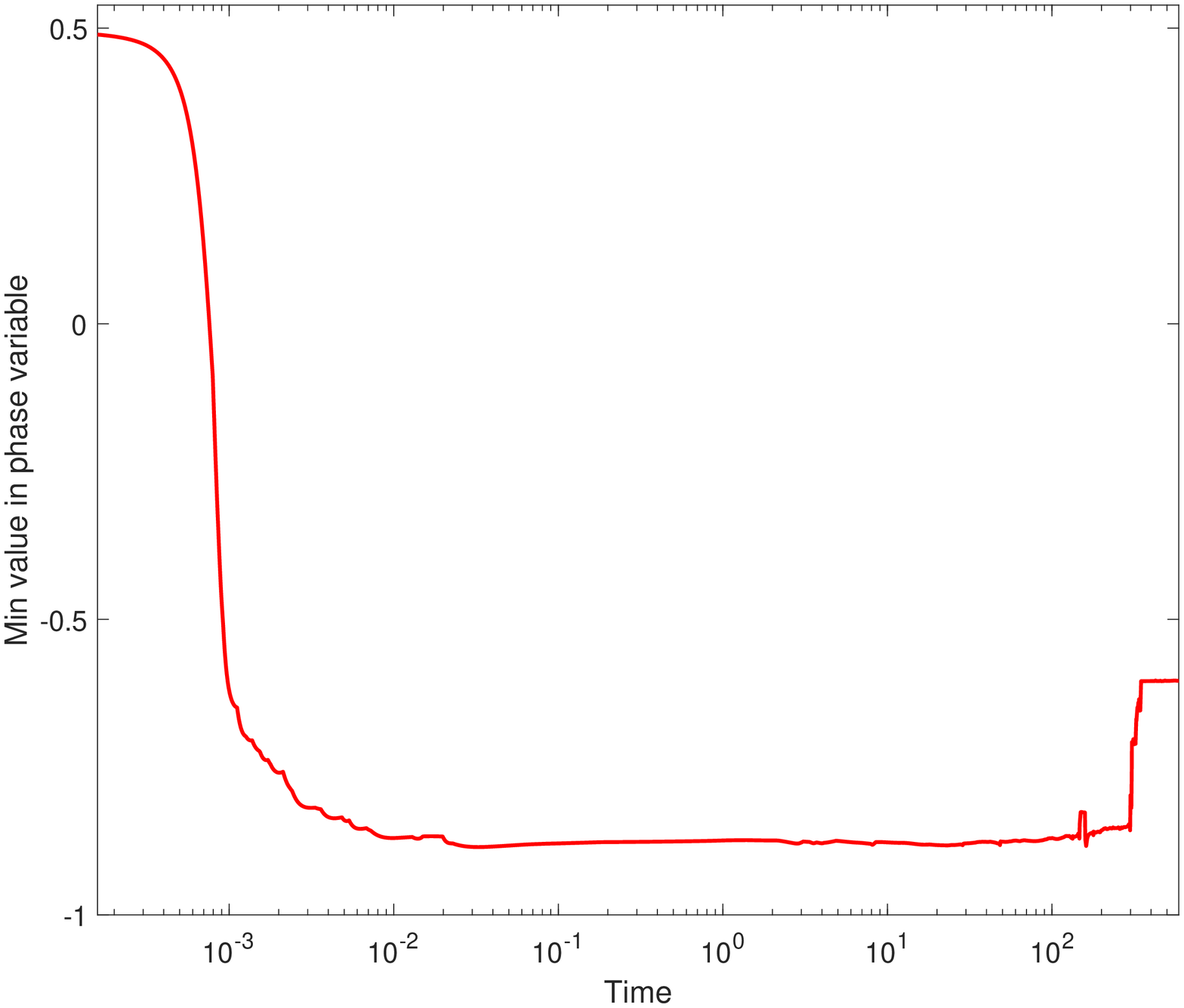}}\label{min}
	\subfigure{\includegraphics[scale=0.3,trim={0 0 0 0},clip]{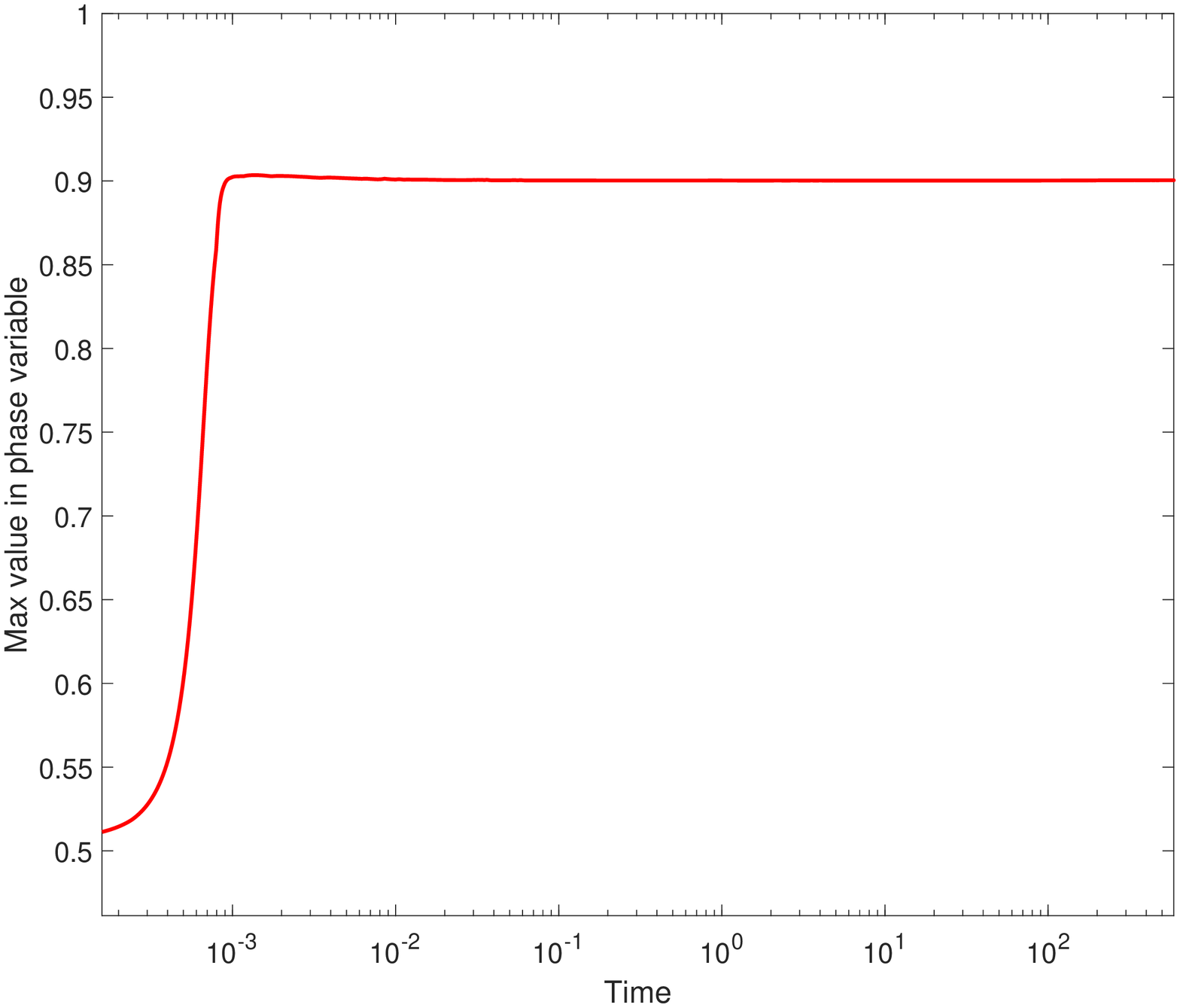}}\label{max}
	\subfigure{\includegraphics[scale=0.29,trim={0 0 0 0},clip]{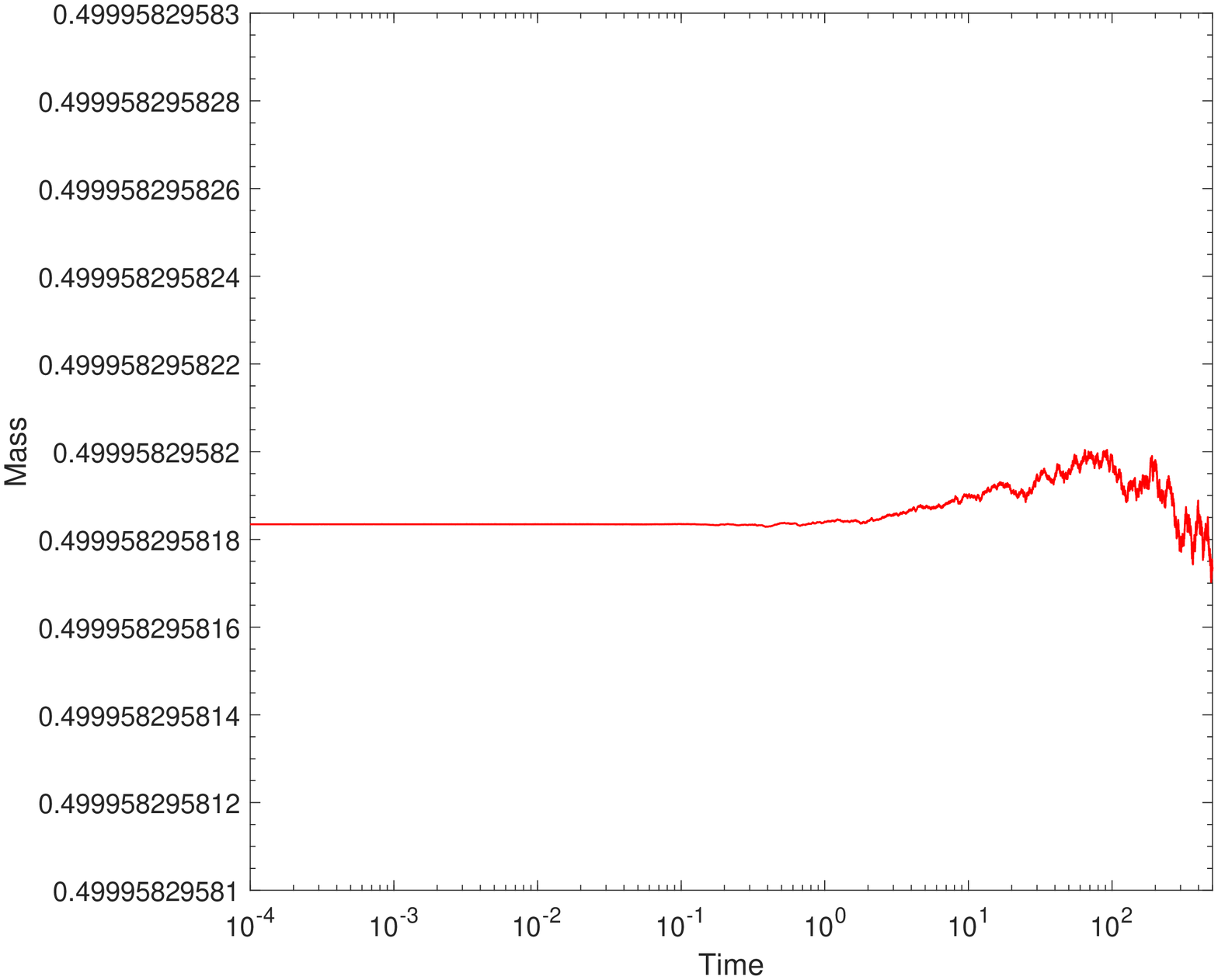}}\label{mass}
	\subfigure{\includegraphics[scale=0.3,trim={0 0 0 0},clip]{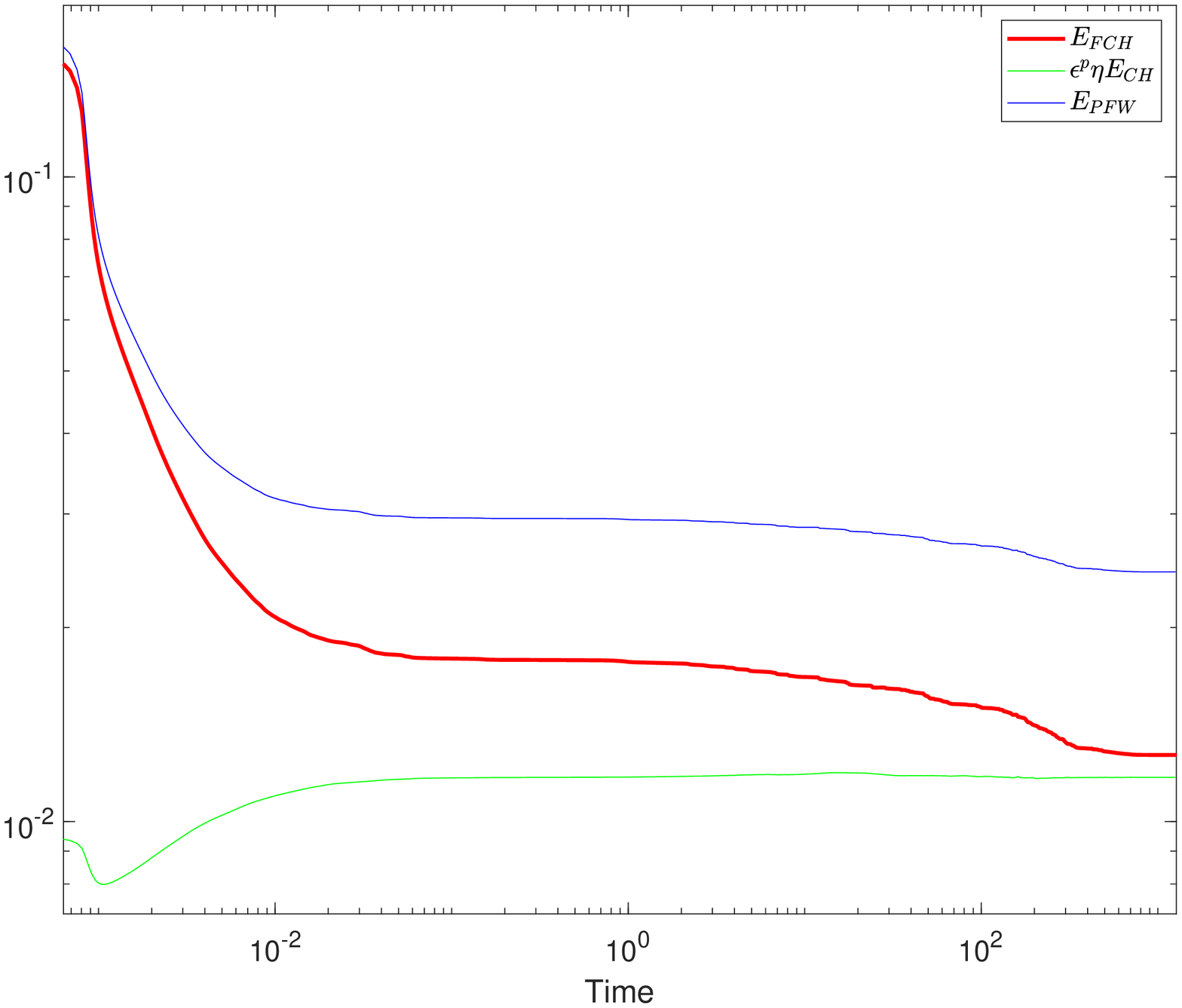}}\label{energy}
	\caption{Minimum value (top left), maximum value (top right), averaged mass (bottom left) and free energy (bottom right) plot in the long time simulation. These four pictures respectively show that our discrete numerical scheme inherits the properties of strict separation from pure states, mass conservation and energy dissipation that have been proved for the continuous problem.}
	\label{randomresult}
\end{figure}


\section{Concluding Remarks}
\label{sec:remarks}	
\setcounter{equation}{0}

In this paper, we present a first order semi-implicit finite difference scheme for the functionalized Cahn-Hilliard equation with a logarithmic Flory-Huggins type potential. The numerical scheme is derived based on the convex splitting technique. The non-convex and non-concave term $-\int_{\Omega}\beta(\phi)\Delta\phi\, \mathrm{d}x$ in the FCH energy is handled by using integration by parts under periodic boundary conditions. \emph{Unique solvability}, \emph{unconditional energy stability} and \emph{mass conservation} property{\tiny {\tiny }} of the numerical scheme are theoretically justified. Moreover, we establish the \emph{positivity-preserving property}, i.e., the phase function $\phi$ stays in $(-1,1)$ at a point-wise sense, with the aid of the singular nature of the logarithmic term $\beta(\phi)$. To the best of our knowledge, our numerical scheme is the first one that combines these four important theoretic properties for the FCH equation with a singular potential. Next, we perform an optimal rate convergence analysis and obtain error estimates in the higher order space $l^{\infty}(0,T;L_h^2)\cap l^2(0,T;H_h^3)$ under a linear refinement requirement $\Delta t\leq C_1 h$. To achieve the goal, we choose to perform a higher order asymptotic expansion up to third order accuracy in time and space, with a careful linearization technique. We first derive an error estimate in the  lower order space $l^{\infty}(0,T;H_h^{-1})\cap l^2(0,T;H^2_h)$ and obtain a rough error estimate in $L_h^2$. Then the $L_h^2$ error estimate combined with the inverse inequality yields a uniform $L_h^{\infty}$ bound on the error function, which implies the strict separation from pure states $\pm 1$ for the numerical solutions. This crucial fact enables us to achieve the refined $L_h^2$ estimate by using the energy method. In the last part of this paper, we present some numerical experiments in the two dimensional case with the PSD solver, which demonstrate the accuracy and robustness of the discrete scheme. Pearling bifurcation, meandering instability and spinodal decomposition are observed via the numerical simulation. The energy stability, mass conservation and the positivity preserving property are also verified numerically.

\section*{Acknowledgements}
W. Chen was partially supported by NSFC 12241101 and NSFC 12071090. H. Wu was partially supported by NSFC 12071084 and the Shanghai Center for Mathematical Sciences at Fudan University. W. Chen and H. Wu are members of the Key Laboratory of Mathematics for Nonlinear Sciences (Fudan University), Ministry of Education of China.




  \bibliographystyle{elsarticle-num}
  \bibliography{refs}


%
%
%
\end{document}